\documentclass[reqno]{amsart}
\usepackage{amsmath,amsfonts,amscd,amssymb,amsthm}
\usepackage{soul}
\usepackage{anyfontsize}
\usepackage{subcaption}
\numberwithin{equation}{section}
\usepackage{xfrac}
\usepackage{tikz}

\usepackage[mathscr]{euscript}

\newtheorem{theorem}{Theorem}[section]

\newtheorem{corollary}[theorem]{Corollary}
\newtheorem{lemma}[theorem]{Lemma}
\newtheorem{proposition}[theorem]{Proposition}

\newtheorem{remark}[theorem]{Remark}

\newcounter{as}[section]

\newtheorem{asser}[as]{Assertion}

\numberwithin{equation}{section}

\usepackage{color}

\newcommand{\mc}[1]{{\mathcal #1}}
\newcommand{\mf}[1]{{\mathfrak #1}}

\newcommand{\bb}[1]{{\mathbb #1}}
\newcommand{\bs}[1]{{\boldsymbol #1}}
\newcommand{\ms}[1]{{\mathscr #1}}

\newcommand{\<}{\langle}
\renewcommand{\>}{\rangle}
\renewcommand{\>}{\rangle}

\definecolor{bblue}{rgb}{.2,0.2,.8}

\usepackage[pdfpagelabels=true,unicode=true]{hyperref}
\hypersetup
{
	pdfauthor={},
	pdftitle={},
	pdfkeywords={},
	pdfsubject={},
	colorlinks=true,
	linkcolor=blue,
	citecolor=blue,
	filecolor=blue,
	urlcolor=blue
}

\title{A one-dimensional non-local singular SPDE}

\author {L. Chiarini, C. Landim}

\address{\noindent \noindent IMPA, Estrada Dona Castorina 110, CEP 22460 Rio de
  Janeiro, Brasil and Utrecht University, Heidelberglaan 8, 3584 CS Utrecht,
  Netherlands.  \newline e-mail: \rm
  \texttt{chiarini@impa.br} }

\address{\noindent IMPA, Estrada Dona Castorina 110, CEP 22460 Rio de
  Janeiro, Brasil and CNRS UMR 6085, Universit\'e de Rouen, Avenue de
  l'Universit\'e, BP.12, Technop\^ole du Madril\-let, F76801
  Saint-\'Etienne-du-Rouvray, France.  \newline e-mail: \rm
  \texttt{landim@impa.br} }

\begin{document}
%

\begin{abstract}
We examine in this article the one-dimensional, non-local, singular SPDE
\begin{equation*}
\partial_t \mf u \;=\;
-\, (-\Delta)^{1/2} \mf u \,-\,  \sinh(\gamma \mf u) \,+\,
\xi\; ,
\end{equation*}
where $\gamma\in \bb R$, $(-\Delta)^{1/2}$ is the fractional Laplacian
of order $1/2$, $\xi$ the space-time white noise in
$\mathbb{R} \times \mathbb{T}$, and $\bb T$ the one-dimensional
torus. We show that for $0<\gamma^2<\pi/7$ the Da Prato--Debussche
method \cite{DaPrato2003} applies. One of the main difficulties lies
in the derivation of a Schauder estimate for the semi-group associated
to the fractional Laplacian due to the lack of smoothness resulting
from the long range interaction.
\end{abstract}



\maketitle

\hfill \emph{Dedicated to Stefano Olla on his sixtieth birthday}

\section{Introduction}
\label{sec00}

In this article, we study the the local well-posedness of the following
SPDE
\begin{equation}
\label{int01}
\partial_t \mf u = -(-\Delta)^{1/2}\mf u - \sinh(\gamma \mf u) + \xi
\quad \text{ in } \bb{R}_+ \times \bb{T}
\end{equation}
in which $\bb T$ is the one-dimensional torus, $\xi$ is the space-time
white-noise, $(-\Delta)^{1/2}$ is the half-Laplacian operator,
$\sinh(\gamma \mf u)$ is taken in a Gaussian Multiplicative Chaos (GMC)
sense and $\gamma \in \bb R$ belongs to a small interval around the
origin.  This equation is a long-range counter part to the equation
\begin{equation*}
\partial_t \mf u = \Delta \mf u - \sinh(\gamma \mf u) + \xi
\quad \text{ in } \bb{R}_+ \times \bb{T}^2
\end{equation*}
which appears in the context of Liouville quantum gravity \cite{Garban18}
and is related to the $\cosh$-interaction, and in Quantum Field Theory
\cite{HaiShe16} when $\sinh$ is replaced by $\sin$. We will refer to
it as $\sinh$-Gordon equation.

The lack of regularity of the white-noise $\xi$ prevents the existence
of function-valued solutions. In consequence, the meaning of the
non-linear terms of the equation is not clear. Several approaches to
circumvent this problem were proposed. The first one, the so-called da
Prato-Debussche perturbative method \cite{DaPrato2003}, provides local
existence in time of a certain class of equations.

More recently, Gubinelli and co-authors \cite{Gub04, Gub15}
introduced an approach to study singular SPDEs based on techniques
from paradifferential calculus and controlled rough paths, and
Hairer \cite{Hai14} proposed a theory for studying
a large class, so-called \textit{subcritical}, of non-linear
SPDE's by using \textit{regularity structures}.

The equation \eqref{int01} studied in this paper falls short of the
scope of the theory of regularity structures for two reasons. The
first one, discussed here in the regime $\gamma$ small, is the
non-locality of the operator $(-\Delta)^{1/2}$ responsible for the
lack of smoothness of the semigroup. In consequence, to apply the
theory of regularity structures, one needs new methods to prove the
well-posedness of the operator (denoted by $\mc K_\gamma$ in
\cite{Hai14}) which works as the abstract counterpart of the
integration against the kernel for regularity structures.  This
question has been adressed in \cite{BergKue17} in the context of
polynomial non-linearity.  It is also not clear if the lack of
regularity of the semigroup is not an obstacle to estimate the
remainder of the Taylor expansions appearing in the BPHZ
renormalisation presented in \cite{ChaHai16}.

The second problem is related to the lack of regularity of the
exponential non-linearity. In \cite{ChaHai16}, the authors require the
noise (or the non-linearity) to have finite cumulants of all
orders. However, as the classical theory indicates, the GMC does not
have finite moments of order larger than $4\pi/\gamma^2$. Notice that
in \cite{Hai19}, by changing the perspective from renormalisation of
regularity structures to renormalisation of graphs, the author does
not need to refer to cumulants with regards to \textit{negative
  renormalisation}. However, one would require \textit{positive
  renormalisation} in order to construct a local solution of
\eqref{int01} when leaving the da Prato-Debussche regime.

Consider, for example, the second of these problems in relation to the
sine-Gordon equation \cite{HaiShe16, CHS18}
\begin{equation}
\label{int02}
\partial_t \mf u = \Delta \mf u + \sin(\gamma \mf u) + \xi
\quad \text{ in } \bb{R}_+ \times \bb{T}^2\;,
\end{equation}
Although this equation also comes from a GMC type of process (by
seeing it as the real part of $\exp(i \beta \mf u)$), the boundedness of
the $\sin$ function implies the finiteness of all moments. Moreover,
the moments only grow in a linear fashion. This makes the equations
\eqref{int01}, \eqref{int02} very different from the perpective of
regularity structures.

With regard to the second problem, an exponential non-linearity has
been examined before in the da Prato-Debussche regime by Garban
\cite{Garban18} for the equation
\begin{equation}
\label{int01b}
\partial_t \mf u = \Delta \mf u + \exp(\gamma \mf u) + \xi
\quad \text{ in } \bb{R}_+ \times \bb{T^2}
\end{equation}
in dimension $2$. There is no diference between proving local
well-posedness between the non-linearities $\sinh(\gamma \mf u)$ and
$\exp(\gamma \mf u)$ when restricted to small values of $\gamma$. However,
the symmetries of the latter make it more likely to have global
well-posedness and for it to appear as a scaling limit of some class
of particle systems.

One of the main contributions of this paper, presented in
Section~\ref{sec04}, is a Schauder estimates for the half-Laplacian on
the Torus in the context of negative Besov spaces. The equivalent
result has already been proven for H\"older spaces (or equivalently,
positive Besov spaces) in \cite{DonKim13}. The starting point of our
proof is the strategy proposed in \cite{Hai14} of writing the
semigroup as an infinite, taking advantage of its scalar
invariance. However, as the semigroup is not smooth, extra arguments
are required to define its convolution with arbitrary
distributions in $C^\alpha$ for $\alpha >-1$.

The same strategy presented here can be used to prove Schauder
estimates in Besov space for any $\alpha < 0$.  This result could be
extended to other fractional powers of the Laplacian, provided that
one is able to derive the necessary bounds for the right-side
derivatives of the semigroup at $t=0$. Finally, we believe that the
strategy presented here could also be used to prove well-posedness of
the operator $\mc K_\gamma$, mentioned above, in the context of
regularity structures, an important step in the use of regularity
structures for non-local equations.

A good motivation to consider non-local SPDE's in one dimension is the
study of scaling limits in statistical mechanics. Although, long-range
models, for some $d$-dimensional lattice, often share similarities
with their short-range counterparts, their finer details might behave
very differently. For instance, take the one-dimensional Ising Model
with interaction of type $J_{x,y} \propto \|x-y\|^{-1-\alpha}$ for
$\alpha \in (0,1]$, also known as the Dyson Model. Here, we will focus
on the critical case $\alpha=1$, closely related to the operator
$(-\Delta)^{1/2}$. For this case, at some critical temperature the
model displays phase transition \cite{frohlich1982phase}, just like its
planar (and nearest-neighbour) counterpart. On the other hand, this
transition is neither sharp \cite{imbrie1988intermediate} nor
continuous \cite{aizenman1988discontinuity}.  In
\cite{mourrat2017convergence} the authors prove that the scaling limit
of the magnetization field of the Glauber dynamics for a class of
Ising-Kac models in dimension 2 is described by the solutions of the
$\Phi^4_2$ equation.  However, the different behaviour at criticality
makes it unclear whether the result of \cite{mourrat2017convergence}
can be proven in dimension 1 for an interaction of type
$J_{x,y} \propto \|x-y\|^{-2}$. Furthermore, we could not find in the
literature even if the one-dimensional $\Phi^4_1$ model with
interaction $J_{x,y} \propto \|x-y\|^{-2}$ and its scaling limit share
the phase diagram of the Dyson model, which is the first indication of
whether a theorem like \cite{mourrat2017convergence} can take place or
not.

\section{Notation and results}
\label{sec01}

We state in this section the main results of the article.  We start
introducing some spaces of continuous functions. Denote by
$\bb T = [-1/2,1/2)$ the one-dimensional torus. Elements of
$\mathbb{R}\times \mathbb{T}$ are represented by the letter $z=(t,x)$.
Denote by $d$ the distance on $\mathbb{R}\times \mathbb{T}$ given by
\begin{equation*}
d(z,z') \;=\;
d \big( (t,x) \,,\,(t^\prime,x^\prime) \big) \;=\;
|t - t^\prime| \;+\; d_{\mathbb{T}}(x,x^\prime) \;,
\end{equation*}
where $d_{\mathbb{T}}(\cdot,\cdot)$ stands for the standard metric in
$\mathbb{T}$. We will often use the notation $\|z\|$ to denote
$d(0,z)$ even though it is not a norm.

For a continuous function $f\colon \bb R\times \bb T \to \bb R$, let
$\Vert\, f\,\Vert_{L^\infty(\bb R\times \bb T)}$ be the norm given by
\begin{equation*}
\Vert\,f\,\Vert_{L^\infty(\bb R\times \bb T)} \;=\;
\sup_{ z \in \mathbb{R}\times \bb  T} \big |\, f(z) \,\big|\;.
\end{equation*}
Denote by $\color{bblue} C^\beta = C^\beta (\bb R \times \bb T)$,
$\beta \in (0,1)$, the set of continuous functions $f: \bb R \times
\bb T \to \bb R$ such that for all $S$, $T\in \bb R$, $S<T$,
\begin{equation}
\label{13}
\|f\|_{C^\beta ([S,T] \times \bb T)}
\; := \; \sup_{ z, z' \in [S,T] \times \bb T } \;
\frac {|f(z)-f(z')|}{d(z,z')^\beta} \;<\; \infty \;.
\end{equation}
Let $\color{bblue} C^\beta_+ = C^\beta_+ (\bb R \times \bb T)$ be the
subset of functions $f$ in $C^\beta (\bb R \times \bb T)$ such that
$f(t,x) =0$ for all $t\le 0$. The space $C^\beta_+ ([0,T] \times \bb
T)$ is defined analogously. \smallskip

Let $\color{bblue} \bb N = \{1, 2, \dots \}$, $\color{bblue} \bb N_0 =
\bb N \cup \{0\}$, and denote by $\color{bblue} C^m(\bb R \times \bb
T)$, $m\in \bb N_0 \cup \{\infty\}$, the set of $m$-times continuously
differentiable functions. $\color{bblue} C^m_c(\bb R \times \bb T)$
represents the subset of functions in $ C^m(\bb R \times \bb T)$ which
have compact support.

For a function $f$ in $ C^m(\bb R\times \bb T)$, $m\in \bb N_0$, let
\begin{equation*}
\|f\|_{C^m(\bb R\times \bb T)} \;:=\; \sum_{ i+j \le m }
\ \Vert \, \partial^i_t \, \partial^j_x \, f \,
\Vert_{L^\infty(\bb R\times \bb T)} \;.
\end{equation*}
In this formula, the sum is carried over all integers $i$, $j$ in
$\bb N_0$ such that $i+j\le m$, and
$\partial^i_t \, \partial^j_x \, f$ stands for the $i$-th derivative
in time of the $j$-th derivative in space of $f$.

\smallskip\noindent{\bf Besov Spaces.} Let
$\color{bblue} \mf C^m_c(\bb R\times \bb T)$, $m\in\bb N_0$, be the dual
of $C^m(\bb R\times \bb T)$. Elements of $\mf C^m(\bb R\times \bb T)$
are denoted by $X$, and by the gothic characters $\mf u$, $\mf v$.  We
represent by $\color{bblue} X (f)$ or $\color{bblue} \< X \,,\, f\>$
the value at $f\in C^m_c(\bb R\times \bb T)$ of the bounded linear
functional $X$.

For $z\in \mathbb{R}\times \mathbb{T}$, denote by $\color{bblue} \bb
B(z,a)$ the open ball in $\bb R \times \bb T$ of radius $a>0$ centered
at $z$.  Let $\color{bblue} B_m$, $m \in \bb N_0$, be the set of all
functions $g$ in $C^m_c(\mathbb{R}\times \bb T)$ whose support is
contained in the ball $\bb B(0,1/4)$ and such that $\|g\|_{C^m(\bb
  R\times \bb T)} \le 1$.

For $0<\delta \le 1$, $z \in \bb R \times \bb T$, and a continuous
function $g\colon \mathbb{R}\times \bb T \to \bb R$ whose support is
contained in $\bb B(0,1/4)$, denote by $S^\delta_z g$ the function
defined by
\begin{equation}
\label{35}
(S^\delta_z g)\,  (w)
\;:=\;
\begin{cases}
\delta^{-2} \,g \big(\, [w-z]/\delta\, \big)\;,
& w \in \bb B(z,\delta) \;, \\
0\;, & \text{otherwise}\;.
\end{cases}
\end{equation}

Fix $\alpha<0$ and set $m= -\, \lfloor \alpha \rfloor$, where, for
$r\in \bb R$,
$\color{bblue} \lfloor r \rfloor \,=\, \max \big\{\, k\in \bb Z : k
\le r\, \big\}$.  For $S<T$, and an element $X$ in
$\mf C^m(\bb R \times \bb T)$, let
$\|\,\cdot\, \|_{C^\alpha([S,T]\times \bb T)}$ be the semi-norm
defined by
\begin{equation}
\label{71}
\|X \|_{C^\alpha([S,T]\times \bb T)} \; := \;
\sup_{\delta \in (0,1] } \,  \sup_{z\in [S,T]\times \bb T} \,
\sup_{g \in B_m} \frac{1}{\delta^\alpha} \,
\big|\, \< X \,,\, S^\delta_z g \> \,\big| \;.
\end{equation}
Denote by $\color{bblue} C^\alpha = C^\alpha (\bb R \times \bb T)$ the
subspace of $\mf C^m(\bb R \times \bb T)$ of all elements $X$ such
that $\| X \|_{C^\alpha([S,T]\times \bb T)} <\infty$ for all $S<T$.

\smallskip\noindent{\bf White noise.}
Let $\xi$ be a white-noise on $\bb R \times \bb T$ defined on some
probability space $\color{bblue} (\Omega, \mc F, \bb P)$. Expectation
with respect to $\bb P$ is represented by $\color{bblue} \bb
E$. Hence, $\xi$ is a centered Gaussian field on $\bb R \times \bb T$
whose covariance is given by $\bb E[\, \xi(t,x)\, \xi(t',x')\,] =
\delta_{x,x'}\, \delta_{t,t'}$, where $\delta_{a,b}=1$ if $a=b$ and
$0$ otherwise.

An elementary computation shows that for every $N\ge 1$, there exists
a finite constant $C_N$ such that
\begin{equation*}
\bb E\big[\, \Big(\, \int (S^\delta_z f)(w)\, \xi(w)\, dw
\Big)^{2N}\Big] \;\le\; C_N \delta^{-2N}
\end{equation*}
for all $z\in \bb R \times \bb T$, $0<\delta \le 1$ and continuous
function $f: \bb R \times \bb T \to \bb R$ whose absolute value is
bounded by $1$ and support is contained in $[-1,1]\times \bb T$, By
\cite[Theorem 2.7]{ChaWeb15}, for every $\alpha<-1$, there is a
version of $\xi$ which belongs to $C^{\alpha} (\bb R \times \bb
T)$. More precisely, there exists a centered Gaussian field
$\widehat \xi$ such that $\< \widehat \xi\,,\, f\> = \< \xi\,,\, f\>$
almost surely for all continuous function $f$ with bounded support,
and
$\bb E[ \, \|\, \widehat \xi \, \|_{C^{\alpha} ([-T,T]\times \bb T)} \,]
<\infty$ for all $T>0$.

Denote by $\color{bblue} \varrho : \bb R^2 \to \bb R_+$ a nonnegative,
symmetric, smooth mollifier whose support is contained in
$[-s_0,s_0]\times (\, -1/4 \,,\, 1/4\, )$, for some $s_0>0$, and which
integrates to $1$:
\begin{equation}
\label{s23}
\varrho(z) \;\ge\; 0 \;, \quad
\varrho (-z) \;=\; \varrho (z) \;, \quad
\int_{\bb R^2} \varrho (z)\, dz \;=\; 1\;.
\end{equation}
As the support in contained in $[-s_0,s_0]\times (-1/4,1/4)$, we also
consider $\varrho$ as a mollifier acting on $\bb R \times \bb T$. For
$\varepsilon>0$, let
\begin{equation}
\label{14}
\varrho_\varepsilon \;=\; S^\varepsilon_0
\varrho\;, \quad \xi_\varepsilon \;=\;  \varrho_\varepsilon \,*\, \xi \;.
\end{equation}
By \cite[Theorem 1.4.2]{AdlTay07}, for every $\varepsilon>0$, almost
surely, the mollified field $\xi_\varepsilon$ is smooth in the sense
that it has derivatives of all orders.

\smallskip\noindent{\bf Fractional Laplacian.}  Denote by
$(-\Delta)^{1/2}$ the fractional Laplacian operator defined on smooth
functions $f: \bb T\to \bb R$ by
\begin{equation}
\label{01}
\big[\, (-\Delta)^{1/2}\, f\, \big]\, (x)
\;:=\;
\frac 1{2\,\pi}\; \mf P \int_{\mathbb{R}}
\frac{2\, \widehat f(x) - \widehat f(x+y) - \widehat f(x-y)}{y^2}
\; dy \; ,
\end{equation}
where $\widehat f:\bb R\to \bb R$ represents the periodic function which
coincides with $f$ on $[-1/2,1/2)$, and $\mf P$ stands for principal value
of the integral. The operator $-\, (-\, \Delta)^{1/2}$ corresponds to
the generator of the Cauchy process. Some properties of this operator
and its semigroup are reviewed in Section \ref{sec03}.

Let $\color{bblue} (P_t: t\ge 0)$ be the associated semigroup, which
acts on continuous functions, and let $\color{bblue} p(t,x)$ be its
density, so that $(P_t f)(x) = \int_{\bb T} p(t,y-x) \, f(y)\, dy$ for
all continuous function $f:\bb T\to \bb R$, $t\ge 0$. We present in
\eqref{04} and explicit formula for $p(t,x)$.

Fix $T_0\ge 1$.  Denote by $\ms H : \bb R \to [0,1]$ a smooth
functions such that $\ms H(t) = 1 $ for $t\le 0$ and $\ms H(t) =0$ for
$t\ge 1$.  Let $\color{bblue} q_{T_0}: \bb R \times \bb T \to \bb R_+$
be given by
\begin{equation*}
q_{T_0}(t,x) \;=\; p(t,x) \, \ms H(t-T_0)\;.
\end{equation*}
We often omit $T_0$ from the notation. Clearly, $q_{T_0}$ coincides
with $p$ on $(-\infty, T_0] \times \bb T$ and it has support contained in
$[0,T_0+1]\times \bb T$.

Let $\mf v_\varepsilon := q * \xi_\varepsilon$, $\varepsilon\ge 0$, be
the centered Gaussian random field on $\bb R \times \bb T$ defined by:
\begin{equation}
\label{08b}
\mf v_\varepsilon (t,x)
\;:=\;
\int_{\bb R} ds\, \int_{\bb T} dy\; q(t-s, x-y)\,
\xi_{\varepsilon}(s,y)\;.
\end{equation}
Here, $\mf v_0$, also denoted by $\mf v$, is the Gaussian random field
given by the previous formula with $\xi_0 = \xi$.

Denote by $Q_\varepsilon$, $\varepsilon\ge 0$, the covariances of the
fields $\mf v_\varepsilon$: For $z$, $z'$ in $\bb R \times \bb T$,
\begin{equation*}
Q_\varepsilon (z,z^\prime) \;=\;
\bb E \big[\, \mf v_\varepsilon (z)\; \mf v_\varepsilon (z')\,\big] \;.
\end{equation*}
A change of variables yields that
$Q_\varepsilon (z,z^\prime) = Q_\varepsilon (0,z^\prime-z)$. Denote
this later quantity by $Q_\varepsilon (z^\prime-z)$.  According to
Lemma \ref{s22}, $\{\mf v (z) : z\in \bb R\times \bb T\}$ is a
log-correlated Gaussian field. More precisely, there exists a exist a
continuous, bounded function $R : \bb R \times \bb T \to \bb R$, such
that
\begin{equation}
\label{58}
Q (z) \; = \;
\frac{1}{2\pi} \ln^+ \frac 1{4 \, \Vert z\Vert} \;+\;  \; R(z) \;,
\quad z\,\in\, \bb R \times \bb T\;,
\end{equation}
where $Q(z) = Q_0(z)$ and $\color{bblue} \ln^+ t = \max \{\ln t , 0\}$.

\smallskip\noindent{\bf Gaussian multiplicative chaos.}  Let
$X_{\gamma, \varepsilon}$, $\varepsilon>0$, $\gamma\in \bb R$, be the
random field defined by
\begin{equation}
\label{30}
\< X_{\gamma, \varepsilon} \,,\, f \> \; :=\;  \int f (z)\,
e^{\gamma \, \mf v_\varepsilon (z) \,-\,
  (\gamma^2/2)\,  E[\mf v_\varepsilon(0)^2]} \, dz\;,
\quad f\in C^\infty_c(\bb R \times \bb T)\;.
\end{equation}
It follows from \eqref{58}, as stated in Lemma \ref{s19c}, that
there exists a finite constant $C(\varrho)$ such that
\begin{equation*}
E[\mf v_\varepsilon(0)^2] \;=\;
\frac 1{2\pi} \, \ln \, \frac 1\varepsilon
\;+\; C(\varrho) \;+\; R (\varrho, \varepsilon)
\;,
\end{equation*}
where $R (\varrho, \varepsilon)$ represents a remainder whose absolute
value is bounded by $C_0(\varrho) \varepsilon^2$. Hence,
\begin{equation*}
\< X_{\gamma, \varepsilon} \,,\, f \>  \;=\;
[1 \,+\, o(\varepsilon)]\,
\int f (z)\, A(\varrho)\, \varepsilon^{\gamma^2/4\pi} \,
e^{\gamma \, \mf v_\varepsilon (z)}\, dz\;,
\end{equation*}
where $A(\varrho) \,=\, \exp\{ -\, (\gamma^2/2)\, C(\rho)\,\}$.

\begin{theorem}
\label{mt1}
Fix $0<\gamma^2 \,< \, (4/15) (5 - \sqrt{10})\, 4\pi$,
$\alpha< \alpha_\gamma := b\, \min\{\, -3 \,,\, - \, \big(\, \sqrt{1 +
  (8/b)} \,-\, 1\,\big) \,\}$, where $b= \gamma^2/4\pi$.  Then, as
$\varepsilon \to 0$, $X_{\gamma,\varepsilon}$ converges in probability
in $C^\alpha$ to a random field, denoted by $X_\gamma$. The limit does
not depend on the mollifier $\varrho$. Moreover, for each
$p\in \bb N$, $1\le p < 8\pi/\gamma^2$, there exists a finite constant
$C(p, \gamma)$ such that
\begin{equation*}
\mathbb{E} \, \big[\, |\,
\langle X_{\gamma} \,,\,
S^\delta_z f \rangle \, |^p \, \big]
\; \le \; C(p,\gamma) \, \Vert f\Vert^p_\infty\,
\delta^{-p(p-1)(\gamma^2/4\pi)}
\end{equation*}
for every $\delta$ in $(0, 1)$, $z\in \bb R \times \bb T$ and
continuous function $f: \bb R \times \bb T \to \bb R$ whose support is
contained in $\bb B (0,1/4)$
\end{theorem}

As $\mf v$ is a log-correlated Gaussian field, $X_\gamma$ is the
so-called Gaussian multiplicative chaos (GMC), introduced by Kahane in
\cite{Kah85}.

\smallskip\noindent{\bf A long-range Sine-Gordon equation.}  Fix 
$u_0$ in $C^{\beta_0}(\bb T)$ for some $\beta_0>0$, $\gamma\in \bb R$,
and denote by $\mf u_\varepsilon$, $0< \varepsilon <1$, the solution
of
\begin{equation}
\label{09}
\begin{cases}
\partial_t \mf u_\varepsilon \;=\; - \,  (-\Delta)^{1/2} \mf u_\varepsilon
\;-\; A(\varrho) \, \, \varepsilon^{\gamma^2/4\pi} 
\sinh(\gamma \mf  u_\varepsilon) \,+\, \xi_\varepsilon \\
\mf  u_\varepsilon(0) \,=\, u_0 \,+\, \mf v_\varepsilon(0)\;,
\end{cases}
\end{equation}
where $\mf v_\varepsilon(0)$ is given by \eqref{08b}.

Let
$\color{bblue} \alpha_\gamma := b\, \min\{\, -3 \,,\, - \, \big(\,
\sqrt{1 + (8/b)} \,-\, 1\,\big) \,\}$, where $b= \gamma^2/4\pi$.
By definition, $\alpha_\gamma<0$ and, according to Remark
\ref{rm3} below, $\alpha_\gamma>-1/2$ if $0<\gamma^2 < \pi/7$.

Fix $\alpha \in (-1/2, \alpha_\gamma)$ and choose $\kappa$ small
enough for $0<2\kappa<1+2\alpha$. Let $\beta= \alpha + 1 - 2\kappa$.
Note that $0< \beta <1$ and $\alpha + \beta>0$.

\begin{theorem}
\label{mt2}
Fix $0<\gamma^2 < \pi/7$, $\alpha \in (-1/2, \alpha_\gamma)$ and $u_0$
in $C^\beta(\bb T)$. There exists an almost surely, strictly positive
random variable $\tau$, $\bb P[\tau>0]=1$, with the following
property.

For each $0< \varepsilon < 1$, there exists a unique solution in
$C^\beta([0,\tau] \times \bb T)$ of the equation \eqref{09}, denoted
by $\mf u_\varepsilon$. As $\varepsilon\to 0$, the sequence
$\mf u_\varepsilon$ converges in probability in
$C^\beta([0,\tau] \times \bb T)$ to a random field $\mf u$
which does not depend on the mollifier $\varrho$.
\end{theorem}

The proof of Theorem \ref{mt2} follow the approach proposed by
\cite{HaiShe16} in the context of sine-Gordon equations, and
\cite{Garban18} for dynamical Liouville equation. It relies on a
Schauder estimate.

\begin{remark}
\label{rm9}
By extending to $\bb R \times \bb T$ the theory of Gaussian
multiplicative chaos, along the lines of \cite{RobVar10}, one can
extend the validity of Theorems \ref{mt1} and \ref{mt2} to a larger
range of $\gamma$. We leave this for a future work in which regularity
structures will be used to extend the range up to criticality.
\end{remark}

\smallskip\noindent{\bf A Schauder estimate for the fractional
  Laplacian.} Let $q_{T_0,z} : \bb R \times \bb T \to \bb R_+$,
$z=(t,x) \in (0,\infty)\times \bb R$, be given by
$\color{bblue} q_{T_0,z}(w) = q_{T_0}(z-w)$. 

\begin{theorem}
\label{mt3}
Fix $-1<\alpha<0$, $0<\kappa<1+\alpha$. Then, there exists a finite
constant $C(\kappa, T_0)$ such that
\begin{equation*}
\big\vert \, X(q_{T_0,z}) \,-\, X(q_{T_0,z'}) \, \big\vert
\;\le\; C(\kappa,T_0) \,
\big\Vert \, z\,-\, z' \, \big\Vert^{1+\alpha-\kappa}\,
\big\Vert \, X \, \big\Vert_{C^{\alpha}([S-T_0-3,T+4]\times  \bb T)}
\;.
\end{equation*}
for all $X\in C^\alpha(\bb R \times \bb T)$, $S<T$, $z=(t,x)$,
$z'=(t',x') \in [S,T] \times \bb T$ such that
$\Vert z- z'\Vert\ \le 1/4$.
\end{theorem}

This result is one of the main novelties of this article. One of the
major difficulties of the proof lies on the fact that the transition
density $p(t,x)$ of the fractional Laplacian does not belong to $C^1$
due to the long jumps. One needs, in particular, to provide a meaning
to $X(q_{T_0,z})$, approximating $q_{T_0,z}$ by smooth functions.

\smallskip\noindent{\bf Sketch of the proof.}  Following DaPrato and
Debussche \cite{DaPrato2003}, we expand the solution
$\mf u_\varepsilon(t,x)$ around the solution of the linear equation
\begin{equation}
\label{62}
\begin{cases}
\displaystyle \vphantom{\Big\{}
\partial_t \mf f_\varepsilon \;=\;
-\, (-\, \Delta)^{1/2}\, \mf f_\varepsilon \,+\;  \xi_\varepsilon \;,
\quad (t,x) \,\in\,  (0, \infty) \times \bb T\;, \\
\displaystyle  \vphantom{\Big\{}
\mf f_\varepsilon (0,x) \,=\,f(x) \;, \quad x\in \bb T \;.
\end{cases}
\end{equation}
The solution of \eqref{62} can be represented in terms of the
semigroup $(P_t : t\ge 0)$ of the Cauchy process as
\begin{equation*}
\mf f_\varepsilon(t) \;=\; \int_0^t P_{t-s} \, \xi_\varepsilon(s)\;
ds\; +\; P_t f\;.
\end{equation*}
Recall the definition of the Gaussian field $\mf v_\varepsilon$
introduced on \eqref{08b}. Comparing $\mf v_\varepsilon$ to
$\mf f_\varepsilon$ and choosing an appropriate initial condition $f$
yields that
\begin{equation*}
\partial_t \mf v_\varepsilon \;+\;
(-\, \Delta)^{1/2}\, \mf v_\varepsilon \,-\;  \xi_\varepsilon
\;=\; R_\varepsilon\;,
\end{equation*}
where $R_\varepsilon$ is a smooth function with nice asymptotic
properties. 

By writing the solution $\mf u_\varepsilon$ of equation \eqref{09} as
$\mf v_\varepsilon + \mf w_\varepsilon$ yields that
$\mf w_\varepsilon$ solves the equation 
\begin{equation}
  \label{61}
\left\{
\begin{aligned}
& \partial_t \mf w_\varepsilon \;=\;
- \,  (-\Delta)^{1/2} \mf w_\varepsilon
\,-\,  \frac 12\, X_{\gamma, \varepsilon} \, e^{\gamma \mf w_\varepsilon}
\,+\,  \frac 12\, X_{-\gamma, \varepsilon} \, e^{-\gamma \mf w_\varepsilon}
\;-\; R_\varepsilon \;, \\
& \mf w_\varepsilon(0) \,=\, u_0  \;,
\end{aligned}
\right.
\end{equation}
where $X_{\gamma, \varepsilon} $ has been introduced in \eqref{30}.

It is not difficult to show that the sequence of random fields
$\mf v_\varepsilon$ converges as $\varepsilon \to 0$. The proof of the
convergence of $\mf u_\varepsilon$ is thus reduced to the one of
$\mf w_\varepsilon$.

The proof of local existence and uniqueness of solutions to \eqref{61}
is divided in two steps. we first show that the sequence
$X_{\gamma, \varepsilon}$ converges in probability in $C^\alpha$ to a
random field, represented by $X_{\gamma,0}$. This is the content of
Theorem \ref{mt1}. Then, writing the solution of \eqref{61} as
\begin{equation*}
\begin{aligned}
\mf w_\varepsilon (t) \;& =\; -\, \frac 12\,
\int_0^t P_{t-s} \Big\{ e^{\gamma \mf w_\varepsilon (s)} \,
X_{\gamma, \varepsilon}  (s)
\;-\, e^{- \gamma \mf w_\varepsilon (s)} \,
X_{\gamma, \varepsilon} (s)  \Big\}\; ds \\
&-\; \int_0^t P_{t-s} R_\varepsilon (s)  \; ds
\;+\; P_t u_0\;,
\end{aligned}
\end{equation*}
we prove the existence and uniqueness of a fixed point for this
equation, including the case $\varepsilon =0$, in an appropriate Besov
space. Moreover, we show that $\mf w_\varepsilon$ converges to $\mf
w_0$ as $\varepsilon \to$ in some H\"older space. The Schauder
estimate is one of main tools here.

\smallskip\noindent{\bf Organization of the paper.}  The paper is
organised as follows: In Section \ref{sec03}, we present the main
properties of the fractional Laplacian needed in the article. In
Section \ref{sec04}, we prove Theorem \ref{mt3}.  In Section
\ref{sec07}, we examine the properties of the log-correlated Gaussian
random field $\mf v_\varepsilon$, introduced in \eqref{08b}. In
Section \ref{sec05}, we prove Theorem \ref{mt1} and, in Section
\ref{sec06}, Theorem \ref{mt2}.

\section{The Cauchy process}
\label{sec03}

We present in this section properties of the solutions of the
fractional heat equation needed in article. We start examining the
fractional Laplacian operator. Denote by $L^2(\bb T)$ the space of
complex valued functions $f: \bb T\to \bb C$ endowed with the scalar
product given by
\begin{equation*}
\< f\,,\, g\> \;=\; \int_{\bb T} f(x)\, \overline{g(x)}\, dx\;,
\end{equation*}
where $\overline{a}$ represents the complex conjugate of $a\in \bb C$.

Denote by $\{e_k : k\in \bb Z\}$ the orthonormal basis of $L^2(\bb T)$ given
by $\color{bblue} e_k(x) = e^{2\pi i k x}$. An elementary computation shows
that the functions $e_k$ are eigenvectors of the operator $(-\Delta)^{1/2}$:
\begin{equation}
\label{05}
(-\Delta)^{1/2} e_k \;=\; 2\, \pi \, |k| \, e_k\;, \quad k\in \bb Z\;.
\end{equation}

Denote by $G: \bb T \to \bb R$ the Green function associated to the
operator $(-\Delta)^{1/2}$:
\begin{equation*}
G (x) \;=\;
\sum_{k \in \mathbb{Z}\setminus\{0\}} \frac{e_k(x)}{2\pi |k|} \;.
\end{equation*}
A straightforward computation on the Fourier space yields that for any
function $f$ in the domain of the operator $(-\Delta)^{1/2}$ and orthogonal to
$e_0$ [that is, such that $\int_{\bb T} f(x)\, dx = 0$]
\begin{equation}
\label{27}
[(-\Delta)^{1/2} f] \,*\, G \;=\; f\;,
\end{equation}
where $f*g$ stands for the convolution of two functions $f$, $g$ in
$L^2(\bb T)$:
\begin{equation*}
(f*g)(x) \;=\; \int_{\bb T} f(x-y)\, g(y)\, dy\;.
\end{equation*}
The Green function can also be computed, it is given by
\begin{equation}
\label{12}
G (x) \;=\; \frac 1{\pi}\sum_{k\ge 1}
\frac{\cos (2\pi k x)}{k}
\;=\; -\, \frac {1}{\pi}\, \ln \big[\, 2 \, \sin (\pi x)\, \big]\;.
\end{equation}

The semigroup $p(t,x)$ associated to the fractional Laplacian
$(-\Delta)^{1/2}$ can be computed explicitly. Denote by $\bs p(t,x)$ the
solution on $\bb R$ of the differential equation
\begin{equation*}
\left\{
\begin{aligned}
& \partial_t p \;=\; \ms A \, p \\
& p(0,x) \,=\, \delta_0 (x) \;,
\end{aligned}
\right.
\end{equation*}
where $\ms A$ stands for the fractional Laplacian $-\,
(-\Delta)^{1/2}$ defined on $\bb R$ (instead of $\bb T$).  In Fourier
coordinates, it is given by
\begin{equation*}
\widehat{\bs p} (t,\theta) \;=\; e^{-\, |\theta|\, t}\;.
\end{equation*}
Inverting the Fourier transform yields that
\begin{equation}
\label{07}
\bs p (t,x) \;=\; \frac 1{\pi}\, \frac{t}{x^2+ t^2}\;,\quad
x\,\in\, \bb R\;,\;\; t\,>\, 0\;,
\end{equation}
and $\bs p (t,x) =0$ if $t<0$.

Let $p: (0,\infty) \times [-1/2, 1/2) \to \bb R_+$ be the projection
of the transition probability $\bs p$ on the torus:
\begin{equation}
\label{04}
p (t,x) \;=\; \bs p(t,x) \;+\; p_\star (t,x)
\;: =\; \bs p(t,x) \;+\;
\sum_{k\not =0} \bs p(t,x+k)\;, \quad x\in  [-1/2, 1/2) \;,
\end{equation}
where the last sum is performed over all intergers $k\in \bb Z$
different from $0$. An elementary computation shows that the function
$p$ is smooth in its domain of definition
$(0,\infty) \times [-1/2, 1/2)$, where $[-1/2, 1/2)$ represents the
torus.

Although $p$ is smooth as a function defined on the torus, this is not
the case of $p_\star$. However, if we assume that $p_\star$ is defined
on $(0,\infty) \times (-1/2 - \kappa, 1/2)$ for some $0<\kappa<1/4$,
it is not difficult to show that this function is smooth on
$(0,\infty) \times (-1/2 - \kappa, 1/2)$ and that it is uniformly
bounded, as well as its derivatives: For all $j$, $k\ge 0$, there
exists a finite constant $C_{j,k}$ such that
\begin{equation*}
\sup_{z\in (0,\infty) \times (-1/2 - \kappa, 1/2)}
\big|\, \partial^j_t\, \partial^k_x \, p_\star (z) \, \big|
\;\le\; C_{j,k}\;.
\end{equation*}

Let $\ms A$ be the annulus on $\bb R^2$ given by
$\color{bblue} \ms A = \{(t,x)\in \bb R^2 : 1/2 < t^2 + x^2 < 2 \}$,
and set
$\color{bblue} \ms A_n = \{ (t,x)\in \bb R^2 : (2^n t, 2^n x) \in \ms
A\}$,
$\color{bblue} \ms A^+_n = \{ (t,x)\in (0,\infty) \times \bb R : (t,
x) \in \ms A_n\}$.
It follows from the previous estimates on $p$ and elementary
computations that for all $j\ge 0$, $k\ge 0$ there exists a finite
constant $C_{j,k}$ such that for all $n\ge 2$,
\begin{equation}
\label{23}
\sup_{z\in \ms A^+_n}
\big|\, \partial^j_t\, \partial^k_x \, p (z) \, \big|
\;\le\; C_{j,k} \, 2^{(1+j+k)n}\;.
\end{equation}

It is also not difficult to show from \eqref{04} that there exists a
finite constant $C_0$ such that
\begin{equation}
\label{23b}
\big|\, \partial_t p \, (z) \, \big|
\;\le\; \frac {C_0}{\Vert z\Vert} \;, \quad
\big|\, \partial_x p \, (z) \, \big|
\;\le\; \frac {C_0}{\Vert z\Vert} \;, \quad
p \, (z) \;\le\; C_0 \Big\{ \, 1 \,+\, 
\frac {1}{\Vert z\Vert} \, \Big\}
\end{equation}
for all $z=(t,x)$ such that $t >0$.

Recall from Section \ref{sec01} that we denote by $(P_t : t \ge 0)$
the semigroup associated to the generator $-(-\Delta)^{1/2}$: $P_t$
acts on continuous functions $f: \bb T \to \bb R$ as
$(P_t f)(x) = \int_{\bb T} p(t,y-x)\, f(y)\, dy$, where $p$ is the
transition density introduced in \eqref{04}.

Denote by $(Z_t:t\ge 0)$ the Cauchy process. This is the Markov
process on $\bb R$ which starts from the origin and whose semigroup is
given by $\bs p$, introduced in \eqref{07}.

\begin{lemma}
\label{s03}
Fix $\beta \in (0,1)$. There exists a constant $C_0 = C_0 (\beta)$,
such that for all $T>0$, $u \in C^{\beta}(\mathbb{T})$,
\begin{equation*}
\| \, P_t \, u \, \|_{C^{\beta}([0,T] \times
\mathbb{T})} \le C_0 \, \| \, u\, \|_{C^{\beta}(\mathbb{T})} \;.
\end{equation*}
\end{lemma}

\begin{proof}
Fix $u \in C^\beta (\mathbb{T})$, $T>0$, $x$, $y \in \mathbb{T}$ and
$0\le s < t \le T$. Denote by $\bb Z_t$ the projection of the Cauchy
process on the torus $[-1/2,1/2)$ so that 
\begin{align*}
\big |\, (P_{t}u) (x) \,-\,  (P_{s}u)(y)\, \big| \;=\;
\big|\,\mathbb{E}\, [\, u(x + \bb Z_t) 
\,-\,  u(y + \bb Z_s)\, ]\, \big|\;,
\end{align*}
where $\bb E$ represents the expectation with respect to the Cauchy
process $Z_t$. Recall that we represent by $|\,\cdot\,|$ the distance
on the torus, although is not a norm. The previous expression is
bounded by
\begin{equation*}
\| u \|_{C^\beta(\mathbb{T})} \,
\mathbb{E}\, \big [\, |\, x + \bb Z_t - y - \bb Z_s\, |^\beta \,\big]
\; \le \;\|u\|_{C^\beta(\mathbb{T})} \, \Big\{\, |\, x-y \, |^\beta
\,+\, \mathbb{E}\, \big[ \, |\, \bb Z_{t} \,-\,
\bb Z_{s} \, |^\beta \,\big] \, \Big\}
\end{equation*}
because $(a+b)^\theta \le a^\theta + b^\theta$ for $a$, $b>0$,
$0<\theta<1$. Since
$|\, \bb Z_{t} \,-\, \bb Z_{s} \, | \,\le\, |\, Z_{t} \,-\, Z_{s} \,
|$, the increments are stationary and the process is self-similar,
\begin{equation*}
\mathbb{E}\, \big[ \, |\, \bb Z_{t} \,-\, \bb Z_{s} \, |^\beta \,\big]
\;\le\; 
\mathbb{E}\, \big[ \, |\, Z_{t} \,-\, Z_{s} \, |^\beta \,\big]
\;=\; 
\mathbb{E}\, \big[ \, |\, Z_{t-s} \, |^\beta \,\big]
\;=\; 
\mathbb{E}\, \big[ \, |\, (t-s) Z_{1} \, |^\beta \,\big]\;.
\end{equation*}
The right-hand side of the penultimate formula is thus bounded above
by  
\begin{equation*}
\;\|u\|_{C^\beta(\mathbb{T})} \, \Big\{\, |\, x-y \, |^\beta
\,+\, C_0\, (t-s)^\beta \, \Big\}\;,
\end{equation*}
where $C_0 = \mathbb{E}\, [ \, |\, Z_{1} \, |^\beta \,]$.
This completes the proof of the lemma.
\end{proof}

\section{A Schauder estimate}
\label{sec04}

We prove in this section a Schauder estimate for the kernel $p (t,x)$
of the fractional Laplacian on the torus. We follow the approach based
on the homogeneity of the kernel under scaling, in the sense that
$\bs p (t/\delta , x/\delta ) = \delta\, \bs p (t,x)$ for all
$(t,x) \in \bb R^2 \setminus \{0\}$, $\delta>0$, cf.  \cite{Sim97,
  Hai14}. However, on the torus, the kernel is not homogeneous, and,
more importantly, due to the non-locality of the generator, the
transition density $p(t,x)$ is not $C^1$ at $t=0$. In particular, it
does not belong to the domain of the distributions in $C^\alpha$, and
a plethora of arguments and bounds are needed to define and bound the
main quantities such as $X(p)$.

Fix $T_0\ge 1$.  Denote by $\ms H : \bb R \to [0,1]$ a smooth
functions such that $\ms H(t) = 1 $ for $t\le 0$ and $\ms H(t) =0$ for
$t\ge 1$.  Let $\color{bblue} q_{T_0}: \bb R \times \bb T \to \bb R_+$
be given by
\begin{equation}
\label{06}
q_{T_0}(t,x) \;=\; p(t,x) \, \ms H(t-T_0)\;.
\end{equation}
We often omit $T_0$ from the notation. Clearly, $q_{T_0}$ coincides
with $p$ on $(-\infty, T_0] \times \bb T$, it has support contained in
$[0,T_0+1]\times \bb T$, it belongs to $C^2(\Omega_{0 \bullet })$ and
for every $t_0>0$, $m\ge 1$, there exists a finite constant $C_m$ such
that
\begin{equation*}
\Vert \, q_{T_0}\, \Vert_{C^m(\Omega_{t_0 \bullet })} \;\le\;
C_m\, \Vert \, p\, \Vert_{C^m(\Omega_{t_0 \bullet})} \;.
\end{equation*}
Here and below, for $s<t$,
$\color{bblue} \Omega_{s,t} = (s,t) \times \bb T$,
$\color{bblue} \Omega_{\bullet t} =(-\infty,t) \times \bb T$,
$\color{bblue} \Omega_{t \bullet } =(t,\infty) \times \bb T$.

Let $q_{T_0,z} : \bb R \times \bb T \to \bb R_+$,
$z=(t,x) \in \Omega_{0\bullet}$, be given by
$\color{bblue} q_{T_0,z}(w) = q_{T_0}(z-w)$. The main result of this
section reads as follows.

\begin{theorem}
\label{t01}
Fix $-1<\alpha<0$, $0<\kappa<1+\alpha$. Then, there exists a finite
constant $C(\kappa, T_0)$ such that
\begin{equation*}
\big\vert \, X(q_{T_0,z}) \,-\, X(q_{T_0,z'}) \, \big\vert
\;\le\; C(\kappa,T_0) \,
\big\Vert \, z\,-\, z' \, \big\Vert^{1+\alpha-\kappa}\,
\big\Vert \, X \, \big\Vert_{C^{\alpha}([S-T_0-3,T+4]\times  \bb T)}
\;.
\end{equation*}
for all $X\in C^\alpha(\bb R \times \bb T)$, $S<T$, $z=(t,x)$,
$z'=(t',x') \in [S,T] \times \bb T$ such that
$\Vert z- z'\Vert\ \le 1/4$.
\end{theorem}

\begin{corollary}
\label{c-t01}
Fix $-1<\alpha<0$, $0<\kappa<1+\alpha$. Then, there exists a finite
constant $C(\kappa,T_0)$ such that
\begin{equation*}
\Vert u \Vert_{C^{1+\alpha-\kappa}([S,T] \times \bb T)} \;\le\;
C(\kappa,T_0)
\big\Vert \, X \, \big\Vert_{C^{\alpha}([S-T_0-3,T+4]\times  \bb T)}
\;.
\end{equation*}
for all $X\in C^\alpha(\bb R \times \bb T)$, $S<T$, where
$u = u_{T_0} : \bb R \times \bb T \to \bb R$ is given by
$u(z) = X(q_{T_0,z})$.
\end{corollary}

Part of the proof of Theorem \ref{t01} consists in giving a meaning to
$X(q_{T_0,z})$ since, as pointed our earlier, $q_{T_0,z}$ does not
belong to the domain of a distribution in $C^\alpha$.  We start with a
simple estimate on $C^\alpha$.

\begin{lemma}
\label{a05}
Fix $\alpha<0$ and let $m= -\, \lfloor \alpha \rfloor$.  There exists
a finite constant $C_0$ such that for all $a<b$, $S<T$,
$0<\delta \le 1$, $z\in \Omega_{S,T}$, and function $g$ in
$C^m(\mathbb{R}\times \bb T)$ whose support is contained in
$\Omega_{a,b}$,
\begin{equation*}
\big|\, \< X \,,\, S^\delta_z g \> \,\big| \;\le\; C_0\, (1+b-a)\;
\delta^\alpha \, \|g\|_{C^m(\mathbb{R}\times \bb T)} \,
\|X \|_{C^\alpha([S+a-1,T+b+1]\times \bb T)} \;.
\end{equation*}
\end{lemma}

\begin{proof}
It follows from the definition of the seminorms $\|X
\|_{C^\alpha([S,T]\times \bb T)}$, introduced in \eqref{71}, that for
all functions $g$ in $C^m(\mathbb{R}\times \bb T)$ whose support is
contained in $\bb B(0,1/4)$, every $0<\delta \le 1$, $z \in
[S,T]\times \bb T$.
\begin{equation}
\label{59}
\big|\, \< X \,,\, S^\delta_z g \> \,\big| \;\le\;
\delta^\alpha \, \|g\|_{C^m(\bb R\times \bb T)} \,
\|X \|_{C^\alpha([S,T]\times \bb T)} \;.
\end{equation}

For each $p\in \bb N$, there exists a function $\varphi$ in $C^p(\bb
R^2)$ whose support is contained in $\bb B(0,1/4)$ and such that
\begin{equation*}
\sum_{j\in \bb Z} \sum_{k=0}^7 \varphi_{j,k}(t,x) \;=\; 1\quad
\text{for all}\;\; (t,x)\,\in\,\bb R \times \bb T\;,
\end{equation*}
where $\varphi_{j,k}(t,x) = \varphi(\, t-(j/8) \,,\, x-(k/8)\, )$.

Fix $p\ge m$ and write $g$ as
$\sum_{j\in \bb Z} \sum_{0\le k \le 7} g_{j,k}$, where
$g_{j,k} = g\, \varphi_{j,k}$. Since the support of $g$ is contained
in $[a,b] \times \bb T$, in the previous sum there are at most
$B_0 (1+b-a)$ terms which do not vanish, for some finite constant
$B_0$. Moreover, for $0<\delta \le 1$, $z\in [S,T]\times \bb T$,
\begin{equation*}
\big|\, \< X \,,\, S^\delta_z g \> \,\big| \;\le\;
\sum_{j, k} \big|\, \< X \,,\, S^\delta_z g_{j,k} \> \,\big| \;.
\end{equation*}
where the sum is performed over the non-vanishing terms.  We may write
$g_{j,k}$ as $S^1_{z_{j,k}} S^1_{- z_{j,k}} g_{j,k}$, where $z_{j,k} =
(\, j/8 \,,\, k/8\, )$. Let $F_{j,k} = S^1_{- z_{j,k}} g_{j,k}$, and
note that the support of $F_{j,k}$ is contained in $\bb B(0,1/4)$.
Since $S^\delta_z S^1_{z_{j,k}} F_{j,k} = S^\delta_{z + \delta\,
  z_{j,k}} F_{j,k}$, the right-hand side of the previous displayed
equation is equal to
\begin{equation*}
\sum_{j, k} \big|\, \< X \,,\, S^\delta_{z + \delta\,
  z_{j,k}} F_{j,k} \> \,\big| \;.
\end{equation*}
By \eqref{59}, this sum is less than or equal to
\begin{equation*}
\delta^\alpha \, \sum_{j, k}
\|F_{j,k}\|_{C^m(\bb R\times \bb T)} \,
\|X \|_{C^\alpha([S+a-1,T+b+1]\times \bb T)}
\end{equation*}
because $z + \delta\, z_{j,k} \in [S+a-1, T+b+1]$ for all $(j,k)$ for
which $F_{j,k}$ does not vanish. Since there exists a finite constant
$B_1$, depending only on $\varphi$, such that $\|F_{j,k}\|_{C^m(\bb
  R\times \bb T)} \le B_1\, \|g\|_{C^m(\bb R\times \bb T)}$, to
complete the proof of the assertion, it remains to recall that are at
most $B_0 (1+b-a)$ non-vanishing terms in the sum.
\end{proof}

As mentioned in the introduction of this section, the kernel $p(t,x)$
does not belong to $C^1(\bb R \times \bb T)$. In particular, if $P_z$,
$z\in (0,\infty) \times \bb T$, stands for the functions defined by
$\color{bblue} P_z(w) = p(z-w)$, $X(P_z)$ is not defined for
distributions $X$ in $C^\alpha (\bb R \times \bb T)$,
$-1<\alpha<0$. The next lemmata provide sufficient conditions which
permit to define $X(P_z)$ as a limit.


Let $\Omega$ be an open set of $\bb R \times \bb T$ and let
$f: \Omega\to \bb R$ be a continuously differentiable function. We
denote by $\Vert \, f\, \Vert_{C^1(\Omega)}$ the norm defined by
\begin{equation*}
\Vert \, f\, \Vert_{C^1(\Omega)} \;=\;
\sum_{j,k} \Vert \, \partial^j_x \, \partial^k_t \, f \,
\Vert_{L^\infty(\Omega)}\;,
\end{equation*}
where the sum is carried out over all $j$, $k$ in $\bb N_0$ such that
$j+k\le 1$.

Let $\varphi: \bb R_+ \to [0,1]$ be the germ of a dyadic partition of
the unity: $\varphi$ is a smooth function such that
\begin{equation}
\label{26}
\varphi(r) \;=\; 0\;\; \text{if}\;\; r\,\not\in\, (1/16 \,,\, 1/4)\;,
\qquad \sum_{n\in\bb Z} \varphi(2^n r) \;=\; 1 \;\;
\text{for}\;\; r>0\;.
\end{equation}
We refer to \cite[Proposition 2.10]{bcd11} for the existence of
$\varphi$. Let $\color{bblue} \varphi_n (r) = \varphi(2^n r)$.  Note
that the supports of $\varphi_n$ and $\varphi_m$ are disjoints
whenever $|n-m|\ge 2$.

Let $\psi: \bb R \to [0,1]$ be a symmetric, smooth function whose
support is contained in $(- 1/2 \,,\, 1/2)$ and such that
\begin{equation}
\label{28}
\sum_{k\in \bb Z} \psi (x \,-\, k/2) \;=\; 1\;, \quad
x\; \in\;  \bb R\;.
\end{equation}
Let $\color{bblue} \psi_n(x) = \psi(2^n x)$, $n\ge 1$. Consider
$\psi$, $\psi_n$ as defined on $\bb R \times \bb T$ and depending only
on the second coordinate. We abuse of notation below and denote by
$k/2^n$ the element $(0,k/2^n)$ of $\bb R \times \bb T$. Note that
\begin{equation}
\label{29}
1 \;=\;  \sum_{k=-2^n+1}^{2^n} \psi (2^n x \,-\, k/2)
\;=\;  \sum_{k=-2^n+1}^{2^n} \psi_n (x \,-\, k 2^{-(n+1)})
\end{equation}
for all $x\in \bb T = (-1/2, 1/2]$.

\begin{lemma}
\label{s15}
Fix $-1<\alpha<0$ and a continuously differentiable function
$f: (0,\infty) \times \bb T \to \bb R$. Assume that there exists
$T_1<\infty$ such that $f(t,x) =0$ for $t\ge T_1$ and that
$\Vert \, f\, \Vert_{C^1(\Omega_{0\bullet})} < \infty$.  Let
$f_n(t,x) = f(t,x)\, \varphi_n(t)$, $n\in \bb N_0$.  Then, there
exists a finite constant $C_0$, independent of $\alpha$ and $f$, such
that
\begin{equation*}
\big| \, X (\, f_n  \,) \, \big| \; \le\;
C_0 \,  2^{-n(1+\alpha)} \Vert \, f\, \Vert_{C^1(\Omega_{0\bullet})}
\, \big \Vert X \big \Vert_{C^\alpha([-1,2]\times \bb T)}
\end{equation*}
for all $X$ in $C^\alpha$, $n\in \bb N_0$.
\end{lemma}

\begin{proof}
For each $n\ge 0$, the function $f_n$ belongs to $C^1_c(\bb R\times
\bb T)$ and its support is contained in $[2^{-(n+4)}, 2^{-(n+2)}]
\times \bb T$. In particular, $X(f_n)$ is well defined.

Recall the definition of $\psi$ introduced in \eqref{28}. By
\eqref{29},
\begin{equation*}
f_n (s,y) \;=\; \sum_{k=-2^n+1}^{2^n}
f(s,y)\, \varphi_n(s) \, \psi_n (x \,-\, k 2^{-(n+1)}) \;.
\end{equation*}

Let $H_n: \bb R \times \bb T \to\bb R$ be given by
\begin{equation*}
H_n(s,y) \;=\; 2^{-2n} \,
f \big(\, 2^{-n} (s,y) \,+\, (0,k 2^{-(n+1)}) \,\big) \,
\varphi(s)\, \psi (y) \;,
\end{equation*}
so that $(S^{2^{-n}}_{(0,k 2^{-(n+1)})} H_n )(s,y) = f(s,y)
\varphi_n(s) \, \psi_n (x \,-\, k 2^{-(n+1)})$. In particular,
\begin{equation}
\label{46}
X(f_n) \;=\;  \sum_{k=-2^n+1}^{2^n}
X\big(\, S^{2^{-n}}_{(0,k 2^{-(n+1)})} H_n \,\big)\;.
\end{equation}
The function $H_n$ belongs to $C^1_c(\bb R \times \bb T)$, it has
support contained in $[1/16,1/4]\times \bb T$, and $\Vert H_n\Vert_{C^1}
\le 2^{-2n} \Vert f\Vert_{C^1(\Omega_{0\bullet})}$. Therefore,
by Lemma \ref{a05},
\begin{equation*}
\big| \, X (\, f_n  \,) \, \big| \; \le\;
C_0 \,  2^n\, 2^{-2n}\, 2^{-n\alpha} \Vert \, f\, \Vert_{C^1(\Omega_{0\bullet})}
\, \big \Vert X \big \Vert_{C^\alpha([-1,2]\times \bb T)}
\end{equation*}
for some finite constant $C_0$. The factor $2^n$ comes from the number
of terms in the sum over $k$.
\end{proof}

\begin{remark}
\label{rm8}
One could be tempted to define $H_n$ as $H_n(s,y) = 2^{-2n}
f(2^{-n} (s,y)) \, \varphi(s)$. But this function is not periodic.
This is the reason for introducing $\psi_n$.
\end{remark}

Let $f: (0,\infty) \times \bb T \to \bb R$ be a function which
satisfies the assumptions of Lemma \ref{s15} and whose support is
contained in $[0,T_1 ]\times \bb T$. Set
\begin{equation*}
\Upsilon (s) \;=\; \sum_{n\ge 0} \varphi_n(s)\;,
\end{equation*}
and write $f$ as $f = f^{(0)} + f^{(1)}$, where $f^{(0)}(t,x) =
f(t,x)\, \Upsilon (t)$, $f^{(1)}(t,x) = f(t,x) \, [1-\Upsilon
(t)]$. In view of Lemma \ref{s15}, we may define $X(f^{(0)})$ as
$\sum_{n\ge 0} X(f_n)$. On the other hand, $f^{(1)}$ belongs to
$C^1_c(\bb R\times \bb T)$ and $X(f^{(1)})$ is well defined. Moreover,
since the support of $f$ is contained in $[0,T_1]\times \bb T$, by Lemma
\ref{a05} with $\delta=1$ and $z=0$,
\begin{equation}
\label{47}
\big|\, X(f^{(1)}) \,\big| \;\le\; C_0\, (1+T_1) \,
\Vert \, f\, \Vert_{C^1(\Omega_{0\bullet})}
\, \big \Vert X \big \Vert_{C^\alpha([-1,T_1+1]\times \bb T)}\;.
\end{equation}
We summarize these observations in the next result.

\begin{corollary}
\label{s15c}
Fix $-1<\alpha<0$.  Let $f: (0,\infty) \times \bb T \to \bb R$ be a
function which satisfies the assumptions of Lemma \ref{s15}. Assume
that the support of $f$ is contained in $[0,T_1]\times \bb T$. Define
$X(f)$ as
\begin{equation*}
X(f) \;=\; \sum_{n\ge 0} X(f_n) \;+\; X(f^{(1)})\;.
\end{equation*}
Then, there exists a finite constant $C_0$, independent of $\alpha$
and $f$, such that
\begin{equation*}
|X(f)| \; \le \, C_0 \, (1+T_1) \, \Vert \, f\, \Vert_{C^1(\Omega_{0\bullet})}
\, \big \Vert X \big \Vert_{C^\alpha([-1,T_1+1] \times \bb T)} \;.
\end{equation*}
\end{corollary}

\begin{proof}
This result follows from \eqref{47} and from Lemma \ref{s15} which
asserts that
\begin{align*}
\sum_{n\ge 0} \big| \, X (\, f_n  \,) \, \big| \; & \le\;
C_0 \, \sum_{n\ge 0}  2^{-n(1+\alpha)}\,  \Vert \, f\, \Vert_{C^1(\Omega_{0\bullet})}
\, \big \Vert X \big \Vert_{C^\alpha([-1,T_1+1] \times \bb T)} \\
& \le\;
C_0 \,
\Vert \, f\, \Vert_{C^1(\Omega_{0\bullet})}
\, \big \Vert X \big \Vert_{C^\alpha([-1,T_1+1] \times \bb T)}
\end{align*}
for some finite constant $C_0$, where we used the fact that
$\alpha>-1$.
\end{proof}

An elementary computation yields that for all $z$, $z'$ in $\bb R
\times \bb T$, $\delta$, $\delta'$ in $(0,1]$ and continuous functions
$f$,
\begin{equation}
\label{36}
S^\delta_z \, S^{\delta'}_{z'} \, f \;=\; S^{\delta\, \delta'}_{z+
  \delta z'} \, f\;.
\end{equation}

\begin{corollary}
\label{s18}
Fix $-1<\alpha<0$. Let $f: (0,\infty) \times \bb T \to \bb R$ be a
function which satisfies the hypotheses of Lemma \ref{s15}.  Assume
that the support of $f$ is contained in $[0,T_1]\times \bb T$.  Then,
there exists a finite constant $C_0$ such that
\begin{equation*}
\big| \, X\big (\, S^\delta_z f \,\big) \, \big| \; \le\;
C_0 \, (1+T_1) \, \delta^\alpha\, \Vert \, f\, \Vert_{C^1(\Omega_{0\bullet})}
\, \big \Vert X \big \Vert_{C^\alpha([S-1,T+T_1+1]\times \bb T)}
\end{equation*}
for all $S<T$, $X$ in $C^\alpha$, $z\in [S,T]\times \bb T$,
$0<\delta\le 1$.
\end{corollary}

\begin{proof}
Recall the decomposition of $f$ as $\sum_{n\ge 0} f_n + f^{(1)}$
introduced in Corollary \ref{s15c}. Since $f^{(1)}$ belongs to $C^1_c
(\bb R\times \bb T)$, by Lemma \ref{a05},
\begin{equation*}
\big| \, X\big (\, S^\delta_z f^{(1)} \,\big) \, \big| \; \le\;
C_0 \, (1+T_1) \, \delta^\alpha\, \Vert \, f\, \Vert_{C^1(\Omega_{0\bullet})}
\, \big \Vert X \big \Vert_{C^\alpha([S-1,T+T_1+1]\times \bb T)}
\end{equation*}

We turn to $X (\, S^\delta_z f_n \,)$. By \eqref{46} and \eqref{36},
\begin{equation*}
X (\, S^\delta_z f_n \,) \;=\;   \sum_{k=-2^n+1}^{2^n}
X\big(\, S^\delta_z\, S^{2^{-n}}_{(0,k 2^{-(n+1)})} H_n \,\big) \;=\;
\sum_{k=-2^n+1}^{2^n}
X\big(\, S^{\delta\, 2^{-n}}_{z + \delta (0,k 2^{-(n+1)})} H_n \,\big) \;.
\end{equation*}
The function $H_n$ belongs to $C^1_c(\bb R \times \bb T)$, its support
is contained in $[1/16,1/4]\times \bb T$, and $\Vert H_n\Vert_{C^1}
\le 2^{-2n} \Vert f\Vert_{C^1(\Omega_{0\bullet})}$. Thus, by Lemma
\ref{a05},
\begin{equation*}
\big| \, X (\, S^\delta_z  f_n  \,) \, \big| \; \le\;
C_0 \,  \sum_{k=-2^n+1}^{2^n}
\delta^\alpha \, 2^{-n\alpha} \,2^{-2n}\,
\Vert \, f\, \Vert_{C^1(\Omega_{0\bullet})}
\, \big \Vert X \big \Vert_{C^\alpha([S-1,T+2]\times \bb T)}
\end{equation*}
for some finite constant $C_0$. As $\alpha>-1$, summing over $n$ we
get that
\begin{equation*}
\sum_{n\ge 0} \big| \, X (\, S^\delta_z  f_n  \,) \, \big| \; \le\;
C_0 \, \delta^\alpha \,
\Vert \, f\, \Vert_{C^1(\Omega_{0\bullet})}
\, \big \Vert X \big \Vert_{C^\alpha([S-1,T+2]\times \bb T)} \;,
\end{equation*}
which completes the proof of the corollary.
\end{proof}

\begin{remark}
\label{rm6}
The set $\{(t,x) : t=0\}$ plays no role in the proof of the previous
results. A similar statement holds for functions $f$ which are smooth
on the set $\Omega_{\bullet t_0} \cup \Omega_{t_0 \bullet}$, $t_0\in
\bb R$. The result also applies to functions which are smooth on sets
of the form $\Omega_{\bullet t_0} \cup \Omega_{t_0, t_1} \cup
\Omega_{t_1,\bullet}$.  For example, in the next lemma, for a function
$f$ given by $f(w) = g(w-z) - g(w-z')$, where $g$, fulfills the
assumptions of Lemma \ref{s15}.
\end{remark}

\begin{lemma}
\label{s02}
Fix a function $f$ satisfying the assumptions of Lemma \ref{s15}, and
assume that its support is contained in $[0,T_1]\times \bb T$. Let
$g_z : \bb R \times \bb T \to\bb R$ be the function given by
$g_z(w) = f(w-z) - f(w)$, where
$z=(t_0,x_0) \in (0,\infty) \times \bb T$ is such that
$\Vert z\Vert \le 1$. Then, there exists a finite constant $C_0$,
independent of $f$ and $z$, such that
\begin{equation*}
\big| \, X\big (\, g_z  \,\big) \, \big| \; \le\;
C_0 \, (1+T_1) \, \Vert z \Vert^{1+\alpha}\,
\Vert \, f\, \Vert_{C^2(\Omega_{0 \bullet})}
\, \big \Vert X \big \Vert_{C^\alpha([-2,(T_1\vee 1)+2] \times \bb T)}
\end{equation*}
for all $X$ in $C^\alpha$.
\end{lemma}

Note that on the right-hand side we have the norm of $f$ in
$C^2(\Omega_{0 \bullet})$. This is not a misprint. It comes from the
fact that we estimate the $L^\infty$ norm of $(\partial_t f)(w-z) -
(\partial_t f)(w)]$ by $\Vert z\Vert \, \{\, \Vert \, \partial^2_t f\,
\Vert_{L^\infty(\Omega_{0 \bullet})} + \Vert \, \partial^2_{t,x} f\,
\Vert_{L^\infty(\Omega_{0 \bullet})}\,\}$.

\begin{proof}[Proof of Lemma \ref{s02}]
Note that $g_z$ is a continuously-differentiable function on
$\Omega_{0, t_0} \cup \Omega_{t_0 \bullet}$ which vanishes on
$\Omega_{\bullet 0}$. It might be discontinuous at $t=0$ and
$t=t_0$. Using the dyadic partition of the unity, we estimate
separately $X (g_z )$ in the regions $\Omega_{kr, (k+1)r}$, $0\le k
\le 2$ and $\Omega_{t_0 \bullet}$, where $r=t_0/3$.

We start with the first region, $\Omega_{0, r}$. The argument is
similar to the one presented in the proof of Lemma \ref{s15}.  Let
$n_1\in \bb Z$ such that $2^{-n_1} \le t_0 < 2^{-n_1+1}$, and denote
by $A(s,y)$ the function given by
\begin{gather*}
A (s,y) \;=\; \sum_{n\ge n_1} \varphi_n(s)\, g_z(s,y) \;.
\end{gather*}
Note that $A(s,y)=0$ if $s\ge t_0/4$.

Recall the definition of the the function $\psi$ introduced in
\eqref{28}. By \eqref{29}, $X(A)$ can be written as
\begin{equation*}
\sum_{n\ge n_1} X\big(\, \varphi_n \, g_z \,\big)
\;=\; \sum_{n\ge n_1}
\sum_{k=-2^n+1}^{2^n}
X\big(\, \psi_n(y-k2^{-(n+1)})\, \varphi_n(s) \, g_z(s,y) \,\big)
\;.
\end{equation*}
Let $H_n: \bb R \times \bb T \to\bb R$ be given by
$H_n(s,y) = 2^{-2n} g_z(2^{-n} (s,y) + (0,k/2^{n+1})) \, \varphi(s)\,
\psi(y)$
so that
$(S^{2^{-n}}_{(0,k/2^{n+1})} H_n )(s,y) = \psi_n(y-k2^{-(n+1)})\,
\varphi_n(s) \, g_z(s,y)$. In particular, the previous sum is equal to
\begin{equation*}
\sum_{n\ge n_1}  \sum_{k=-2^n+1}^{2^n}
X\big(\, S^{2^{-n}}_{(0,k/2^{n+1})} H_n \,\big)\;.
\end{equation*}
The function $H_n$ belongs to $C^1_c(\bb R \times \bb T)$, it has
support contained in $[1/16,1/4]\times \bb T$, and $\Vert H_n\Vert_{C^1}
\le 2^{-2n} \Vert f\Vert_{C^1(\Omega_{0\bullet})}$. Therefore,
by Lemma \ref{a05}, as $\alpha>-1$,
\begin{align*}
\big| \, X (\, A  \,) \, \big| \; & \le\;
C_0 \,  \sum_{n\ge n_1} 2^{-n(1+\alpha)}
\Vert \, f\, \Vert_{C^1(\Omega_{0\bullet})}
\, \big \Vert X \big \Vert_{C^\alpha([-1,2]\times \bb T)} \\
& \le\;
C_0 \,  2^{-n_1(1+\alpha)}
\Vert \, f\, \Vert_{C^1(\Omega_{0\bullet})}
\, \big \Vert X \big \Vert_{C^\alpha([-1,2]\times \bb T)}
\end{align*}
for some finite constant $C_0$. By definition of $n_1$,
$2^{-n_1(1+\alpha)}\le t_0^{1+\alpha} \le \Vert
z\Vert^{1+\alpha}$. This completes the proof of the first estimate.

We turn to the second one, $\Omega_{t_0/3, 2t_0/3}$, squeezed between
the first and the third regions.  Let
$\Upsilon_1 (s) = \sum_{n\ge n_1} \varphi_n(s)$,
$\Upsilon_2 (s) = \sum_{n\ge n_1} \varphi_n(t_0-s)$. We need to
estimate
$\widehat g_z = g_z [1-\Upsilon_1-\Upsilon_2] = - \, f \,
[1-\Upsilon_1-\Upsilon_2]$.

At the beginning of the proof, we pointed out that the support of
$\Upsilon_1$ is contained in $[0,t_0/4]$. On the other hand, since the
supports of $\varphi_n$ and $\varphi_m$ are disjoints whenever
$|n-m|\ge 2$, $\Upsilon_1(s)=1$ for $0<s\le t_0/32$ because $t_0/32
\le (1/16) 2^{-n_1}$. A similar result holds for $\Upsilon_2$, so that
the support of $\widehat g_z$ is contained in $[a,b]\times \bb T$,
where $a=t_0/32$, $b=(31/32) t_0$. In this set the function $f$ is
$C^1$, which implies that $\widehat g_z$ belongs to $C^1_c(\bb R
\times \bb T)$.

Recall that $2^{-n_1} \le t_0 < 2^{-n_1 +1}$ and write $\widehat g_z$
as
\begin{equation*}
\begin{aligned}
\widehat g_z (s,y)  \; &=\; \sum_{k=-2^{n_1}+1}^{2^{n_1}}
\psi_{n_1}(y-k2^{-({n_1}+1)})\, \, \widehat g_z(s,y) \\
\; & := \; \sum_{k=-2^{n_1}+1}^{2^{n_1}}
(S^{2^{-{n_1}}}_{(t_0/2,k/2^{{n_1}+1})} \widehat H_z)(s,y)\;.
\end{aligned}
\end{equation*}
From the last identity we get that
$\widehat H_z(t,x) = 2^{-2{n_1}} \psi(x) \widehat g_z (t_0/2 +
t/2^{n_1} , k/2^{{n_1}+1} + x/2^{n_1})$.
The support of $\widehat H_z$ is thus contained in
$[-1,1]\times \bb T$ and
$\Vert \widehat H_z \Vert_{C^1_c(\bb R \times \bb T)} \le 2^{-2{n_1}}
\Vert f \Vert_{C^1(\Omega_{0\bullet})}$.
In this later estimate, observe that the time derivative of
$\Upsilon_1$ is of order $t^{-1}_0$. These bounds and Lemma
\ref{a05} yield that
\begin{align*}
\big| \, X (\, \widehat g_z  \,) \, \big| \; & \le \;
\sum_{k=-2^{n_1}+1}^{2^{n_1}}
\big| \, X (\, S^{2^{-{n_1}}}_{(t_0/2,k/2^{{n_1}+1})}
\widehat H_z  \,) \, \big| \\
& \le\;
C_0 \,  2^{-n_1 (1+\alpha)}
\Vert \, f\, \Vert_{C^1(\Omega_{0\bullet})}
\, \big \Vert X \big \Vert_{C^\alpha([-2,2+(t_0/2)] \times \bb
  T)}\;.
\end{align*}
This provides a bound for the second region since $2^{-n_1} \le t_0
\le \Vert z\Vert$.

We estimate $X (\, g_z \,)$ in the third region,
$\Omega_{2t_0/3 ,t_0}$, as in the first one. It remains to consider
the set $\Omega_{t_0 \bullet}$. The result follows from Corollary
\ref{s15c}, Remark \ref{rm6} and the fact that the $L^\infty$ norm of
$g_z (w) = f (w-z) - f(w)$ is bounded by
$\{\, \Vert \partial_t f \Vert_{L^\infty(\Omega_{0\bullet})} +
\Vert \partial_x f \Vert_{L^\infty(\Omega_{0\bullet})}\,\}\, \Vert
z\Vert$. A similar inequality holds for the $L^\infty$ norm
of the first derivatives of $g_z$.
This requires $f$ to be in $C^2$ and provides an estimate of
$\Vert g_z \Vert_{C^1(\Omega_{t_0 \bullet})}$ in terms of
$\Vert z\Vert\, \Vert f \Vert_{C^2(\Omega_{0 \bullet})}$.
\end{proof}

We turn to the proof of the Schauder estimate.  Recall the definition
of the kernel $q_{T_0}$ introduced at the beginning of this section.
We omit the subscript $T_0$ from the notation.  The function $q$ is
smooth in $(0,\infty) \times \bb T$, it diverges at the origin and is
not $C^1$ at $t=0$. For $n\ge 0$, let
$q_{n}: \bb R_+ \times \bb T\to\bb R_+$ be given by
$\color{bblue} q_{n} (z) = \phi_n(z) \, q(z)$, where
$\color{bblue} \phi_n (z) = \varphi_n (|z|)$.

Let $q_{n,z}: \bb R \times \bb T\to\bb R$, $z \in \bb R \times \bb T$,
$n\in \bb N_0$, be given by
\begin{equation*}
q_{n,z} (w) \;=\; q_{n} (z-w) \;=\; \phi_n(z-w) \, q(z-w) \;.
\end{equation*}
The function $q_{n,z}$ fulfills the assumptions of Lemma
\ref{s15}. Hence, by Corollary \ref{s15c}, we may define
$X(q_{n,z})$. The next lemma provides a bound for this quantity.

\begin{lemma}
\label{s14}
There exists a finite constant $C_0$ such that for all
$X\in C^\alpha$, $S<T$, $z = (t,x)$ in $[S,T] \times \bb T$ and
$n\ge 0$,
\begin{equation*}
|\, X(q_{n,z})\, | \;\le\; C_0\, 2^{-(1+\alpha)n} \,
\big\Vert \, X \, \big\Vert_{C^\alpha([S-2,T+2] \times \bb T)}\;.
\end{equation*}
\end{lemma}

\begin{proof}
Let $\delta = 2^{-n} \le 1$, and $g: \bb R \times \bb T \to \bb R$ be
given by $g (w) = q_n(-\delta w)$.  Since $\phi$ is symmetric and
$\delta = 2^{-n}$, $g (w) = \phi (w)\, q(-\delta w)$. In
particular, the support of $g$ is contained in $\bb B(0,1/4)$. As $q$ and
$p$ coincide on $\bb B(0,1/4)$, $g (w) = \phi (w)\, p(-\delta w)$.
Hence,
$g(s,y) =0$ for $s\ge 0$, $g$ satisfies the assumptions of Lemma
\ref{s15}, and, by \eqref{23}, there exists a finite constant $C_0$
such that
\begin{equation*}
\Vert g \Vert_{C^1((-\infty,0) \times \bb T)}\;\le\;
C_0\, \delta^{-1} \;.
\end{equation*}

On the other hand, an elementary computation yields that
$\delta^2\, (S^\delta_z g)(w) = q_{n,z}(w)$. Hence, by Corollary
\ref{s18} and Remark \ref{rm6}, there exists a finite constant $C_0$
such that
\begin{align*}
& \big|\, X\big(\, q_{n,z} \,\big) \,\big| \;=\;  \delta^2\,
\big|\, X\big(\, S^\delta_z g\,\big) \,\big| \;\le\;
C_0\,  \delta^{2+\alpha} \, \Vert g\Vert_{C^1((-\infty,0) \times \bb T)}\,
\Vert X \Vert_{C^\alpha([S-2,T+2] \times \bb T)} \\
&\qquad
\;\le\; C_0 \, 2^{-(1+\alpha)n} \, \Vert X \Vert_{C^\alpha([S-2,T+2]
  \times \bb T)}
\end{align*}
for all $n\ge 0$, $S<T$, $z = (t,x)$ in $[S,T] \times \bb T$.
This completes the proof of the lemma.
\end{proof}

Let
\begin{gather*}
\Phi  \;=\; \sum_{k\ge 0} \phi_k \;, \;\;
\Psi  \;=\; \sum_{k< 0} \phi_k \;\;
\text{so that }\; \Phi  \,+\, \Psi \,=\, 1\;, \\
Q  \;=\; \Phi \, q_{T_0} \;, \quad
R \;=\; \Psi \, q_{T_0} \;.
\end{gather*}
Since the supports of $\varphi_n$ and $\varphi_m$ are disjoints
whenever $|n-m|\ge 2$, $\Psi(w) =0$ for $w\in \bb B(0,1/8)$. Let
$Q_{z}: \bb R \times \bb T \to\bb R_+$ be given by
$\color{bblue} Q_{z} (w) = Q (z-w) = \Phi(z-w) \, q_{T_0}(z-w)$.  The
previous lemma permits to define $X(Q_{z})$ as
\begin{equation}
\label{20}
X(Q_{z}) \;=\; \sum_{k\ge 0} X(q_{k,z})\;.
\end{equation}

\begin{lemma}
\label{s17}
Fix $-1<\alpha<0$, $0<\kappa<1+\alpha$.  Then, there exists a finite
constant $C(\kappa)$ such that
\begin{equation*}
\big\vert \, X(Q_{z}) \,-\, X(Q_{z'}) \, \big\vert
\;\le\; C(\kappa) \,
\big\Vert \, z\,-\, z' \, \big\Vert^{1+\alpha-\kappa}\,
\big\Vert \, X \, \big\Vert_{C^{\alpha}([S-4,T+4]\times  \bb T)}
\end{equation*}
for all $X\in C^\alpha(\bb R \times \bb T)$, $S<T$, $z=(t,x)$,
$z'=(t',x') \in [S,T] \times \bb T$ such that
$\Vert z- z'\Vert\ \le 1/4$.
\end{lemma}

\begin{proof}
In view of \eqref{20}, we have to estimate
$\vert \, X(q_{n,z}) \,-\, X(q_{n,z'}) \, \big\vert$, $n\ge 0$.  Let
$n_1\in \bb Z$ such that
$2^{-(n_1+1)} < \Vert z- z'\Vert \le 2^{-n_1}$.  We first bound this
expression for $n$ large and then we examine the case of $n$ small.

Note that for $n\ge 0$, $q_{n,z} = p_{n,z}$, where
$p_{n,z}(w) = \phi_n(z-w) \, p(z-w)$, because, in this range of $n$,
$q(z-w) = p(z-w)$ if $\phi_n(z-w) \not =0$. Moreover, by \eqref{23},
there exists a finite constant $C_0$ such that
\begin{equation}
\label{48}
\begin{aligned}
& \sup_{w\in \bb R\times \bb T} \big|\,  p_{n,z} (w) - p_{n,z'}(w) \,\big|
\;\le\; C_0\, 2^{2n}\, \Vert z' - z\Vert\, \;, \\
&\quad
\sup_{w\in \bb R\times \bb T} \big|\,  Dp_{n,z} (w) - Dp_{n,z'} (w)\,\big|
\;\le\; C_0\, 2^{3n}\,  \Vert z' - z\Vert
\end{aligned}
\end{equation}
for all $z$, $z'$ in $\bb R\times \bb T$, $n\in\bb N_0$. In this
formula, $D$ stands for $\partial_t$ or $\partial_x$. As $q_{n,z} =
p_{n,z}$, these bounds hold for $q_{n,z}$ and we keep working with $q$.

Fix $n\ge n_1$. By Lemma \ref{s14}, since $t$, $t'\in [S,T]$, there
exists a finite constant $C_0$ such that
\begin{equation}
\label{75}
\begin{aligned}
& \vert \, X(q_{n,z}) \,-\, X(q_{n,z'}) \, \big\vert
\; \le\; \vert \, X(q_{n,z}) \, \big\vert
\,+\, \vert \,  X(q_{n,z'}) \, \big\vert \\
& \qquad \le\; C_0\, 2^{-(1+\alpha)n} \,
\big\Vert \, X \, \big\Vert_{C^\alpha([S-2,T+2] \times \bb T)} \;.
\end{aligned}
\end{equation}

We turn to small $n$'s. Assume that $0\le n<n_1$.  We claim that there
exists a finite constant $C_0$ such that
\begin{equation}
\label{56}
\big|\, X(q_{n,z}) \,-\, X(q_{n,z'}) \,\big| \;\le\; C_0 \,
\Vert z'-z \Vert^{1+\alpha} \,
\Vert X \Vert_{C^\alpha([-4,t+4] \times \bb T)} \;.
\end{equation}

Let $\delta = (1/2)(\Vert z' - z\Vert + 2^{-n})$. Note that
$\delta\le 1$ because $\Vert z' - z\Vert \le 1/4$ and $n\ge 0$.  Let
$g: \bb R \times \bb T \to \bb R$ be given by
$g (w) = \delta^2 [q_{n,z'-z} (\delta w) \,-\, q_{n,0}(\delta w)]$.

Assume, without loss of generality, that $t'>t$, and set
$\Omega = [(-\infty, 0) \cup(0,a)] \times \bb T$, where
$a=(t'-t)/\delta>0$.  The function $g$ is smooth on $\Omega$, it
vanishes on $\Omega_{a \bullet}$ and its time-derivative is
discontinuous at $s=0$ and $s=a$, where $w=(s,y)$.  Moreover, by
\eqref{48}, there exists a finite constant $C_0$ such that
\begin{equation*}
\Vert g \Vert_{L^\infty (\bb R \times \bb T)} \;\le\; C_0\, \delta^2\, 2^{2n}\,
\Vert z' - z\Vert\, \;, \quad
\Vert D g \Vert_{L^\infty (\bb R \times \bb T)}
\;\le\; C_0\, \delta^3\,  2^{3n}\,  \Vert z' - z\Vert\;,
\end{equation*}
where, as above, $D$ stands either for $\partial_t$ or for
$\partial_x$. By definition of $n_1$, since $n<n_1$,
$2^n \, \Vert z' - z\Vert \le 2^{n_1} \, \Vert z' - z\Vert \le 1$ so
that $\delta 2^n \le 1$.  In particular,
$\Vert g \Vert_{C^1(\Omega)} \le C_0 \, \Vert z' - z\Vert$.

On the other hand, by definition of $\varphi$, introduced in
\eqref{26}, the support of $g$ is contained in
$\bb B(\delta^{-1}\, (z' - z) \,,\, \delta^{-1}\, 2^{-(n+1)})
\cup \bb B(0 \,,\, \delta^{-1}\, 2^{-(n+1)}) $. Since $\delta^{-1}\,
\Vert z' - z\Vert \le 2$ and $\delta^{-1}\, 2^{-(n+1)} \le 1$, the
latter set is contained in $\bb B(0 \,,\, 3)$.

An elementary computation yields that
$(S^\delta_z g)(w) = q_{n,z'}(w )\,-\, q_{n,z} (w)$, so that
\begin{equation*}
\big|\, X(q_{n,z}) \,-\, X(q_{n,z'}) \,\big| \;=\;
\big|\, X(S^\delta_z g) \,\big| \;.
\end{equation*}
Since the support of $g$ is contained in $\bb B(0,3)$, $X$ belongs to
$C^\alpha(\bb R \times \bb T)$ and $z$ to $[S,T]\times \bb T$, by
Corollary \ref{s18} and Remark \ref{rm6},
\begin{equation*}
\big|\, X(S^\delta_z g) \,\big|
\;\le\; C_0\, \delta^\alpha \,
\Vert g\Vert_{C^1 (\Omega)}\,
\Vert X \Vert_{C^\alpha([S-4,T+4] \times \bb T)}
\end{equation*}
for some finite constant $C_0$. As
$\Vert g \Vert_{C^1 (\Omega)} \le C_0 \, \Vert z' -
z\Vert$
and $\delta \ge (1/2) \Vert z'-z \Vert$, the previous expression is
less than or equal to
\begin{equation*}
C_0 \,  \Vert z'-z \Vert^{1+\alpha}
\, \Vert X \Vert_{C^\alpha([S-4,T+4] \times \bb T)} \;.
\end{equation*}
This proves \eqref{56}.

We are now in a position to prove the lemma.  By definition,
\begin{equation}
\label{73}
\big|\, X(Q_{z}) \,-\, X(Q_{z'}) \,\big| \;\le\;
\sum_{n\ge 0} \big|\,  X(q_{n,z}) \,-\, X(q_{n,z'}) \,\big| \;.
\end{equation}
By \eqref{75} and \eqref{56}, the right hand side of this expression
is bounded above by
\begin{equation*}
C_0 \,
\Vert z'-z \Vert^{1+\alpha} \,
\Vert X \Vert_{C^\alpha([S-4,T+4] \times \bb T)}
\; n_1
\;+\; C_0 \,
\big\Vert \, X \, \big\Vert_{C^\alpha([S-2,T+2] \times \bb  T)}
\sum_{n\ge n_1}  2^{-(1+\alpha)n} \;.
\end{equation*}
As $2^{n_1} \le \Vert z'-z \Vert^{-1}$, we may estimate $n_1$ by
$C_0\, \ln \Vert z'-z \Vert^{-1}$. Since $0<\kappa<1+\alpha$ and the
map $t \mapsto t^\kappa [1 + \ln t^{-1}]$ is bounded in the interval
$[0,1]$, the previous expression is less than or equal to
\begin{equation*}
C_0 (\kappa)\,
\Vert z'-z \Vert^{1+\alpha-\kappa} \,
\Vert X \Vert_{C^\alpha([S-4,T+4] \times \bb T)}
\;+\; C_0 \,  \Vert z'-z \Vert^{1+\alpha} \,
\big\Vert \, X \, \big\Vert_{C^\alpha([S-2,T+2] \times \bb  T)}
\;.
\end{equation*}
for some finite constant $C_0(\kappa)$. We used here that
$2^{- n_1} \le 2 \Vert z'-z \Vert$ to bound the second term.
This completes the proof of the lemma.
\end{proof}

Recall the definition of $R$ given just above \eqref{20}.  The
function $R$ fulfills the hypotheses of Lemma \ref{s15} and its
support is contained in $[0,T_0] \times \bb T$.  Let
$R_{z}: \bb R \times \bb T \to\bb R_+$ be given by
$\color{bblue} R_{z} (w) = R (z-w) = \Psi(z-w) \, q_{T_0}(z-w)$.  The
next lemma provides estimates for $X(R_{z})$.

\begin{lemma}
\label{s01}
Fix $-1<\alpha<0$. Then, there exists a finite constant
$C_0 = C_0(T_0)$ such that
\begin{equation*}
\big\vert \, X(R_{z}) \,-\, X(R_{z'}) \, \big\vert
\;\le\; C_0 (T_0)\,
\big\Vert \, z\,-\, z' \, \big\Vert^{1+\alpha} \,
\big\Vert \, X \, \big\Vert_{C^{\alpha}([S-T_0-3,T++2]\times  \bb T)}
\end{equation*}
for all $X\in C^\alpha(\bb R \times \bb T)$, $S<T$, $z=(t,x)$,
$z'=(t',x') \in [S,T] \times \bb T$ such that
$\Vert z- z'\Vert\ \le 1$.
\end{lemma}

\begin{proof}
Fix $z=(t,x)$, $z'=(t',x')$ such that $\Vert z'-z\Vert \le 1$.
Without loss of generality, assume that $t<t'$ and let
$f(w) = R_{z'}(w)$. The function $f$ satisfies the assumptions of
Lemma \ref{s15}, with the plane $\{(s,y) : s=t'\}$ replacing
$\{(s,y) : s=0\}$. Its support in contained in $[t'-T_0-1, t']$.
Clearly, $R_{z}(w) - R_{z'}(w) = f(w - \hat z) - f(w)$, where
$\hat z = z-z'$. By Lemma \ref{s02} applied to the function $f$,
\begin{equation*}
\big| \, X(R_{z}) - X(R_{z'}) \, \big| \; \le\;
C_0 \, (1+T_0) \, \Vert \hat z \Vert^{1+\alpha}\,
\Vert \, f\, \Vert_{C^2(\Omega_{\bullet t'})}
\, \big \Vert X \big \Vert_{C^\alpha([t'-T_0-3,t'+2] \times \bb T)}\;.
\end{equation*}
By definition of $f$ and $R$, and since $\Psi(w) = 0$ for
$w\in \bb B(0,1/8)$, there exists a finite constant $C_0$ such that
\begin{equation*}
\Vert \, f\, \Vert_{C^2(\Omega_{\bullet t'})} \;=\;
\Vert \, R\, \Vert_{C^2(\Omega_{0 \bullet})}  \;\le\;
C_0 \Vert \, p\, \Vert_{C^2(\Omega_\star)} \;,
\end{equation*}
where $\Omega_\star = \Omega_{0,T_0+1} \setminus \bb B(0,1/8)$.
To complete the proof, it remains to recall that $z' \in [S,T]\times
\bb T$.
\end{proof}

\begin{remark}
\label{rm7}
Note that the proofs of Lemmata \ref{s15}, \ref{s02}, \ref{s14} and
\ref{s17} permit to extend the domain of definition of a distribution
$X$ in $C^\alpha$ to functions which do not belong to
$C^1(\bb R \times \bb T)$. For instance to functions of the type
$g(t,x)=f(t,x) \chi_{(S,T)}(t)$, where $f$ belongs to
$C^1(\bb R \times \bb T)$ and $\chi_A$ is the indicator of the set
$A$. This property is further exploited below in the definition of the
distribution $X^+$.
\end{remark}

\begin{proof}[Proof of Theorem \ref{t01}]
The proof is a consequence of Lemmata \ref{s17} and \ref{s01}.
\end{proof}

\begin{remark}
\label{rm1}
The proof of Lemma \ref{s17} can be extended to $\alpha<-1$, but this
result will not be needed here.
\end{remark}

\subsection{The distribution $X^+$}

Let $f$ be a function in $C^1_c(\bb R \times \bb T)$ and denote by
$\color{bblue} \chi_A$, $A\subset \bb R \times \bb T$, the indicator
function of the set $A$. Lemma \ref{s15} permits also to define $X(f
\chi_{\bb R_+ \times \bb T})$ as a sum. We denote this quantity by
$X^+(f):$
\begin{equation*}
 X^+(f) \;:=\; X \big(\, f \, \chi_{\bb R_+ \times \bb T} \,\big)\;.
\end{equation*}
Clearly, $X^+(f)=0$ for all $f$ in $C^1_c(\bb R\times \bb T)$ whose
support is contained in $(-\infty,0] \times \bb T$.

Fix $-1<\alpha<0$, $0<\kappa<1+\alpha$. We claim that there exists a
finite constant $C(\kappa, T_0)$ such that
\begin{equation}
\label{80}
\big\vert \, X^+(q_{T_0,(t,x)}) \, \big\vert
\;\le\; C(\kappa,T_0) \, t^{1+\alpha-\kappa}\,
\big\Vert \, X \, \big\Vert_{C^{\alpha}([-T_0-3,5]\times  \bb T)}
\end{equation}
for all $X\in C^\alpha(\bb R \times \bb T)$, $0\le t \le 1/4$.

To prove this claim, fix $z=(t,x)$ and write
$X^+(q_{T_0,(t,x)}) = X\big(\, q_{T_0,(t,x)} \, \chi_{\bb R_+ \times \bb
  T}  \,\big)$ as
\begin{equation}
\label{85}
X\big(\,  q_{T_0,(t,x)} \,-\, q_{T_0,(0,x)} \,\big) \;+\;
X\big(\,  q_{T_0,(0,x)} \,-\, q_{T_0,(t,x)}  \,[ 1- \chi_{\bb R_+
  \times \bb T}  ]  \,\big) \;.
\end{equation}
By Theorem \ref{t01} with $S=0$, $T=1$, the absolute value of the
first term is bounded by
$C_0 \, t^{1+\alpha-\kappa}\, \big\Vert \, X \,
\big\Vert_{C^{\alpha}([-T_0-3,5]\times \bb T)}$ for some finite
constant $C_0 = C_0(\kappa, T_0)$.  On the other hand, since the
support of $q_{T_0,(0,x)} $ is contained in
$(-\infty, 0] \times \bb T$, the function appearing in the second term
can be written as
$ [\, q_{T_0,(0,x)} \,-\, q_{T_0,(t,x)} \,] \,[ 1- \chi_{\bb R_+
  \times \bb T} ] $. In Lemmata \ref{s17} and \ref{s01}, we estimated
$X(q_{T_0,(0,x)} \,-\, q_{T_0,(t,x)} )$ in each region separately. In
particular, it follows from these results that the absolute value of
the second term in \eqref{85} is bounded by
$C_0 \, t^{1+\alpha-\kappa}\, \big\Vert \, X \,
\big\Vert_{C^{\alpha}([-T_0-3,5]\times \bb T)}$ for some finite
constant $C_0 = C_0(\kappa, T_0)$. This proves \eqref{80}.

For similar reasons, the proofs of Lemmata \ref{s17} and \ref{s01}
yield that for fixed $-1<\alpha<0$, $0<\kappa<1+\alpha$, there exists
a finite constant $C(\kappa, T_0)$ such that
\begin{equation}
\label{86}
\big\vert \, X^+(q_{T_0,z}) \, -\,  X^+(q_{T_0,z'})  \,\big\vert
\;\le\; C(\kappa,T_0) \, \Vert z- z'\Vert^ {1+\alpha-\kappa}\,
\big\Vert \, X \, \big\Vert_{C^{\alpha}([-T_0-3,5]\times  \bb T)}
\end{equation}
for all $X\in C^\alpha(\bb R \times \bb T)$, $z$,
$z' \in [0,1] \times \bb T$ such that $\Vert z- z'\Vert\ \le 1/4$.

\begin{corollary}
\label{s04}
Fix $-1<\alpha<0$, $0<2\kappa<1+\alpha$.  Then, there exists a finite
constant $C(\kappa, T_0)$ such that
\begin{equation*}
\Vert u \Vert_{C^{1+\alpha-2\kappa}([0,T] \times \bb T)} \;\le\;
C(\kappa,T_0) \, T^\kappa\, 
\big\Vert \, X \, \big\Vert_{C^{\alpha}([-T_0-3,5]\times  \bb T)}
\;.
\end{equation*}
for all $X\in C^\alpha(\bb R \times \bb T)$, $0<T\le 1/4$, where
$u = u_{T_0} : \bb R \times \bb T \to \bb R$ is given by
$u(z) = X^+(q_{T_0,z})$.
\end{corollary}

\begin{proof}
Fix $X\in C^\alpha(\bb R \times \bb T)$, $0<T\le 1/4$, $z=(t,x)$,
$z'=(t',x') \in [0,T]\times \bb T$.
Suppose first that $|x-x'|\le T$. In this case, 
$\Vert \, z\,-\, z' \, \Vert^{\kappa} \le (2T)^\kappa$. Hence, by
\eqref{86},
\begin{equation*}
\big\vert \, X^+(q_{T_0,z}) \,-\, X^+(q_{T_0,z'}) \, \big\vert
\;\le\; C(\kappa,T_0) \,
\big\Vert \, z\,-\, z' \, \big\Vert^{1+\alpha-2\kappa}\, T^\kappa
\big\Vert \, X \, \big\Vert_{C^{\alpha}([-T_0-3,5]\times  \bb T)}
\;.
\end{equation*}

In contrast, if $|x-x'|> T$, $t^{1+\alpha-\kappa} \le 
T^{1+\alpha-2\kappa} \, T^\kappa \le |x-x'|^{1+\alpha-2\kappa}\, T^\kappa 
\le \Vert z -z'\Vert^{1+\alpha-2\kappa}\, T^\kappa $. A similar
inequality holds with $t'$ in place of $t$. Hence, by \eqref{80}, 
\begin{equation*}
\begin{aligned}
& \big\vert \, X^+(q_{T_0,z}) \,-\, X^+(q_{T_0,z'}) \, \big\vert
\;\le\; \big\vert \, X^+(q_{T_0,z}) \, \big\vert
\;+\; \big\vert \, X^+(q_{T_0,z'}) \, \big\vert
\\
&\quad \le\; C(\kappa, T_0) 
\big\Vert \, z\,-\, z' \, \big\Vert^{1+\alpha-2\kappa}\, T^\kappa
\big\Vert \, X \, \big\Vert_{C^{\alpha}([-T_0-3,5]\times  \bb T)} \;.
\end{aligned}
\end{equation*}
The assertion follows from the two previous estimates.
\end{proof}

\section{A log-correlated Gaussian random field}
\label{sec07}


We introduce in this section a Gaussian random field closely related
to the linear SPDE 
\begin{equation}
\label{03}
\left\{
\begin{aligned}
& \partial_t \mf f \,=\,
-\, (-\Delta)^{1/2}\, \mf f \,+\, \xi\;, \\
& \mf f (0,x) \;=\; f(x)  \;,
\end{aligned}
\right.
\end{equation}
where $\xi$ represents the space-time white noise and
$f: \bb T \to \bb R$ a continuous function.  Denote by
$\mf f_\varepsilon$ the solution of the previous equation with $\xi$
replaced by its regularized version $\xi_\varepsilon$ introduced in
\eqref{14}.

The solution of these equations can be expressed in terms of the
semigroup $p(t,x)$ introduced in \eqref{04}:
\begin{equation}
\label{76}
\mf f_\varepsilon (t,x) \;=\;
\int_0^t ds\, \int_{\bb T} dy\; p (t-s, x-y)\,
\xi_{\varepsilon}(s,y)  \;+\; \int_{\bb T} p (t, x-y)\,
f (y) \; dy
\end{equation}
because $p(t,\cdot)$ is symmetric.  As $p(t,x) =0$ for $t<0$, in the
first term we may change the interval of integration from $[0,t]$ to
$[0,\infty)$. Moreover, if we set the initial condition $f$ to be
$f(x) = \int_{-\infty}^0 ds \int_{\bb T} p(-s, x-y)
\xi_\varepsilon(s,y)\, dy$, as $p$ is a semigroup,
\begin{equation*}
\mf f_\varepsilon (t,x) \;=\; 
\int_{-\infty}^\infty ds\, \int_{\bb T} dy\, p (t-s, x-y)\,
\xi_{\varepsilon}(s,y) \;.
\end{equation*}

We replace, in the previous convolution, $p$ by a kernel $q$ with
bounded support to avoid problems of integrability at infinity.  
Fix $T_0>1$ and recall the definition of the function
$q_{T_0}: \bb R \times \bb T \to \bb R_+$ introduced in
\eqref{06}. Note that
\begin{equation}
\label{49}
\text{$q_{T_0}(t, x) \;=\; p(t,x)$ for $t\le T_0$ and that $q_{T_0}(s,
  \cdot)$ is symmetric} 
\end{equation}
for all $s\in \bb R$, $q_{T_0}(s,x) = q_{T_0}(s,-x)$, because so is
$p(s,\cdot)$. As $T_0$ is fixed, we omit it from the notation.

Let $\mf v := q * \xi$, $\mf v_\varepsilon := q * \xi_\varepsilon$ be
the centered Gaussian random fields on $\bb R \times \bb T$ defined
by:
\begin{equation}
\label{08}
\begin{gathered}
\mf v (t,x) \;:=\;
\int_{\bb R} ds\, \int_{\bb T} dy\; q(t-s, x-y)\,
\xi (s,y)\;, \\
\mf v_\varepsilon (t,x)
\;:=\;
\int_{\bb R} ds\, \int_{\bb T} dy\; q(t-s, x-y)\,
\xi_{\varepsilon}(s,y)\;.
\end{gathered}
\end{equation}

Let $\color{bblue} r=q-p$, and note that $r$ is smooth and vanishes in
the time-interval $(-\infty, T_0]$.  As $q$ coincides with $p$ on the
time-interval $(-\infty, T_0]$ and vanishes outside the interval
$[0,T_0+1]$, for $0<t< T_0$, the field $\mf v_\varepsilon(t)$ can be
rewritten as
\begin{equation}
\label{77}
\mf v_\varepsilon (t) \;=\;
\int_{-(T_0+1)}^{T_0} r(t-s)\, \xi_{\varepsilon}(s) \; ds 
\;+\; \mf g_\varepsilon (t)
\end{equation}
where 
\begin{equation*}
\mf g_\varepsilon (t) \; =\;
\int_0^t p(t-s)\, \xi_{\varepsilon}(s) \; ds \;+\;
p(t) \int_{-(T_0+1)}^0  p(-s)\,\xi_{\varepsilon}(s) \; ds \;.   
\end{equation*}
We used in the last step the fact that $p$ is a semigroup so that
$p(t-s) = p(t)p(-s)$ for all $s<0<t$.

Denote by $\mf G_\varepsilon (t)$ the first term on the right-hand
side of \eqref{77}. Since $r$ is a smooth function, by \cite[Theorem
1.4.2]{AdlTay07}, almost surely, the field $\mf G_\varepsilon$ is
$C^\infty$ in the set $\bb R \times \bb T$. On the other hand, by
\eqref{76}, the second term on the right-hand side of \eqref{77},
represented by $\mf g_\varepsilon$, solves \eqref{03} with initial
condition $\mf g_\varepsilon (0) = \int_{[-(T_0+1), 0]}
p(-s)\,\xi_{\varepsilon}(s) \; ds$. Therefore, for $0<t<T_0$,
\begin{equation}
\label{78}
\partial_t \mf v_\varepsilon \,+\, (-\Delta)^{1/2}\, 
\mf v_\varepsilon \,-\, \xi_\varepsilon \;=\;
\partial_t \mf G_\varepsilon \,+\, (-\Delta)^{1/2}\, 
\mf G_\varepsilon \;,
\end{equation}
A similar conclusion holds with $\mf v_\varepsilon$,
$\xi_\varepsilon$, $\mf G_\varepsilon$ replaced by $\mf v$, $\xi$,
$\color{bblue} \mf G =\int_{[-(T_0+1), T_0]} r(t-s)\, \xi(s) \; ds$,
respectively.

\subsection{The correlations of the fields $\mf v_\varepsilon$,
  $\mf v$}

Denote by $Q$, $Q_\varepsilon$ the covariances of the fields $\mf
v$, $\mf v_\varepsilon$, respectively: For $z$, $z'$ in $\bb R \times
\bb T$,
\begin{equation*}
Q (z,z^\prime) \;=\;
\bb E \big[\, \mf v (z)\, \mf v (z')\,\big] \;, \quad
Q_\varepsilon (z,z^\prime) \;=\;
\bb E \big[\, \mf v_\varepsilon (z)\, \mf v_\varepsilon (z')\,\big]
\;.
\end{equation*}
A change of variables yields that
$Q_\varepsilon (z,z^\prime) = Q_\varepsilon (0,z^\prime-z)$. Denote
this later quantity by $Q_\varepsilon (z^\prime-z)$ and define
$Q(\cdot)$ similarly.  

The main result of this section reads as follows.  Recall that $\ln^+
t = \ln t$ if $t\ge 1$ and $\ln^+ t = 0$ if $0<t\le 1$.

\begin{proposition}
\label{s22}
There exist a continuous, bounded function
$R : \bb R \times \bb T \to \bb R$, such that
\begin{equation*}
Q (z) \; = \;
\frac{1}{2\pi} \ln^+ \frac 1{4 \, \Vert z\Vert} \;+\;  \; R(z) \;.
\end{equation*}
\end{proposition}

The proof of Proposition \ref{s22} relies on some lemmata.  An
elementary computation yields that
\begin{equation}
\label{10b}
\begin{gathered}
Q(z) \; =\;
\int_{\bb R\times \bb T}  q(-w)  \, q(z-w) \; dw
\; =\;  \int_{\bb R\times \bb T} q(w) \, q(z+w) \; dw \;,
\\
Q_\varepsilon(z) \;=\;
\int_{\bb R\times \bb T}  (q * \varrho_\varepsilon)  (w)  \,
(q * \varrho_\varepsilon) (z+w) \; dw \;.
\end{gathered}
\end{equation}

A change of variables shows that $Q(-z) = Q(z)$. On the other hand, as
$q(t,-x)= q(t,x)$, it is not difficult to see that $Q(t,-x) = Q(t,x)$
for all $(t,x) \in \bb R \times \bb T$. Moreover, as the support of
$q$ is contained in $[0,T_0+1]\times \bb T$ and $T_0\ge 1$, the one of
$Q$ is contained in $[-2T_0,2T_0]\times \bb T$.

Let $\color{bblue} \mc P = \{(0,x) : x\in \bb T\}$.  It is not
difficult to show that $Q$ is smooth in
$(\bb R \times \bb T) \setminus \{(0,0)\}$, that $(-\Delta)^{1/2} Q$
is well defined in the set $(\bb R \times \bb T) \setminus \mc P$, and
that
\begin{equation}
\label{16}
[(-\Delta)^{1/2} Q] (z)
\; =\; \int_{\bb R \times \bb T}  q(w) \;
[(-\Delta)^{1/2} q]\, (w + z) \; dw \;, \quad
z\,\in\, (\bb R \times \bb T) \setminus \mc P \;.
\end{equation}

Finally, by the definition \eqref{01} of the operator
$(-\Delta)^{1/2}$, as $Q(-z) = Q(z)$, a change of variables yields
that $[(-\Delta)^{1/2} Q] (-z) = [(-\Delta)^{1/2} Q] (z)$. We
summarize these properties in the next formula:
\begin{equation}
\label{24}
\begin{gathered}
Q(-z) \;=\;  Q(z) \;, \;\; Q(t,-x) \;=\; Q(t,x) \;, \\
[(-\Delta)^{1/2} Q]\, (-z) = [(-\Delta)^{1/2} Q] \, (z)
\;, \quad
z\,\in\, (\bb R \times \bb T) \setminus \mc P\;.
\end{gathered}
\end{equation}
These identities extend to $Q_\varepsilon$. Moreover,
\begin{equation}
\label{25}
Q_\varepsilon (z) \;=\; (Q * \bar{\varrho}_\varepsilon) (z)\;,
\end{equation}
where $\bar{\varrho}_\varepsilon$ is the mollifier given by
$\color{bblue} \bar{\varrho}_\varepsilon := \varrho_\varepsilon * (\mc
T \varrho_\varepsilon)$, and $\mc T$ is the operator defined by
$\color{bblue} (\mc T f) (z):= f(-z)$. We define
$\bar{\varrho}$ similarly.

\begin{lemma}
\label{s20}
There exists a continuous, bounded function $R_0: \bb R\times \bb T \to
\bb R$ such that
\begin{equation*}
[(-\Delta)^{1/2} Q] (z) \;=\; (1/2)\,
\big[\, q(-z) \,+\, q(z) \,\big]  \;+\; R_0(z) \;,
\quad z\,\in\, (\bb R \times \bb T) \setminus \mc P \;.
\end{equation*}
\end{lemma}

\begin{proof}
By the properties \eqref{49} of the kernel $q$,
\begin{equation}
\label{11}
\left\{
\begin{aligned}
& \partial_t q \;+\; (-\Delta)^{1/2}\, q \;=\; R \;, \quad t>0\;, \\
& q(0,\cdot) \;=\; \delta_0(\cdot)\;, \quad t\,=\,0\;, \\
& q(t,\cdot) \;=\; 0 \;, \quad t\,<\,0\;,
\end{aligned}
\right.
\end{equation}
where $\delta_0$ is the Dirac distribution concentrated at $x=0$ and
$R: \bb R \times \bb T \to \bb R$, given by $R = [\, \partial_t \,+\,
(-\Delta)^{1/2}\,] (q-p)$, is a smooth function with compact
support. Note that $R(s,y)=0$ if $s\le 0$.

Fix $z =(t,x)\in \bb R \times \bb T$ and assume that $t<0$.  By
\eqref{16},
\begin{equation*}
[(-\Delta)^{1/2} Q] (z)
\; =\; \int_{\bb R \times \bb T}  q(w) \;
[(-\Delta)^{1/2} q]\, (w + z) \; dw \;.
\end{equation*}
Since $q(s,y)$ vanishes for $s<0$, we may restrict the integration to
the time-interval $(-t, +\infty)$. By \eqref{11}, on $(0, +\infty)
\times \bb T$, $[(-\Delta)^{1/2} q] (s,y) = R(s,y) - (\partial_s
q)(s,y)$. The previous integral is thus equal to
\begin{equation*}
\int_{-t}^\infty ds \int_{\bb T}  dy\; q(w) \; R (w + z)
\;-\;  \int_{-t}^\infty ds \int_{\bb T}  dy\; q(w) \;
(\partial_s q)\, (w + z) \;,
\end{equation*}
where $w=(s,y)$ As $R(s,y)$ vanishes for $s\le 0$, in the first
integral we may integrate over $\bb R \times \bb T$. After an
integration by parts, as $q(0,\cdot) = \delta_0(\cdot)$, the second
integral becomes
\begin{equation*}
q(-z) \;+\;  \int_{-t}^\infty ds \int_{\bb T}  dy\, (\partial_s q)\, (w) \;
q\, (w + z)\;.
\end{equation*}
Using again the first identity in \eqref{11}, we can write this
expression as
\begin{equation*}
q(-z) \;+\;  \int_{(-t,\infty) \times \bb T}  R\, (w) \;
q\, (w + z) \; dw  \;-\;
\int_{(-t,\infty)\times\bb T}
[(-\Delta)^{1/2} q]\, (w) \; q\, (w + z) \; dw \;.
\end{equation*}
Since $(-\Delta)^{1/2}$ is a symmetric operator, $q(s,y)$ is smooth
away from $\mc P$ and vanishes for $s<0$, the last integral is equal
to
\begin{align*}
& -\;\int_{(-t,\infty) \times\bb T} q\, (w) \;
[(-\Delta)^{1/2} q]\, (w + z) \; dw \; \\
&\quad  =\; -\;\int_{\bb R \times\bb T} q\, (w) \;
[(-\Delta)^{1/2} q]\, (w + z) \; dw \;=\;
-\, [(-\Delta)^{1/2} Q]\, (z)\;.
\end{align*}

Putting together the previous terms yields that for $z=(t,x)$ with
$t<0$,
\begin{equation*}
    2\, [(-\Delta)^{1/2} Q] (z) \;=\; q(-z) \;+\;
    \int_{\bb R\times \bb T} q\, (w) \, \big[\,  R\, (w+z) \, +\,
    R\, (w - z) \,\big] \; dw\;.
\end{equation*}
Since $q(z)=0$, we may add $q(z)$ to the right-hand side to complete
the proof of the lemma in the case $t<0$ with
\begin{equation*}
R_0(z) \;=\; \frac 12
\int_{\bb R\times \bb T} q\, (w) \, \big[\,  R\, (w+z) \, +\,
R\, (w - z) \,\big] \; dw\;.
\end{equation*}
The proof in the case $t>0$ is analogous.

The function $R_0$ is continuous because $R$ is uniformly continuous
and $q$ is integrable. It is bounded because $R$ is uniformly bounded
and $q$ is integrable.
\end{proof}

Recall from \eqref{12} that we denote by $G$ the Green function
associated to $(-\Delta)^{1/2}$. Let
$\color{bblue} q_s(z) = (1/2) (q(-z)+q(z))$.

\begin{lemma}
\label{s21}
There exists a continuous, bounded function $R_1: \bb R\times \bb T
\to \bb R$ such that for all $(t,x)\in \bb R \times \bb T$, $(t,x) \,
\not =\, (0,0)$,
\begin{equation*}
Q (t,x) \;=\;  [\, q_s (t, \cdot) * G\,] (x) \;+\; R_1(t,x) \;.
\end{equation*}
\end{lemma}

\begin{proof}
Recall from \eqref{10b} the expression of the function $Q$. An
elementary computation yields that for each $t\in \bb R$,
\begin{equation*}
Q(t) \;:=\; \int_{\bb T} Q(t,x) \, dx  \;=\;
\int_{\bb R} q(s) \, q(t+s) \, ds \;,
\end{equation*}
where $q(t) = \int_{\bb T} q(t,y) \, dy$.  By definition of $q(s,y)$,
the function $q(t)$ is bounded, integrable and discontinuous only at
the origin. Moreover, it vanishes outside a compact set, and it is
equal to $0$, resp. $1$, for $s<0$, resp. $0\le s \le T_0$. In
particular, $Q$ is bounded and continuous.

Let $\bar Q(t,x) = Q(t,x) - Q(t)$.  Clearly, for all
$(t,x) \in (\bb R \times \bb T)\setminus \mc P$,
$[(-\Delta)^{1/2} \bar Q(t,\cdot)](x) = [(-\Delta)^{1/2}
Q(t,\cdot)](x)$. Hence, by the previous lemma,
\begin{equation*}
[(-\Delta)^{1/2} \bar Q(t,\cdot)] \, (x) \;=\; q_s (t,x)
\;+\; R_0(t,x) \;, \quad (t,x) \;\in\;
(\bb R \times \bb T)\setminus \mc P\;.
\end{equation*}
Since, for all $t' \in \bb R$, $\int_{\bb T} \bar Q(t',y) \, dy = 0$,
by \eqref{27}, taking the convolution with respect to $G$ on both
sides of the previous equation yields that
\begin{equation*}
Q (t,x) \;=\; [\, q_s (t, \cdot) * G\,] (x) \;+\; R_1(t,x)
\quad (t,x) \;\in\;
(\bb R \times \bb T)\setminus \mc P\;,
\end{equation*}
with $R_1 (t,x) = [R_0 (t, \cdot) * G] (x) + Q(t)$.

Since $R_0$ is bounded and continuous, and $G$ is integrable,
$[R_0 (t, \cdot) * G] (x)$ is bounded and continuous. As we already
showed that $Q(t)$ is bounded and continuous, the proof is complete
for $(t,x) \,\in\, (\bb R \times \bb T)\setminus \mc P$. Since all
terms are continuous on $(\bb R \times \bb T)\setminus \{(0,0)\}$,
this identity can be extended to $(t,x)\not = (0,0)$.
\end{proof}

We are now in a position to prove Proposition \ref{s22}.

\begin{proof}[Proof of Proposition \ref{s22}]
In view of Lemma \ref{s21}, we need to show that
\begin{equation*}
R_3(t,x) \;:=\; [\, q_s (t, \cdot) * G\,] (x)
\;-\; \frac{1}{2\pi} \ln^+ \frac 1{4 \, \Vert z\Vert}
\end{equation*}
is a continuous, bounded function.

Recall that $q_s (t, x) = (1/2) \, [\, q (t, x) + q(-t,-x)\,]$. Since
$q(s,y)=0$ for $s<0$ and since $q(s, \cdot)$ is symmetric,
$q_s (t, x) = (1/2)\, q (|t|, x)$ for $t\not =0$. By the explicit
expression \eqref{12} of the Green function,
\begin{equation*}
R_3(t,x) \;=\; -\, \frac 1{2 \pi} \, \int_{\bb T} q(|t|, y)\,
\ln \Big\{\, 2 \, \big|\, \sin (\pi |x-y] )\, \big|\, \Big\} \, dy
\;-\; \frac{1}{2\pi} \ln^+ \frac 1{4 \, \Vert z\Vert} \;.
\end{equation*}
By definition of $q$, for any $\delta>0$, $R_3$ is bounded and
continuous on $\bb B(0,\delta)^c$ because the function $\ln t$ is
integrable on the interval $(0,1)$.

We turn to the behavior of $R_3$ on $\bb B(0,\delta)$. Fix $0<\delta <
1/32$ and recall that $T_0>1$.  On the set $\bb B(0,\delta)$, $\ln^+ (\,4
\, \Vert z\Vert\,)^{-1} = \ln (\, 4 \, \Vert z\Vert\,)^{-1}$, and the
right-hand side of the previous equation can be written as
\begin{equation*}
-\, q(|t|) \, \frac 1{2 \, \pi} \, \ln\, 2  \;-\;
\frac 1{2 \, \pi} \, \int_{\bb T} q(|t| , y)\,
\ln \Big\{ \, \big|\, \sin (\pi [x-y])\, \big|\, \Big\}  \, dy \;-\;
\frac{1}{2\pi} \ln \frac 1{4 \, \Vert z\Vert}
\end{equation*}
where $t\mapsto q(|t|)$ is the continuous and bounded function defined
in the previous proof as $q(t) := \int_{\bb T} q(t, y)\, dy$.

On the set $\{(s,y): |s|\le \delta\}$, $q(s, y) = p(s, y)$. We may,
therefore, replace $q$ by $p$ in the previous integral.  At this
point, it remains to show that the function $R_4$ given by
\begin{equation*}
R_4(t,x) \;=\; \int_{\bb T} p(|t|, y)\,
\ln \, \big|\, \sin (\pi [x-y] )\, \big|\, \, dy \;-\;
\ln \Vert z\Vert
\end{equation*}
is bounded and continuous in $\bb B(0,\delta)$.

We prove the boundedness, the continuity being similar. Assume that
$t>0$. Let $F:\bb R\to \bb R$ be the one-periodic function given by
$\color{bblue} F(x) = \ln | \, \sin (\pi x)\, |$. Note that $F$ is
symmetric: $F(-x) = F(x)$.  We claim that there exists a finite
constant $C_0$, depending solely on $\delta$, such that
\begin{equation}
\label{31}
\Big|\, \int_{\bb T} p(t, y)\, F(x-y)  \, dy
\;-\; \frac 1\pi\,
\int_{-\infty}^{\infty}    \frac {\ln \, |\, ty\, |}
{1 + [y +  (x/t)]^2} \; dy  \, \Big| \;\le\; C_0\;.
\end{equation}
Hereafter, without further notice, the value of constants $C_0$'s may
change from line to line.

The proof of \eqref{31} is divided in several steps. Recall the
definition \eqref{07} of $\bs p(t,x)$. By \eqref{04} and a change of
variables, the first term in \eqref{31} is equal to
\begin{equation*}
\sum_{k\in\bb Z} \int_{k-1/2}^{k+1/2} \bs p(t, y)\,
F(x-y)  \, dy \;=\;  \int_{\bb R} \bs p(t, y)\,
F(x-y)  \, dy
\end{equation*}
because $F$ is periodic.

Fix $4\delta <a<1/8$. We claim that there exists a constant $C_0$,
depending only on $\delta$ and $a$, such that
\begin{equation}
\label{18}
\int_{[x-a,x+a]^c} \bs p(t, y)\, |\, F(y-x)\,|  \; dy \;\le\; C_0\;.
\end{equation}
Let $A = \cup_{k\in\bb Z} [k-a,k+a]$. The function $F$ is uniformly
bounded on the complement of $A$: There exists a constant
$C_0=C_0(a)$, such that $|\, F(y)\,| \le C_0$ for all $y\in A^c$. Let
$x+A = \{x+y : y\in A\}$. Since $|\, F(y-x)\,| \le C_0$ for all
$y\in [x+A]^c$, and $\bs p(t, \cdot)$ is a probability density,
\begin{equation*}
\int_{[x+A]^c} \bs p(t, y)\, |\, F(y-x)\,|  \; dy \;\le\; C_0\;.
\end{equation*}
On the other hand, by the explicit form of the density $\bs p(t, y)$,
and a change of variable,
\begin{equation*}
\sum_{k\ge 1} \int_{k+x-a}^{k+x+a} \bs p(t, y)\, |\, F(y-x)\,|  \; dy
\;\le\;
\frac 1\pi\, \sum_{k\ge 1} \frac {t}{t^2 + [k+x-a]^2}
\int_{-a}^{a} |\, F(y)\,|  \; dy \;.
\end{equation*}
Since $F$ is integrable in a neighborhood of the origin, and since
$|x| + a \le 1/4$, the previous sum is bounded by
$C_0\, t\, \sum_{k\ge 1} [k-(1/4)]^{-2} \le C_0\, t$. A similar bound
can be derived for the sum $k\le -1$. This proves \eqref{18}.

By the explicit form of $\bs p(t, \cdot)$ and a change of variables,
\begin{equation*}
\int_{x-a}^{x+a} \bs p(t, y)\, F(y-x)  \; dy \;=\;
\frac 1\pi\, \int_{-a/t}^{a/t} \frac {F(ty)}{1 + [y + (x/t)]^2}
\; dy \;.
\end{equation*}
There exists a finite constant $C_0=C_0(a)$ such that
$| \, \sin (\pi y)\, /y |\,\le\, C_0$ for all $y$ in the interval
$[-a,a]$. Hence, since $\bs p(t, \cdot )$ is a probability density,
\begin{equation*}
\Big|\, \int_{-a/t}^{a/t} \frac {F(ty)}{1 + [y + (x/t)]^2}
\; dy \;-\;
\int_{-a/t}^{a/t} \frac {\ln \, |\, ty \,|}{1 + [y + (x/t)]^2}
\; dy \,\Big|\;\le\; C_0\;.
\end{equation*}

We claim that there exists a finite constant $C_0$ such that
\begin{equation}
\label{32}
\int_{a/t}^{\infty} \frac
{\big|\, \ln \, (\, ty \,) \, \big|}
{1 + [y + (x/t)]^2} \;    \; dy
\;\le\; C_0 \;,
\end{equation}
with a similar bound if the domain of integration is replaced by
$(-\infty, -a/t]$. By a change of variables, this integral is bounded
by
\begin{equation*}
C_0 \, \int_{(a+x)/t}^{2/t} \frac {1}{y^2} \;  dy \;+\;
\int_{2/t}^{\infty} \frac {\ln t + \ln y}{y^2} \;  dy
\;\le\; C_0\, \{t \,\ln t + t^{1/2}\}
\end{equation*}
because $\ln y \le C_0 y^{1/2}$ for $y\ge 1$. This proves
\eqref{32}. Putting together all previous estimates yields
\eqref{31}.

It remains to show that the absolute value of
\begin{equation*}
\frac 1\pi\, \int_{-\infty}^{\infty} \frac {\ln |\, ty\,|}{1 + [y + (x/t)]^2}
\; dy \;-\; \ln \Vert z\Vert
\end{equation*}
is bounded on $\bb B(0,\delta)$.  Since
$\Vert z\Vert = |x| + t = t(1+|x|/t)$ and $\bs p(t,\cdot)$ is a
probability density, this difference can be written as
\begin{equation*}
\frac 1\pi\, \int_{-\infty}^{\infty}
\frac {\ln \,(t \, |y|\,)}{1 + [y + \eta]^2} \, dy \;-\;
\ln t \, (1+\eta) \;=\; \frac 1\pi\,
\int_{-\infty}^{\infty}    \frac {\ln |y|}{1 + [y + \eta]^2} \, dy \;-\;
\ln (1+\eta)\;,
\end{equation*}
where $\eta = |x|/t$.

Fix $K\ge 1$. If $\eta\le K$, it is easy to show that both terms are
bounded separately by a constant which depends on $K$. Assume that
$\eta>K$. We may replace $\ln (1+\eta)$ by $\ln \eta$, paying a
constant. After this replacement, by a change of variables the
previous difference becomes
\begin{equation*}
\frac 1 \pi\,  \int_{-\infty}^{\infty}
\frac {\eta}{1 + \eta^2 [y + 1]^2} \, \ln |y|\; dy \;.
\end{equation*}
Fix $0<c_1<1/2$, $c_2>2$.  On the interval $|y| \le c_1$, the function
$\ln |y|$ is integrable and the ratio is bounded by $C_0/\eta$. On
the interval $|y|\in [c_1,c_2]$, the function $\ln |y|$ is bounded and the
ratio is a probability density. Finally, on the interval
$|y|\in [c_2,\infty)$, as $c_2\ge 2$, the ratio is bounded by
$4/\eta y^2$, and $\ln |y|/y^2$ is integrable in this interval. This
completes the proof of the lemma.
\end{proof}

We conclude this section with some consequences of Proposition
\ref{s22}.  The next two results are Lemmas 3.1 and 3.7 in
\cite{HaiShe16}.  Recall the definition of the mollifier
$\bar{\varrho}_\varepsilon$, $0<\varepsilon<1$, introduced just before
Lemma \ref{s20}, and that the support of $\varrho$ is contained in
$\bb B(0,s_1)$, where $\color{bblue} s_1 = s_0 +1$.


\begin{lemma}
\label{s19}
For every $r>0$, there exists a finite constant $C_0 = C_0(r)$ such that
\begin{equation*}
Q_\varepsilon(z) \; \le \;
- \, \frac{1}{2\pi} \, \ln \,\big(\, \|z\| + \varepsilon \,\big)\,
\;+\; C_0
\end{equation*}
for all $0<\varepsilon<1$, $\Vert z\Vert \le r$.
\end{lemma}

\begin{proof}
Fix $r>0$ and $z\in \bb R \times \bb T$ such that $\Vert z\Vert \le
r$.  By Proposition \ref{s22} and by \eqref{25}, there exists a finite
constant $C_0$, whose value may change from line to line, such that
\begin{equation}
\label{53}
Q_\varepsilon(z) \; \le \; \frac{1}{2\pi} \, \int_{D_\varepsilon} 
\ln \, \frac 1{ 4\, \|z - \varepsilon
w \| }\; \bar{\varrho} (w)\, dw
\;+\; C_0 \;,
\end{equation}
where the integral is performed over the set $D_\varepsilon =\{ w:
\Vert z - \varepsilon w\Vert \le 1/4\}$.

Let $A= 4s_1$, where $s_1$ has been introduced just before the
statement of the lemma. Assume first that $\Vert z\Vert \le A
\varepsilon$. By extracting the factor $4\varepsilon$ from the
logarithm, we may bound the right-hand side of the previous equation
by
\begin{align*}
& \frac{1}{2\pi} \, \ln \, \frac 1{\varepsilon} \;+\;
\frac{1}{2\pi} \, \int_{D_\varepsilon}  
\ln \, \frac 1{\|z/\varepsilon -  w \| }\;
\bar{\varrho} (w)\, dw \;+\; C_0 \\
& \quad \le\; \frac{1}{2\pi} \, \ln \, \frac 1{\varepsilon} \;+\;
\frac{1}{2\pi} \, \int_{D'_\varepsilon}  
\ln \, \frac 1{\|w \| }\;
\bar{\varrho} (z/\varepsilon -  w)\, dw \;+\; C_0
\;,
\end{align*}
where we performed a change of variables and $D'_\varepsilon =\{ w:
\Vert w\Vert \le 1/4 \varepsilon\}$.  By hypothesis, the support of
$\varrho$ is contained in $\bb B(0,s_1)$, where, recall,
$s_1=s_0+1$. Hence, by definition, the support of $\bar{\varrho}$ is
contained in $\bb B(0,2s_1)$, and we may restrict the previous
integral to points $w$ such that $\Vert w \Vert \le A + 2s_1 = 6s_1$
because $\Vert z \Vert \le A \varepsilon$. The previous expression is
thus bounded by
\begin{equation*}
\frac{1}{2\pi} \, \ln \, \frac 1{\varepsilon} \;+\;
\frac{\Vert \bar{\varrho} \Vert_\infty}{2\pi} \, 
\int_{\Vert w \Vert \le A + 2s_1}  
\Big|\, \ln \, \|w \| \, \Big| \, dw \;+\; C_0
\;\le\;
\frac{1}{2\pi} \, \ln \, \frac 1{\varepsilon} \;+\; C_0
\;.
\end{equation*}
To complete the argument note that $\varepsilon^{-1} \le (A+1)/(\Vert
z\Vert + \varepsilon)$ on the set where $\Vert z\Vert \le A
\varepsilon$. Hence, the previous expression is less than or equal to
\begin{equation*}
\frac{1}{2\pi} \, \ln \, \frac 1{\Vert z\Vert + \varepsilon} \;+\; C_0 
\;.
\end{equation*}

We turn to the case $A \varepsilon < \Vert z\Vert \le r$. In this
case, for $w$ in the support of $\bar{\varrho}$, $\Vert z -
\varepsilon w \Vert \ge \Vert z \Vert - \varepsilon \Vert w \Vert \ge
\Vert z \Vert - 2 s_1 \varepsilon \ge \Vert z \Vert/2$ because $A =
4s_1$ and $\Vert z\Vert > A \varepsilon$. Therefore, the expression on
the right-hand side of \eqref{53} is bounded above by 
\begin{equation*}
\frac{1}{2\pi} \, \ln \, \frac 1{\|z \| }\;  \int_{D_\varepsilon} 
\bar{\varrho} (w)\, dw \;+\; C_0 \;,
\end{equation*}
where we extracted the factor $2$ from the logarithm. Considering
separately the cases $\|z \| <1$ and $1\le \|z \| \le r$, we can bound
the previous expression by 
\begin{equation*}
\frac{1}{2\pi} \, \ln \, \frac 1{\|z \| } \;+\; C_0(r) 
\end{equation*}
for some finite constant $C_0(r)$ depending on $r$. As $\Vert z\Vert >
A \varepsilon$, this term is less than or equal to
\begin{equation*}
\frac{1}{2\pi} \, \ln \, \frac {A+1}{ A\, (\|z \| + \varepsilon)} \;+\; C_0(r)
\; \le\; \frac{1}{2\pi} \, \ln \, \frac1 { \|z \| + \varepsilon} \;+\;
C_0(r)\;,
\end{equation*}
absorbing the $A$'s in the constant $C_0(r)$.  This completes the
proof of the lemma.
\end{proof}


\begin{lemma}
\label{s19b}
For each $r>0$, there exist finite constants $C_0=C_0(r)$ such that
\begin{gather*}
\big| \, Q(z) \,-\, Q_\varepsilon(z)\, \big| \;\le\;
C_0 (r)\, \frac{\varepsilon}{\|z\|} \;\cdot
\end{gather*}
for all $\Vert z \Vert \le r$, $0< \varepsilon \le 1$.
\end{lemma}

\begin{proof}
Fix $r>0$.
By \eqref{23b} and the definition of $q=q_{T_0}$, there exists a
finite constant $C_0$ such that
\begin{equation}
\label{54}
\big|\, \partial_t q \, (z) \, \big|
\;\le\; \frac {C_0}{\Vert z\Vert} \;, \quad
\big|\, \partial_x q \, (z) \, \big|
\;\le\; \frac {C_0}{\Vert z\Vert} \;, \quad
q \, (z) \;\le\; C_0 \, \Big\{ \, 1 \,+\, 
\frac {1}{\Vert z\Vert} \, \Big\}
\end{equation}
for all $z=(t,x)$ such that $t \not =0$.

By the definition \eqref{10b} of $Q$ and by \eqref{25}, 
\begin{equation*}
\big|\, Q_\varepsilon (z) \,-\, Q(z)  \, \big|\;\le \; 
\int dw \; \bar{\varrho} (w)
\int du \; q(u)\, \big|\, 
q(u+z-\varepsilon w) \,-\, q(u+z)\,\big|\;.
\end{equation*}
By \eqref{54}, this expression is bounded above by
\begin{equation}
\label{55}
C_0\, \int_0^\varepsilon dt\, \int dw \; \bar{\varrho} (w) 
\int du \;  \Big\{ \, 1 \,+\, 
\frac {1}{\Vert u\Vert} \, \Big\}\, 
\frac {1}{ \big\Vert \, u+z- tw\,\big\Vert} \;
\chi_A (u) \;,
\end{equation}
where $A=[0,T_0] \times \bb T$. Here and below, $C_0$ is a finite
constant which may change from line to line.

Fix $z' \not =0$. Decompose the set $A$ in four pieces:
$A_1 = \bb B(0, \Vert z'\Vert/2)$, $A_2=\bb B(z', \Vert z'\Vert/2)$,
$A_3 = \bb B(0, 4 \Vert z'\Vert) \setminus (A_1\cup A_2)$ and
$A_4 = A\setminus \bb B(0, 4 \Vert z'\Vert)$. Estimating the the
integral below in each of these sets, we show that there
exists a finite constant $C_0=C_0(r)$ such that
\begin{equation*}
\int du \;  \frac {1}{ \big\Vert \, u+z'\,\big\Vert} \;
\chi_A (u) \;\le\; C_0 \;, \quad
\int du \;   \frac {1}{\Vert u\Vert} \, 
\frac {1}{ \big\Vert \, u+z'\,\big\Vert} \;
\chi_A (u) \;\le\; C_0 \Big\{ \, 1 \,+\, 
\ln \frac {1}{\, \Vert z'\,\Vert} \, \Big\}\;.
\end{equation*}
The region $A_4$ is responsible for the log factor.  Hence, \eqref{55}
is less than or equal to
\begin{equation}
\label{57}
C_0\, \varepsilon \;+\; C_0
\int_0^\varepsilon dt\, \int dw \; \bar{\varrho} (w) 
\, \ln \frac {1}{\, \Vert z - t w \,\Vert} \;.
\end{equation}

Assume that $\Vert z \Vert \ge 4s_1 \varepsilon$, where, recall, the
ball of radius $2 s_1$ contains the support of $\bar{\varrho}$. In
this case
$\Vert \, z - t w \,\Vert \ge \Vert \, z\,\Vert - 2 \, \varepsilon s_1
\ge (1/2) \Vert \, z\,\Vert$.
The previous expression is thus bounded above by
\begin{equation*}
C_0 \, \varepsilon \, \Big\{ \, 1 \,+\, \ln \frac 1{\Vert \,z\,\Vert}
\Big\} \;\le\;
C_0(r) \, \frac{\varepsilon}{\Vert \,z\,\Vert}  
\end{equation*}
because $\Vert z\Vert \le r$.

Assume, now, that $\Vert z \Vert \le 4s_1 \varepsilon$ and consider
the second integral in \eqref{57}. We first examine the integral in
the interval $0\le t\le \Vert z\Vert/4s_1$ [note that
$\Vert z\Vert/4s_1 \le \varepsilon$]. In this case, as the support of
$\bar{\varrho}$ is contained in $\bb B(0,2 s_1)$,
$\Vert \,z -tw \,\Vert \ge \Vert \,z \,\Vert/2$. Hence,
\begin{equation*}
\int_0^{\Vert z\Vert/4s_1} dt\, \int dw \; \bar{\varrho} (w) 
\, \ln \frac {1}{\, \Vert z - t w \,\Vert} \;\le\;
\frac {\Vert z\Vert}{4s_1} \, \ln \frac {2}{\Vert z\Vert}
\;\le\; C_0
\end{equation*}
because $\Vert z \Vert \le 4s_1 \varepsilon \le 4s_1 $.

We turn to the integral on the interval
$\Vert z\Vert/4s_1 \le t\le \varepsilon$. Rewriting
$\Vert z - t w \,\Vert$ as $t\, \Vert (z/t) - w \,\Vert$ and changing
variables as $w' = (z/t) - w$, the corresponding integral becomes
\begin{equation*}
\int_{\Vert z\Vert/4s_1}^\varepsilon dt\, \int dw \; \bar{\varrho} 
((z/t) -w) 
\, \Big\{ \ln \frac 1t \,+\, \ln \frac {1}{\, \Vert w \,\Vert} 
\Big\}\;.
\end{equation*}
As $\ln t^{-1}$ is integrable in the interval $[0,1]$ and since
$\Vert z/t\Vert \le 4s_1$, the previous expression is bounded by
\begin{equation*}
C_0 \;+\; C_0 \int_{\Vert z\Vert/4s_1}^\varepsilon dt\, \int_{\bb B(0,6s_1)}
 dw \, \ln \frac {1}{\Vert\, w \,\Vert} \;\le\; C_0 \;.
\end{equation*}
In conclusion, if $\Vert z \Vert \le 4s_1 \varepsilon$, the sum in
\eqref{57} is bounded above by
$C_0 \le C_0 \varepsilon / \Vert z\Vert$. This completes the proof of
the lemma.
\end{proof}

The proof of the next lemma follows from a straightforward computation
based on the formula for $Q$ presented in Proposition \ref{s22}, and
on the fact that $\varrho$ has compact support.

\begin{lemma}
\label{s19c}
For all $0< \varepsilon \le 1$,
\begin{equation*}
Q_\varepsilon (0) \;=\;
\frac 1{2\pi} \, \ln \, \frac 1\varepsilon
\;+\; \frac 1{2\pi} \, \int_{\bb R\times \bb T} \bar{\varrho}
(w)\, \ln \frac{1}{4\, \Vert w\Vert} \; dw
\;+\; \frac 1{2\pi} \, \int_{\bb R\times \bb T} \bar{\varrho}
(w)\, R(\varepsilon\, w) \, dw \;,
\end{equation*}
where $R$ is the function appearing in the statement of Proposition
\ref{s22}.
\end{lemma}

We conclude this section with some result whose proofs are similar to
the previous ones.  For $0<\varepsilon'\,,\, \varepsilon \le 1$, let
$Q_{\varepsilon,\varepsilon'}: \bb R \times \bb T \to \bb R$ be given
by
\begin{equation*}
Q_{\varepsilon,\varepsilon'}(w) \,=\, \bb E[ \, \mf v_{\varepsilon}(0)
\, \mf v_{\varepsilon'}(w)\, ] \;.
\end{equation*}
It follows from the proofs of Lemmata \ref{s19} and \ref{s19b} that
for every $r>0$, there exists a finite constant $C_0 = C_0(r)$ such
that
\begin{equation}
\label{74}
\begin{gathered}
Q_{\varepsilon,\varepsilon'}(z) \; \le \;
- \, \frac{1}{2\pi} \, \ln \,\big(\, \|z\| + \varepsilon' \,\big)\,
\;+\; C_0(r) \\
\big| \, Q(z) \,-\, Q_{\varepsilon,\varepsilon'}(z)\, \big| \;\le\;
C_0 (r)\, \frac{\varepsilon}{\|z\|} 
\end{gathered}
\end{equation}
for all $0<\varepsilon' \le \varepsilon<1$, $\Vert z\Vert \le r$.

Let $\vartheta :\bb R^2 \to \bb R_+$ be a mollifier satisfying the
conditions \eqref{s23}. For $0<\varepsilon \le 1$,
let
$Q^{\varrho,\vartheta}_{\varepsilon}: \bb R \times \bb T \to \bb R$ be
given by
\begin{equation*}
Q^{\varrho,\vartheta}_{\varepsilon}(w) \,=\, \bb E[ \, \mf
v_{\varepsilon}(0) \, \widehat{\mf v}_{\varepsilon}(w) \, ] \;,
\end{equation*}
where
$\widehat{\mf v}_{\varepsilon}(w) := q* \widehat{\xi}_{\varepsilon}$,
$\widehat{\xi}_{\varepsilon} = \vartheta_\varepsilon * \xi$,
$\vartheta_\varepsilon = S^\varepsilon_0 \vartheta$. By the proofs of
Lemmata \ref{s19} and \ref{s19b}, for every $r>0$, there exists a
finite constant $C_0 = C_0(r)$ such that
\begin{equation}
\label{74b}
\begin{gathered}
Q^{\varrho,\vartheta}_{\varepsilon}(z) \; \le \;
- \, \frac{1}{2\pi} \, \ln \,\big(\, \|z\| + \varepsilon \,\big)\,
\;+\; C_0(r) \\
\big| \, Q(z) \,-\, Q^{\varrho,\vartheta}_{\varepsilon}(z)\, \big| \;\le\;
C_0 (r)\, \frac{\varepsilon}{\|z\|} 
\end{gathered}
\end{equation}
for all $0<\varepsilon<1$, $\Vert z\Vert \le r$.

\section{Gaussian multiplicative chaos}
\label{sec05}

Recall the definition of the Gaussian random field $\mf v$,
$\mf v_\varepsilon$, $\varepsilon>0$, introduced in \eqref{08}.  Let
$X_{\gamma, \varepsilon}$, $\gamma\in \bb R$, be the random field
defined by
\begin{equation*}
X_{\gamma, \varepsilon} (f) \;=\;
\< X_{\gamma, \varepsilon} \,,\, f \> \; :=\;  \int f (z)\,
e^{\gamma \, \mf v_\varepsilon (z) \,-\,
(\gamma^2/2)\,  E[\mf v_\varepsilon(z)^2]} \, dz\;,
\end{equation*}
for $f\in C^\infty_c(\bb R \times \bb T)$.

By Lemma \ref{s19c} and since
$\int_{\bb R\times \bb T} \bar{\varrho} (w)\, w \; dw \,=\, 0$,
there exists a finite constant $C(\varrho)$ such that
\begin{equation*}
E[\mf v_\varepsilon(0)^2] \;=\;
\frac 1{2\pi} \, \ln \, \frac 1\varepsilon
\;+\; C(\varrho) \;+\; R (\varrho, \varepsilon)
\;,
\end{equation*}
where $R (\varrho, \varepsilon)$ represents a remainder whose absolute
value is bounded by $C_0(\varrho) \varepsilon^2$. Hence,
\begin{equation*}
\< X_{\gamma, \varepsilon} \,,\, f \>  \;=\;
[1 \,+\, o(\varepsilon)]\,
\int f (z)\, A(\varrho)\, \varepsilon^{\gamma^2/4\pi} \,
e^{\gamma \, \mf v_\varepsilon (z)}\, dz\;,
\end{equation*}
where $A(\varrho) \,=\, \exp\{ -\, (\gamma^2/2)\, C(\rho)\,\}$.

The main result of this section, Theorem \ref{s27} below, states that,
for certain values of $\gamma$, the sequence of random fields
$X_{\gamma, \varepsilon}$ converge in probability, as
$\varepsilon \to 0$, in $C^\alpha$ to a random field $X_{\gamma}$ and
that the limit does not depend on the mollifier $\varrho$ chosen.

Let $\color{bblue} b= \gamma^2/4\pi$ and recall the definition of
$\alpha_\gamma$ introduced just before the statement of Theorem
\ref{mt2}.

\begin{theorem}
\label{s27}
Fix $0<\gamma^2 \,< \, (4/15) (5 - \sqrt{10})\, 4\pi$,
$\alpha< \alpha_\gamma$.  Then, as $\varepsilon \to 0$,
$X_{\gamma,\varepsilon}$ converges in probability in $C^\alpha$ to a
random field, denoted by $X_\gamma$. The limit does not depend on the
mollifier $\varrho$. Moreover, for each $p\in \bb N$,
$1\le p < 8\pi/\gamma^2$, there exists a finite constant
$C(p, \gamma)$ such that
\begin{equation*}
\mathbb{E} \, \big[\, |\,
\langle X_{\gamma} \,,\,
S^\delta_z f \rangle \, |^p \, \big]
\; \le \; C(p,\gamma) \, \Vert f\Vert^p_\infty\,
\delta^{-p(p-1)(\gamma^2/4\pi)}
\end{equation*}
for every $\delta$ in $(0, 1)$, $z\in \bb R \times \bb T$ and
continuous function $f: \bb R \times \bb T \to \bb R$ whose support is
contained in $\bb B (0,1/4)$
\end{theorem}

Recall from \eqref{10b} that we represent by $Q$ the covariances of
the Gaussian field $\{\mf v(z) : z\in \bb R\times \bb T\}$. The proof
relies essentially on the fact that the field is log-correlated:
According to Proposition \ref{s22},
\begin{equation}
\label{34}
Q(z) \;=\; \frac 1{2\pi} \ln^+ \frac{1}{4\, \Vert z\Vert}
\;+\; R(z)\;,
\end{equation}
where $R$ is a bounded, continuous function.

The proof of the next result is similar to the one of
\cite[Proposition A1]{Garban18}.

\begin{lemma}
\label{s2}
Fix $0<\gamma^2<4\pi$. For each $p\in \bb N$,
$1\le p < 8\pi/\gamma^2$, there exists a finite constant
$C(p, \gamma)$ such that 
\begin{equation*}
\mathbb{E} \, \big[\, |\,
\langle X_{\gamma, \varepsilon} \,,\,
S^\delta_z f \rangle \, |^p \, \big]
\; \le \; C(p,\gamma) \, \Vert f\Vert^p_\infty\,
\delta^{-p(p-1)(\gamma^2/4\pi)} 
\end{equation*}
for all $0<\varepsilon \le 1$, $\delta$ in $(0, 1)$,
$z\in \bb R \times \bb T$ and continuous function
$f: \bb R \times \bb T \to \bb R$ whose support is contained in
$\bb B (0,1/4)$,
\end{lemma}

\begin{proof}
Consider first the case $p=1$. For each $w\in \bb R \times \bb T$,
$\bb E[ \, \exp\{ \, \gamma \mf v_\varepsilon (w) \,-\, (\gamma^2/2)\,
\bb E[\mf v_\varepsilon(w)^2] \, \}\,]=1$. Hence
\begin{equation*}
\mathbb{E} \, \big[\, |\,
\langle X_{\gamma, \varepsilon} \,,\,
S^\delta_z f \rangle \, | \, \big] \;\le\;
\int_{(\bb R \times \bb T)} dz_1
\,\big|\, (S^\delta_z f) (z_1) \,\big| \;.
\end{equation*}
As this integral is bounded by $\Vert f\Vert_\infty$, the lemma is
proved for $p=1$.

We turn to the case $p\ge 2$.  By definition of
$X_{\gamma, \varepsilon}$ and $Q_\varepsilon$, the expectation
appearing on the left-hand side of the statement is bounded by
\begin{equation*}
\int_{(\bb R \times \bb T)^p} dz_1 \cdots  dz_p\,
\prod_{j=1}^p \,\big|\, (S^\delta_z f) (z_j) \,\big|\,\, 
\exp\Big\{ \gamma^2 \sum_{1\le i<j\le p} Q_\varepsilon (z_j-z_i)\Big\} \;.
\end{equation*}
The change of variables $z'_j = z_j -z$ and the fact that the support
of $f$ is contained in $\bb B(0,1/4)$ permits to bound the previous
expression by
\begin{equation*}
\delta^{-2p} \, \Vert f\Vert^p_\infty\, 
\int_{\bb B(0,\delta/4)^p} dz_1 \cdots  dz_p\,
\exp\Big\{ \gamma^2 \sum_{1\le i<j\le  p} Q_\varepsilon (z_j-z_i)\Big\} \;.
\end{equation*}
By Lemma \ref{s19}, this expression is less than or equal to
\begin{equation*}
C(p, \gamma)\, \delta^{-2p} \, \Vert f\Vert^p_\infty\, 
\int_{\bb B(0,\delta/4)^p} dz_1 \cdots  dz_p\,
\prod_{1\le i<j\le  p} \frac 1{\Vert z_j-z_i \Vert^{\gamma^2/2\pi} }
\end{equation*}
for some finite constant $C(p, \gamma)$. The change of variables
$z'_j = z_j/\delta$ yields that this expression is bounded by
\begin{equation*}
C(p, \gamma)\, \Vert f\Vert^p_\infty\,
\delta^{-p(p-1)(\gamma^2/4\pi)} \, \int_{\bb B(0,1/4)^p} dz_1 \cdots  dz_p\,
\prod_{1\le i<j\le  p} \frac 1{ \Vert z_j-z_i \Vert^{\gamma^2/2\pi} }\;\cdot
\end{equation*}
By Corollary \ref{s12}, as $\gamma^2/2\pi <2$ and $p<8\pi/\gamma^2$, the
previous expression is less than or equal to
$C(p, \gamma)\, \Vert f\Vert^p_\infty\,
\delta^{-p(p-1)(\gamma^2/4\pi)}$.
\end{proof}

The proofs of the next two lemmas are similar to the one of
\cite[Theorem 3.2]{HaiShe16}.

\begin{lemma}
\label{s1}
Fix $0<\gamma^2<4\pi$ and let $a= \gamma^2/2\pi<2$. There exists a
constant $C(\gamma)$ such that 
\begin{equation*}
\mathbb{E} \, \big[\, |\,
\langle X_{\gamma, \varepsilon}
\,-\, X_{\gamma,  \varepsilon'} \,,\,
S^\delta_z f \rangle \, |^2 \, \big]
\; \le \; C(\gamma) \, \Vert f\Vert^2_\infty\,
\varepsilon^{2\kappa}\, \delta^{-a- 2\kappa} 
\end{equation*}
for all $\varepsilon$, $\varepsilon'$, $\delta$ in $(0, 1)$,
$0<2\kappa < 1\wedge (2-a)$, $z\in \bb R \times \bb T$ and continuous
function $f: \bb R \times \bb T \to \bb R$ whose support is contained
in $\bb B (0,1/4)$,
\end{lemma}

\begin{proof}
Fix $\varepsilon$, $\varepsilon'$, $\delta$ in $(0, 1)$ and assume,
without loss of generality that $\varepsilon' \le \varepsilon$. By
definition of $X_{\gamma, \varepsilon}$, the left-hand side of
the previous formula is equal to
\begin{equation*}
\int_{\bb R \times \bb T} dz_1 \int_{\bb R \times \bb T}  dz_2\,
(S^\delta_z f) (z_1)\, (S^\delta_z f) (z_2) \,
R_{\gamma, \varepsilon, \varepsilon'}(z_2-z_1)\;,
\end{equation*}
where
\begin{equation*}
R_{\gamma, \varepsilon, \varepsilon'}(w) \;=\;
e^{\gamma^2 Q_\varepsilon (w)}
\,-\, 2\, e^{\gamma^2 Q_{\varepsilon, \varepsilon'}(w)}
\,+\, e^{\gamma^2 Q_{\varepsilon'}(w)}\;,
\end{equation*}
and $Q_{\varepsilon,\varepsilon'}(w)$ has been introduced at the end
of Section \ref{sec07}.  By definition of $S^\delta_z f$ and since
the support of $f$ is contained in $\bb B(0,1/4)$, the absolute value
of the last integral is bounded above by
\begin{equation}
\label{37}
C_0\, \frac{\Vert f\Vert^2_\infty}{\delta^2}\,
\int_{\bb B(0,\delta/2)}
\big|\, R_{\gamma, \varepsilon, \varepsilon'}(w) \,\big| \; dw
\end{equation}
for some finite constant $C_0$.

Suppose first that $\varepsilon > \delta/2$.  By Lemma \ref{s19},
there exists a finite constant $C(\gamma)$ such that
$\exp\{\gamma^2 Q_{\varepsilon}(w)\} \,\le\, C(\gamma) \, \Vert
w\Vert^{- a}$
for all $\Vert w \Vert \le 1$. Here, recall, $a= \gamma^2/2\pi<2$. By
\eqref{74}, a similar bound holds for
$Q_{\varepsilon, \varepsilon'}(w)$ and $Q_{\varepsilon'}(w)$ in place
of $Q_{\varepsilon}(w)$.  Hence, \eqref{37} is less than or equal to
\begin{equation*}
C(\gamma) \, \frac{\Vert f\Vert^2_\infty}{\delta^2}\,
\int_{\bb B(0,\delta/2)} \frac 1{\Vert w\Vert^{a}} \; dw
\;\le\; C(\gamma) \, \Vert f\Vert^2_\infty\, \delta^{-a}
\;\le\; C(\gamma) \, \Vert f\Vert^2_\infty\,
\varepsilon^{2\kappa}\, \delta^{-a- 2\kappa}
\end{equation*}
for all $\kappa>0$ because $\delta/2 \le \varepsilon$.

We turn to the case $\varepsilon \le \delta/2$. We first consider the
integral appearing in \eqref{37} on the set $\bb B(0,\varepsilon)$.
By the bounds on $Q_{\varepsilon}$, $Q_{\varepsilon'}$,
$Q_{\varepsilon, \varepsilon'}$ presented above,
\begin{equation}
\label{17}
\frac{1}{\delta^2} \int_{\bb B(0,\varepsilon)}
\big|\, R_{\gamma, \varepsilon, \varepsilon'}(w) \,\big| \; dw\;
\le\; C(\gamma) \, \frac{\varepsilon^{2-a}}{\delta^2}
\le\; C(\gamma) \, \varepsilon^{2\kappa}\, \delta^{-a- 2\kappa}
\end{equation}
provided $2\kappa < 2-a$ because $\varepsilon \le \delta/2$.

We next consider the integral \eqref{37} on
$\bb B(0,\delta/2) \setminus \bb B(0,\varepsilon)$. Note that
\begin{align*}
& \Big|\, e^{\gamma^2 Q_{\varepsilon}(w)} \,-\,
e^{\gamma^2 Q_{\varepsilon, \varepsilon'}(w)} \,\Big| \\
&\qquad \;\le\;
e^{\gamma^2 Q(w)} \, \Big\{\,
\Big|\, e^{\gamma^2 [Q_{\varepsilon}(w) - Q(w)]} \,-\, 1\, \Big|
\,+\, \Big|\, e^{\gamma^2 [Q_{\varepsilon, \varepsilon'}(w) - Q(w)]}
\,-\, 1 \,\Big| \, \Big\}\;.
\end{align*}
By Proposition \ref{s22}, there exists a finite constant $C_0$ such
that $\exp\{\gamma^2 Q(w)\} \le C_0 \Vert w\Vert^{-a}$ for all
$\Vert w\Vert \le 1/2$ [consider, separately, the cases
$\Vert w\Vert \le 1/4$ and $1/4 < \Vert w\Vert \le 1/2$]. Hence, by
Lemma \ref{s19b} and \eqref{74}, as $\varepsilon \le \Vert w\Vert$,
\begin{equation*}
\Big|\, e^{\gamma^2 Q_{\varepsilon}(w)} \,-\,
e^{\gamma^2 Q_{\varepsilon, \varepsilon'}(w)} \,\Big| \;\le\;
C(\gamma) \, \frac {\varepsilon}{\Vert w\Vert^{1+a}}\;\cdot
\end{equation*}
A similar bound holds for $Q_{\varepsilon'}$ instead of
$Q_{\varepsilon}$ [because $\varepsilon' \le \varepsilon$]. Thus,
\begin{equation*}
\frac{1}{\delta^2} \int_{\bb B(0,\delta/2) \setminus \bb B(0,\varepsilon)}
\big|\, R_{\gamma, \varepsilon, \varepsilon'}(w) \,\big| \; dw\;
\le\; \frac{C(\gamma)}{\delta^2} \int_{\bb B(0,\delta/2) \setminus
  \bb B(0,\varepsilon)}
\frac {\varepsilon}{\Vert w\Vert^{1+a}}\; dw\; .
\end{equation*}
This expression  is bounded above by
\begin{equation*}
\frac{C(\gamma)}{\delta^2}\,
\begin{cases}
\displaystyle \varepsilon  \, \delta^{1-a}
& \text{if $a<1$}\;, \\
\displaystyle \varepsilon  \,  \ln
(\delta/\varepsilon) & \text{if $a=1$}\;, \\
\displaystyle \varepsilon^{2-a} & \text{if
  $1<a<2$}\;.
\end{cases}
\end{equation*}
It is easy to check that these expressions are bounded by $C(\gamma)
\, \varepsilon^{2\kappa}\, \delta^{-a- 2\kappa}$  in all three cases
provided $2\kappa < 1 \wedge (2-a)$. In conclusion, under the previous
assumption on $\kappa$,
\begin{equation*}
\frac{1}{\delta^2} \int_{\bb B(0,\delta/2) \setminus \bb B(0,\varepsilon)}
\big|\, R_{\gamma, \varepsilon, \varepsilon'}(w) \,\big| \; dw\;
\le\; C(\gamma) \, \varepsilon^{2\kappa}\, \delta^{-a- 2\kappa}\;.
\end{equation*}
This estimate together with \eqref{17} completes the proof of
the lemma.
\end{proof}

Recall that $\vartheta$ is another mollifier and recall the definition
of the Gaussian random field
$\{\widehat{\mf v}_{\varepsilon}(w) : w\in \bb R \times \bb T\}$
introduced above \eqref{74b}. Let $\widehat X_{\gamma, \varepsilon}$,
$\gamma\in \bb R$, be the random field defined by
\begin{equation*}
\widehat X_{\gamma, \varepsilon} (f) \;=\;
\< \widehat X_{\gamma, \varepsilon} \,,\, f \> \; :=\;  \int f (z)\,
e^{\gamma \, \widehat{\mf v}_\varepsilon (z) \,-\,
(\gamma^2/2)\,  E[\widehat{\mf v}_\varepsilon(z)^2]} \, dz\;,
\end{equation*}
for $f\in C^\infty_c(\bb R \times \bb T)$. The proof of the previous
result and the estimates \eqref{74b} yield the next result.

\begin{lemma}
\label{s1b}
Fix $0<\gamma^2<4\pi$ and let $a= \gamma^2/2\pi<2$. There exists a
constant $C(\gamma)$ such that 
\begin{equation*}
\mathbb{E} \, \big[\, |\,
\langle X_{\gamma, \varepsilon}
\,-\, \widehat X_{\gamma,  \varepsilon} \,,\,
S^\delta_z f \rangle \, |^2 \, \big]
\; \le \; C(\gamma) \, \Vert f\Vert^2_\infty\,
\varepsilon^{2\kappa}\, \delta^{-a- 2\kappa} 
\end{equation*}
for all $\varepsilon$, $\delta$ in $(0, 1)$,
$0<2\kappa < 1\wedge (2-a)$, $z\in \bb R \times \bb T$ and continuous
function $f: \bb R \times \bb T \to \bb R$ whose support is contained
in $\bb B (0,1/4)$,
\end{lemma}

Recall from Section \ref{sec01} that we denote by
$\mf C^{n} (\bb R \times \bb T)$, $n \in \bb N$ the dual of
$C^{n}_c(\bb R \times \bb T)$. Consider the random fields
$X_{\gamma, \varepsilon}$ as elements of
$\mf C^{\mf n}(\bb R \times \bb T)$ for some $n \in \bb N$. Next
result follows from Lemmata \ref{s2}, \ref{s1} and \ref{s1b}.

\begin{corollary}
\label{s4}
For every $0<\gamma^2<4\pi$, as $\varepsilon \to 0$, the sequence of
random fields $X_{\gamma, \varepsilon}$ converges in $L^2$ to a random
field represented by $X_{\gamma}$.  The limit does not depend on the
mollifier $\varrho$. Moreover, for each $p\in \bb N$,
$1\le p < 8\pi/\gamma^2$, there exists a finite constant
$C(p, \gamma)$ such that
\begin{equation*}
\mathbb{E} \, \big[\, |\,
\langle X_{\gamma} \,,\,
S^\delta_z f \rangle \, |^p \, \big]
\; \le \; C(p,\gamma) \, \Vert f\Vert^p_\infty\,
\delta^{-p(p-1)(\gamma^2/4\pi)}
\end{equation*}
for every $\delta$ in $(0, 1)$, $z\in \bb R \times \bb T$ and
continuous function $f: \bb R \times \bb T \to \bb R$ whose support is
contained in $\bb B (0,1/4)$
\end{corollary}

\begin{remark}
\label{rm4}
The limit field $X_\gamma$ is called a Gaussian multiplicative chaos
(GMC). It has been introduced by Kahane \cite{Kah85}. We refer to the
reviews \cite{RobVar10, RhoVar14, DupRhoSheVar17} for properties of
these fields.
\end{remark}

\begin{lemma}
\label{s26b}
Fix $0<\gamma^2<4\pi$, $0<\nu<1$ and $p\in \bb N$ such that $2\le p
<8\pi/\gamma^2$.  Let $p_\nu = p - \nu (p-2)$.  Then, there exists a
finite constant $C(p,\gamma)$ such that
\begin{equation*}
\mathbb{P} \, \big[\, |\,
\langle X_{\gamma, \varepsilon}
\,-\, X_{\gamma} \,,\, S^\delta_z f \rangle \, | \,>\, \eta \, \big]
\;\le \; \frac{C(p,\gamma)}{\eta^{p_\nu}} \, \, \Vert f\Vert^{p_\nu}_\infty\,
(\varepsilon/\delta)^{2\kappa\nu}\, \delta^{-(\gamma^2/4\pi) p(p-1)} 
\end{equation*}
for all $\eta>0$, $0 < \varepsilon \le 1$, $0<\delta \le 1$, $0<
2\kappa < 1 \wedge 2a$, $z\in \bb R \times \bb T$ and continuous
function $f: \bb R \times \bb T \to \bb R$ whose support is contained
in $\bb B (0,1/4)$.
\end{lemma}

\begin{proof}
Fix $\eta >0$, $0 < \varepsilon \le 1$, $0<\delta \le 1$, $z\in \bb R
\times \bb T$, and a continuous function $f: \bb R \times \bb T \to
\bb R$ whose support is contained in $\bb B (0,1/4)$.

Recall that $a= \gamma^2/2\pi$. By Lemma \ref{s1}, after
sending $\varepsilon'\to 0$, and Tchebycheff inequality, there exists a
constant $C(\gamma)$ such that
\begin{equation*}
\mathbb{P} \, \big[\, |\,
\langle X_{\gamma, \varepsilon}
\,-\, X_{\gamma} \,,\, S^\delta_z f \rangle \, | \,>\, \eta \, \big]
\; \le \;  \frac{C(\gamma)}{\eta^2} \, \Vert f\Vert^2_\infty\,
(\varepsilon/\delta)^{2\kappa}\, \delta^{-a} 
\end{equation*}
for all $0< 2\kappa < 1 \wedge 2a$. 

By Lemma \ref{s2}, Corollary \ref{s4} and Tchebycheff inequality,
there exists a constant $C(\gamma)$ such that
\begin{equation*}
\mathbb{P} \, \big[\, |\,
\langle X_{\gamma, \varepsilon}
\,-\, X_{\gamma} \,,\, S^\delta_z f \rangle \, | \,>\, \eta \, \big]
\; \le \;   \frac{C(p,\gamma)}{\eta^p} \, \Vert f\Vert^p_\infty\,
\delta^{-p(p-1)a/2} \;.
\end{equation*}
We have used above that $(a+b)^p \le 2^p(a^p + b^p)$ and that $a=
\gamma^2/2\pi$.

Fix $0<\nu <1$.  Take the first inequality to the power $\nu$,
the second one to the power $1-\nu$ and multiply them to get that
\begin{equation*}
\mathbb{P} \, \big[\, |\,
\langle X_{\gamma, \varepsilon}
\,-\, X_{\gamma} \,,\, S^\delta_z f \rangle \, | \,>\, \eta \, \big]
\; \le \;   \frac{C(p,\gamma)}{\eta^{p_\nu}} \, \Vert f\Vert^{p_\nu}_\infty\,
(\varepsilon/\delta)^{2\kappa\nu}\, \delta^{-p(p-1)a/2} \,
\delta^{a\nu [p(p-1)-2]/2}\;,
\end{equation*}
where $p_\nu = 2 \nu + (1-\nu)p = p - \nu (p-2)$. This
completes the proof of the lemma because $p\ge 2$.
\end{proof}

Let
\begin{equation*}
p_0 \;=\; p_0(\gamma) \;:=\;
\max\Big\{\, p\in \bb N : p \,<\, \frac{8\pi}{\gamma^2}\,\Big\}\;.
\end{equation*}

\begin{corollary}
\label{s3}
Fix $0<\gamma^2<4\pi$ and $2\le s \le p_0-1$.  Then, there exists a
finite constant $C(\gamma,s)$ such that
\begin{equation*}
\mathbb{P} \, \big[\, |\,
\langle X_{\gamma, \varepsilon}
\,-\, X_{\gamma} \,,\, S^\delta_z f \rangle \, | \,>\, \eta \, \big]
\;\le \; \frac{C(s,\gamma)}{\eta^{s_\kappa}} \,
\, \Vert f\Vert^{s_\kappa}_\infty\,
(\varepsilon/\delta)^{2\kappa^2}\, \delta^{-(\gamma^2/4\pi) (s-1/2)^2} 
\end{equation*}
for all $0 < \varepsilon \le 1$, $0<\delta \le 1$, $0<2\kappa < 1
\wedge 2a$, $z\in \bb R \times \bb T$ and continuous function $f: \bb
R \times \bb T \to \bb R$ whose support is contained in $\bb B
(0,1/4)$, where $s_\kappa = s - \kappa (p-2)$. 
\end{corollary}

\begin{proof}
Fix $2\le s \le p_0-1$, and let $q\in\bb N$ such that $q\le
s<q+1$. Note that $2\le m < (8\pi/\gamma^2)$ for $m=q$, $q+1$. We
may, therefore, apply Lemma \ref{s26b} with $q$ and $q+1$. Let $0<
\theta\le 1$ be such that $s = \theta\, q + (1-\theta) \, (q+1)$.

Apply Lemma \ref{s26b} with $p=q$, $\nu = \kappa$ and take the
inequality to the power $\theta$. Repeat the same procedure with $p=q+1$,
$\nu = \kappa$ and take the power $1-\theta$. Multiply both
inequalities to conclude that
\begin{equation*}
\mathbb{P} \,
\big [\, |\, \langle X_{\gamma, \varepsilon} \,-\, X_{\gamma} \,,\, 
S^\delta_z f \rangle \, | \,>\, \eta \, \big] \;\le\;
\frac{C(\gamma)}{\eta^{s_\kappa}} \, \, \Vert f\Vert^{s_\kappa}_\infty\,
(\varepsilon/\delta)^{2\kappa^2}\, \delta^{-(\gamma^2/4\pi) \, B_\theta} \;,
\end{equation*}
where $s_\kappa = s - \kappa (s-2)$ and $B_\theta = \theta q(q-1) +
(1-\theta)q(q+1)$. As $s = \theta\, q + (1-\theta) \, (q+1)$, an
elementary computation yields that $B_\theta \le s(s-1) + (1/4) =
[s-(1/2)]^2$.
\end{proof}

\subsection{Convergence in $C^\alpha$}

We Prove in Theorem \ref{s27} below that the sequence of random
fields $X_{\gamma, \varepsilon}$ converges in $C^\alpha([-T,T]\times
\bb T)$ for all $T>0$. 

We start introducing an orthonormal basis of
$L^2(\bb R \times \bb T)$.  We refer to \cite[Chapter 3]{Mey92},
\cite[Chapter 1]{Tri08} and \cite[Section 3]{Hai14} for a proof of all
assertions made below.  Let $\varphi: \bb R \to \bb R$ be the scaling
function of a multiresolution of $\bb R$, the ``father wavelet''. This
is a function in $L^2(\bb R)$ such that
\begin{itemize}
\item[(i)] $\displaystyle \int_{\bb R} \varphi(x)\, \varphi(x+k)\, dx
  \;=\; \delta_{0,k}$ for every $k\in\bb Z$;
\item[(ii)] There exist constants $(a_k : k\in \bb Z)$ auch that
  $\displaystyle \varphi(x) = \sum_{k\in\bb Z} a_k\, \varphi(2x-k)$.
\end{itemize}

For every $r\in\bb N$, there exists a compactly supported function
$\varphi$ in $C^r(\bb R)$ satisfying (i) and (ii). Moreover, in (ii),
$a_k=0$ for all but a finite number of integers $k$.

For $p$, $n\in \bb Z$, let $\varphi^n_p (x) \,=\, 2^{n/2}\, \varphi
(2^n x - p)$,
\begin{equation*}
\psi (x) \;=\; \sum_{q\in\bb Z} (-1)^q \, a_{1-q}\, \varphi(2x-q) \;,
\quad
\psi^n_p (x) \;=\; 2^{n/2}\, \psi (2^n x - p) \;.
\end{equation*}
For each integer $0\le m\le r$,
\begin{equation}
\label{38}
\int_{\bb R} \psi(x)\, x^m\, dx \;=\; 0\;,
\end{equation}
and, for every $n\in \bb Z$, the set
\begin{equation}
\label{39}
\big\{ \varphi^n_p : p\in\bb Z \big\} \,\cup\,
\big\{ \psi^m_p : m\ge n \,,\, p\in \bb Z \big\}
\end{equation}
forms an orthonormal basis of $L^2(\bb R)$.

A multiresolution analysis is also available on the torus $\bb T$. Fix
$L$ sufficiently large for $\varphi^L_0$, $\psi^L_0$ to have a support
contained in $(-1/2, 1/2)$. Let
$\color{bblue} P_j = \{m \in \bb Z : 0\le m < 2^{j} \}$.

For $j\ge L$, $m\in P_j$, let
\begin{equation}
\label{40}
\varphi^j_{\tau, m}(x) \;=\; \sum_{\ell\in \bb Z} \varphi^{j}_{m}(x-\ell)
\;, \quad
\psi^{j}_{\tau, m}(x) \;=\; \sum_{\ell\in \bb Z} \psi^{j}_{m}(x-\ell)
\;.
\end{equation}
The functions $\varphi^j_{\tau, m}$, $\psi^{j}_{\tau, m}$ are
periodic, with period $1$. Let $\varphi^{\tau,j}_{m}$,
$\psi^{\tau, j}_{m}: \bb T \to \bb R$ be the functions defined by
\begin{equation*}
\varphi^{\tau,j}_{m}(x) \;=\; \varphi^j_{\tau, m}(x)
\;, \quad
\psi^{\tau, j}_{m}(x) \;=\; \psi^{j}_{\tau, m}(x)\;, \quad x\in \bb
T\;=\; [-1/2 , 1/2) \;.
\end{equation*}
Since the support of $\varphi^L_0$, $\psi^L_0$ are contained in
$(-1/2,1/2)$, for each fixed $x\in \bb R$, $j\ge L$ and $m\in P_j$, in
the sums \eqref{40} there is only one $\ell\in \bb Z$ such that
$\varphi^{j}_{m}(x-\ell)\not = 0$. 

Extend the operator $S^{\delta}_{z}$ introduced in \eqref{35} to
functions defined on $\bb R$: For $0<\delta \le 1$, $y\in \bb R$ and
$g:\bb R\to \bb R$, let
$(S^{\delta}_{y} g)(x) = \delta^{-1} g(\delta^{-1}(x-y))$. By
definition,
\begin{equation}
\label{42}
\varphi^{\tau,j}_{m} \;=\; 2^{-j/2}\, S^{2^{-j}}_{m/2^j}
\varphi\;, \quad \psi^{\tau,j}_{m} \;=\; 2^{-j/2}\, S^{2^{-j}}_{m/2^j}
\psi\;.
\end{equation}
There is a slight abuse of notation in this identify, as $\varphi$,
$\psi$ are functions defined on $\bb R$. For $x\in \bb T$,
$2^{-j}\, (S^{2^{-j}}_{m/2^j} \varphi)(x)$ has to be understood as
$\varphi(2^j[x - \ell - (m/2^j)])$, where $\ell$ is the unique integer
which turns the point $x - \ell - (m/2^j)$ an element of
$[-1/2, 1/2)$.

The set
$\{\varphi^{\tau,L}_{m} : m\in P_L\} \cup \{\psi^{\tau, n}_{m} : n\ge
L \,,\, m\in P_{n}\}$
forms an orthonormal basis of $L^2(\bb T)$.  Clearly, tensor products
provide an orthonormal basis of $L^2(\bb R \times \bb T)$, but we
proceed differently to have products of functions equally scaled.  Let
$\ms B \,=\, \{\phi^{0,L}_{p,q}, \phi^{\iota,n}_{p,m} : \iota\in \{1,
2, 3\} \,,\, p\in \bb Z \,,\, q \in P_L \,,\, n \ge L \,,\, m\in
P_{n}\}$, where
\begin{align*}
& \phi^{0,L}_{p,q} (t,x)\;=\; \varphi^L_p (t)\,
\varphi^{\tau,L}_q (x)\;, \quad
\phi^{1,n}_{p,m} (t,x)\;=\;
\varphi^{n}_p (t)\, \psi^{\tau, n}_m (x)\;,  \\
& \quad \phi^{2,n}_{p,m} (t,x)\;=\;
\psi^{n}_p (t) \, \varphi^{\tau,n}_m (x) \;, \quad
\phi^{3,n}_{p,m} (t,x) \;=\;
\psi^{n}_p (t)\, \psi^{\tau, n}_m (x) \;.
\end{align*}
It is not difficult to show that this family is orthogonal.  It
follows from property (ii) and \eqref{39} that it generates
$L^2(\bb R \times \bb T)$. Moreover, in view of \eqref{42}, the
elements of this basis can be represented in terms of the operator
$S^\delta_z$:
\begin{equation}
\label{41}
\phi^{0,L}_{p,q} \;=\;
2^{-L} \, S^{2^{-L}}_{(p/2^L, q/2^L)} \, \Phi_{0}\;, \quad
\phi^{\iota,n}_{p,m} \;=\;
2^{-n} \, S^{2^{-n}}_{(p/2^n, m/2^n)}
\, \Phi_{\iota}\;,
\end{equation}
with the same convention as in \eqref{42}, and where
$\Phi_{0} (t,x) = \phi(t) \, \phi(x)$,
\begin{equation*}
\Phi_{1} (t,x) = \varphi(t)\, \psi(x) \;, \quad
\Phi_{2} (t,x) = \psi(t) \, \varphi(x) \;, \quad
\Phi_{3} (t,x) = \psi(t) \, \psi(x)\;.
\end{equation*}

Let $X$ be an element in the dual of $C^r_0(\bb R \times \bb T)$ for
some $r>0$.  Fix $T_1>0$, and let $A_\iota = A_\iota(T_1)$, $A = A(T_1)$ be
given by 
\begin{equation}
\label{45}
\begin{gathered}
A_\iota \;=\; \sup_{n\ge L}
\max_{m\in P_n} \max_{p}
2^{n\alpha}\, \big|\, \< X \,,\,
S^{2^{-n}}_{(p/2^n, m/2^n)} \Phi_\iota  \>\, \big|\;,
\quad 1\,\le\, \iota \,\le\, 3\;, \\
A_0 \;=\; \max_{m\in P_L} \max_{p}
2^{L\alpha}\, \big|\, \< X \,,\,
S^{2^{-L}}_{(p/2^L, m/2^L)} \Phi_0  \>\, \big|\;, \quad
A \;=\; \max_{0\le \iota \le 3} A_\iota \;,
\end{gathered}
\end{equation}
where the maximum in the first line is carried over all $p\in \bb Z$
such that $|p/2^n|\le T_1+1$. In the second line, it is carried over all
$p\in \bb Z$ such that $|p/2^L|\le T_1+1$.

\begin{lemma}
\label{l01}
Let $\alpha<0$, $X$ be an element in the dual of
$C^r_0(\bb R \times \bb T)$, for some $r\in \bb N$, $r>-\alpha$. Fix
$T_1>0$. Then, there exists a constant $C_0$ such that 
\begin{equation*}
\big|\, \< X \,,\, S^\delta_z h \>
\,\big|\;\le\; C_0 \, A\, \delta^\alpha
\end{equation*}
for all $z\in [-T_1,T_1]\times \bb T$, $\delta \in (0,1]$ and function $h$
in $C^r_0(\bb R \times \bb T)$ whose support is contained in
$\bb B(0,1/4)$ and such that $\Vert h\Vert_{C^r} \le 1$, In
particular, if $A<\infty$, $X \in C^\alpha([-T_1,T_1]\times \bb T)$ and
$\Vert X \Vert_{C^\alpha([-T_1,T_1]\times \bb T)} \le C_0 A$.
\end{lemma}

\begin{proof}
Fix $z\in [-T_1,T_1]\times \bb T$, $\delta \in (0,1]$ and a function $h$
in $C^r_0(\bb R \times \bb T)$ whose support is contained in
$\bb B(0,1/4)$ and such that $\Vert h\Vert_{C^r} \le 1$. In this
proof, $C_0$ represents a constant which does not depend on $T_1$,
$z$, $\delta$ and $h$, and may change from line to line.

As $S^\delta_z h$ belongs to $L^2(\bb R \times \bb T)$ and
$\phi^{0,L}_{p,q}$, $\phi^{\iota,n}_{p,m} $ to $C^r_0(\bb R \times \bb
T)$,
\begin{equation}
\label{44}
\begin{aligned}
\< X \,,\, S^\delta_z h \> \; &=\;
\sum_{\iota = 1}^3 \sum_{n=L}^\infty \sum_{p\in \bb Z}  \sum_{m\in P_{n}}
\< X \,,\,  \phi^{\iota,n}_{p,m} \>\,
\< \phi^{\iota ,n}_{p,m} \,,\, S^\delta_z h \> \\
&+\;
\sum_{p\in \bb Z} \sum_{q\in P_{L}}  \< X \,,\,  \phi^{0,L}_{p,q} \>\,
\< \phi^{0,L}_{p,q} \,,\, S^\delta_z h \> \;.
\end{aligned}
\end{equation}
Clearly, $\< \phi^{\iota ,n}_{p,m} \,,\, S^\delta_z h \> = 0$ if
the supports of $ \phi^{\iota ,n}_{p,m}$ and $S^\delta_z h$ are
disjoint. Hence, in view of \eqref{41}, \eqref{45}, the absolute value
of the first sum is bounded by
\begin{equation}
\label{43}
A \sum_{\iota = 1}^3 \sum_{n=L}^\infty \sum_{p\in \bb Z} \sum_{m\in P_{n}}
2^{-n(1+\alpha)}\, \big|\, \< \phi^{\iota ,n}_{p,m} \,,\, 
S^\delta_z h \> \, \big| \;.
\end{equation}

Let $n_0$ be the integer such that $2^{-n_0}\le \delta < 2^{-n_0+1}$.
Since $h$ belongs to $C^r_0(\bb R \times \bb T)$ and
$\Vert h\Vert_{C^r} \le 1$, by a Taylor expansion, \eqref{38} and
Schwarz inequality, for $\iota\in \{1,2,3\}$, $p\in\bb Z$, $n\ge n_0$,
$m\in P_n$,
\begin{align*}
\big|\, \< \phi^{\iota,n}_{p,m} \,,\, S^\delta_z h \> \, \big| \;\le\;
C_0 \, 2^{-(n-n_0)(1+r)} \, 2^{n_0}\;.
\end{align*}
Note that the constant $C_0$ does not depend on $h$. Here and below,
one only uses the fact that the support of $h$ is contained in $\bb
B(0,1/4)$ and that $\Vert h\Vert_{C^r} \le 1$.

For a fixed $\delta$, there are less than
$(3\delta)^2 2^{2n} \le C_0 2^{2(n-n_0)}$ pairs $(p,m)$ for which the
supports of $ \phi^{\iota ,n}_{p,m}$ and $S^\delta_z h$ are not
disjoints. Hence the sum of the terms $n\ge n_0$ in \eqref{43} is
bounded by
\begin{equation*}
C_0\, \sum_{n=n_0}^\infty 2^{-n(1+\alpha)}\, 2^{2(n-n_0)}\,
2^{-(n-n_0)(1+r)} \, 2^{n_0} \;\le\;
C_0\, \sum_{n=n_0}^\infty 2^{-(n-n_0)(r + \alpha)}\, 2^{-n_0\alpha}\;.
\end{equation*}
As $r>-\alpha$, this expression is bounded by $C_0 2^{-n_0\alpha}
\,\le \, C_0 \delta^\alpha$.

We turn to the terms $n\le n_0$. Estimating $\phi^{\iota,n}_{p,m} $
by its $L^\infty$ norm yields that
\begin{align*}
\big|\, \< \phi^{\iota,n}_{p,m} \,,\, S^\delta_z h \> \, \big| \;\le\;
C_0 \, 2^{n} \;.
\end{align*}
The number of pairs for which
$\< \phi^{\iota,n}_{p,m} \,,\, S^\delta_z h \>$ does not vanish is
bounded by $C_0 \delta^2 2^{2n}$. Hence, the contribution of the terms
$n\le n_0$ to the sum in \eqref{43} is bounded by
\begin{equation*}
C_0\, \sum_{n=L}^{n_0} 2^{-n(1+\alpha)}\, 2^n\, \delta^2\, 2^{2n}\,
\;=\; C_0\, \sum_{n=L}^{n_0} 2^{(2-\alpha)n}\,  \delta^2
\;\le\; C_0\, 2^{(2-\alpha)n_0}\,  \delta^2 \;\le \;
C_0 \, \delta^\alpha\;.
\end{equation*}

It remains to estimate the second sum in \eqref{44}. We may proceed as
for the terms $n\le n_0$ to conclude that the absolute value of the
second term on the right-hand side of \eqref{44} is bounded above by
\begin{equation*}
C_0\, A\,  2^{(2-\alpha)L}\,  \delta^2 \;= \;
C_0\, A\, \delta^2
\;\le \;  C_0 \, A\, \delta^\alpha\;.
\end{equation*}
This completes the proof of the lemma.
\end{proof}

We turn to the proof of Theorem \ref{s27}.  We showed in Lemma
\ref{s1b} that the sequence of random fields $X_{\gamma, \varepsilon}$
converges in $L^2$ and that the limit does not depend on the mollifier
$\varrho$. We also derived the bounds claimed in the statement of the
theorem. It remains to show that the convergence takes also place in
$C^\alpha$.

In the proof of Theorem \ref{s27}, we need the following simple
facts. Assume that $0<\gamma^2< 2\pi$ and recall that
$\color{bblue} b= \gamma^2/4\pi \in (0,1/2)$. Clearly,
$5b \,-\, 4 \,<\, -\, 3b$. Let
$\color{bblue} s_0 = (1/2)[1-(\alpha/b)]$. Then,
\begin{equation}
\label{19}
2 \; \le \; s_0 \;<\; \frac{8\pi}{\gamma^2} \,-\, 2 \quad\text{if and only if}
\quad
5b \,-\, 4 \,<\, \alpha \,\le \, -\, 3b \;.
\end{equation}
The conditions on the left-hand side ensure that the ones which appear
in the statement of Corollary \ref{s3} are in force.

On the other hand, it is easy to check that $(4/15) (5 - \sqrt{10})
<1/2$ and that
\begin{equation}
\label{19b}
5b \,-\, 4 \,<\, -\, b \, \big(\, \sqrt{1 + (8/b)} \,-\, 1\,\big)
\end{equation}
provided $\gamma^2/4\pi = b < (4/15) (5 - \sqrt{10})$.  This inequality
explains the hypotheses of the theorem.

\begin{proof}[Proof of Theorem \ref{s27}]
Fix $0<\gamma^2 \,< \, (4/15) (5 - \sqrt{10})\, 4\pi$ so that
\eqref{19b} is in force, and
$5b \,-\, 4 < \alpha< b\, \min\{\, -3 \,,\, - \, \big(\, \sqrt{1 +
  (8/b)} \,-\, 1\,\big) \,\}$. Fix $T_1>$ and $\eta>0$.
Since convergence in $C^\alpha$ implies convergence in $C^{\alpha'}$
for $\alpha'<\alpha$, it is enough to show that
\begin{equation*}
\lim_{\varepsilon \to 0} \bb P \big[\, \Vert
X_{\gamma,\varepsilon} - X_\gamma \Vert_{C^\alpha([-T_1,T_1]\times \bb T)}
\,>\, \eta\, \big] \;=\; 0\;.
\end{equation*}

Let $A_{\iota, \varepsilon} = A(T_1, \iota, \varepsilon, \gamma)$,
$0\le \iota\le 3$, be given by equation \eqref{45} with $X$ replaced
by $X_{\gamma,\varepsilon} - X_\gamma$. By Lemma \ref{l01}, it is
enough to show that for all $\eta>0$,
\begin{equation*}
\lim_{\varepsilon \to 0} \bb P \big[\, A_{\iota, \varepsilon}
\,>\, \eta\, \big] \;=\; 0\;.
\end{equation*}
Denote by $B^L_{0, \varepsilon, \eta}$,
$B^n_{\iota, \varepsilon, \eta}$, $1\le \iota\le 3$, $n\ge L$, the
events defined by
\begin{gather*}
B^L_{0, \varepsilon, \eta} \;:=\;
\Big\{  \max_{m\in P_L} \max_{p}
2^{L\alpha}\, \big|\, \< X_{\gamma,\varepsilon} - X_\gamma \,,\,
S^{2^{-L}}_{(p/2^L, m/2^L)} \Phi_0  \>\, \big| \,>\, \eta \, \Big\}\;, \\
B^n_{\iota, \varepsilon, \eta} \;:=\;
\Big\{ \max_{m\in P_n} \max_{p}
2^{n\alpha}\, \big|\, \< X_{\gamma,\varepsilon} - X_\gamma \,,\,
S^{2^{-n}}_{(p/2^n, m/2^n)} \Phi_\iota  \>\, \big| \,>\, \eta \,
\Big\}\;,
\end{gather*}
where the maximum over $p$ is carried over the same set appearing in
\eqref{45}.

Clearly, for $1\le \iota\le 3$,
\begin{equation*}
\bb P \big[\, A_{\iota, \varepsilon}
\,>\, \eta \,\big]  \;\le\;
\bb P \Big[  \, \bigcup_{n\ge L} B^n_{\iota, \varepsilon, \eta} \,\Big]
\;\le\; \sum_{n \ge L}
\bb P \big[\, B^n_{\iota, \varepsilon, \eta} \,\big] \;.
\end{equation*}
Fix $n\ge L$. Let $s_0 = (1/2)[1-(\alpha/b)]$. By assumption,
$0<\gamma^2< 2\pi$ (cf. the observation just before \eqref{19b}).
Moreover, by \eqref{19}, $2\le s_0< 8\pi/\gamma^2 - 2$. The hypotheses
of Corollary \ref{s3} are therefore in force. Hence, by this result,
\begin{align*}
\bb P \big[\, B^n_{\iota, \varepsilon, \eta} \,\big] \; & \le\;
\sum_{m\in P_n} \sum_{p} \, \bb P  \Big[ \, 
2^{n\alpha} | \,\<  X_\gamma - X_{\gamma, \varepsilon}
\,,\, S^{2^{-n}}_{(p/2^n, m/2^n)} \Phi_\iota \> \, |  \, >\, \eta \,\Big ] \\
&\le\; \frac{C(s_0,\gamma)}{\eta^{s_\kappa}} \, \varepsilon^{2\kappa^2}
\sum_{m\in P_n} \sum_{p} 2^{2 n \kappa^2}\, 2^{ n \alpha s_\kappa}\, 2^{n
  (\gamma^2/4\pi) [s_0-(1/2)]^2}\;,
\end{align*}
where $s_\kappa = s_0 - \kappa(s_0-2)$.

Summing over $n\ge L$ we get that
\begin{equation}
\label{21}
\bb P \big[\, A_{\iota, \varepsilon}
\,>\, \eta \,\big]  \;\le\;  
\frac{C(s_0,\gamma, T_1)}{\eta^{s_\kappa}} \, \varepsilon^{2\kappa^2} \,
\sum_{n \ge L} \, 2^{n \kappa [2 \kappa  - \alpha (s_0-2)]}\, 
2^{[ \,2+ M(s_0)\, ] \, n} \;,
\end{equation}
where the factor $2^{2n}$ appeared to take care of the volume, and
$M(s) = \alpha \, s + (\gamma^2/4\pi)\, [s-(1/2)]^2$.

A simple computation shows that $s_0$ is the critical value of
$M(s)$. This explains our choice. On the other hand, as $\alpha < - \,
b \, (\sqrt{1 + (8/b)} -1)$, we have that $2+M(s_0) <0$.  Choose
$\kappa$ sufficiently small for $\kappa [2 \kappa - \alpha (s_0-2)] +
2 + M(s_0)$ to be negative. In this case, the sum appearing in
\eqref{21} is finite. Hence, for all $\eta>0$, $1\le \iota\le 3$,
\begin{equation*}
\bb P \big[\, A_{\iota, \varepsilon}
\,>\, \eta \,\big]  \;\le\;  C(s_0, \kappa, \eta, \gamma, T_1) \,
\varepsilon^{2\kappa^2}\;.
\end{equation*}
The probability of the event $\{A_{0, \varepsilon}
\,>\, \eta \}$ is easier to estimate as it does not involve the sum over
$n$. This completes the proof of the theorem.
\end{proof}

\begin{remark}
\label{rm3}
Let $\ms L_b = 5b-4$, $\ms R_b = b\, \min\{\, -3 \,,\, - \, \big(\,
\sqrt{1 + (8/b)} \,-\, 1\,\big) \,\}$, where $b= \gamma^2/4\pi$. With
this notation, the hypothesis on $\alpha$ of Theorem \ref{s27} becomes
$\alpha \in (\ms L_b , \ms R_b)$.  A straightforward computation shows
that $(\ms L_b , \ms R_b) \cap (-1/2,0) \not = \varnothing$ if
$0<\gamma^2< \pi/7$.
\end{remark}

\begin{remark}
\label{rm2}
By applying H\"older's inequality in the proofs of Lemma \ref{s26b}
and Corollary \ref{s3}, we lost some factors. If we could prove
Lemmata \ref{s2} or \ref{s26b} for real-valued $p$'s (instead of
integer-valued), we could improve Theorem \ref{s27} by extending the
ranges of $\gamma$ and $\alpha$.
\end{remark}

\smallskip

We conclude this section with the proof of a bound needed in Lemma
\ref{s2}.  Fix $0<a<2$ and denote by $Z(p,r)$, $p\ge 1$, $r>0$,
the truncated partition function given by
\begin{equation*}
Z(p,r) \;:=\; \int_{\bb B(0,r)^p} dz_1 \cdots  dz_p\,
\prod_{1\le i<j\le  p} \frac 1{ \Vert z_j-z_i \Vert^{a} }\;\cdot
\end{equation*}
Here and below, $\bb B(0,r)^p$ is considered as a subset of
$\bb R^{2p}$ (and not of $(\bb R \times \bb T)^{p}$). The integral on
the right-hand side can be rewritten as
\begin{align*}
& \int_{\bb B(0,r)} dz_1 \int_{\bb B(0,r)^{p-1}} dz_2 \cdots  dz_p\,
\prod_{1\le i<j\le  p} \frac 1{ \Vert z_j-z_i \Vert^{a} } \\
&\quad \le\; 
\int_{\bb B(0,r)} dz_1 \int_{\bb B(z_1,2r)^{p-1}} dz_2 \cdots  dz_p\,
\prod_{1\le i<j\le  p} \frac 1{ \Vert z_j-z_i \Vert^{a} }
\;\cdot
\end{align*}
After a change of variables, this last expression becomes
\begin{equation*}
\int_{\bb B(0,r)} dz_1 \int_{\bb B(0,2r)^{p-1}} dz_2 \cdots  dz_p\,
f(z_2, \dots, z_p) 
\end{equation*}
where
\begin{equation*}
f(w_1, \dots, w_q) \;=\; \prod_{j=1}^q \frac 1{|w_j|^a}
\, \prod_{1\le j<k \le q} \frac 1{|w_j-w_k|^a}\; \cdot
\end{equation*}
Hence, for all $p\ge 2$,
\begin{equation}
\label{51}
Z(p,r) \;\le\; C(r)
\int_{\bb B(0,2r)^{p-1}} dz_2 \cdots  dz_p\,
f(z_2, \dots, z_p)
\end{equation}
for some finite constant $C(r)$. The previous integral is estimated
below.

The proof of the next result is similar to the one of
\cite[Lemma A3]{Garban18}.

\begin{lemma}
\label{s05}
Fix $a<2$. For every $p\in \bb N$ such that $p<(4/a)-1$ and
$0<r<\infty$,
\begin{equation*}
W(p,r) \;:=\; 
\int_{\bb B(0,r)^p} f(z_1, \dots, z_p) \; dz_1 \cdots  dz_p\,
\;< \; \infty \;.
\end{equation*}
\end{lemma}

\begin{proof}
We proceed by induction. Since $a<2$, for every $r>0$, there
exists a finite constant $C(r)$ such that
\begin{equation*}
W(1,r) \;=\; \int_{\bb B(0,r)} f(z_1)  \, dz_1 \;\le \; C(r) \;.
\end{equation*}

Fix $p\in \bb N$ such that $p<(4/a)-1$ and assume that
$W(q,r) < \infty$ for all $1\le q<p$, $r>0$. For $\epsilon >0$, let
\begin{equation*}
I_{p, r, \epsilon} \;=\; 
\int_{N_{r, \epsilon}} f(z_1, \dots, z_p) \,  dz_1\cdots dz_p  \;,
\end{equation*}
where
$N_{r, \epsilon} = \{(z_1, \dots, z_p): \epsilon \le |z_j| \le r \,,\,
|z_j-z_k|\ge \epsilon\}$. Write $I_{p, r, \epsilon}$ as
\begin{equation*}
I_{p, r, \epsilon} \;=\; I_{p, r/2, \epsilon} \;+\; J_{p, r, \epsilon} \;.
\end{equation*}
This equation defines $J_{p, r, \epsilon}$. We estimate both terms on the
right hand side separately.

On the one hand, $I_{p, r/2, \epsilon} \le I_{p, r/2, \epsilon/2}$. A change
of variables yields that
\begin{equation*}
I_{p, r/2, \epsilon/2} \;=\; 2^{p[a (p+1)-4]/2} \, I_{p, r, \epsilon} \;.
\end{equation*}
As $p+1< 4/a$, the previous expression is bounded by
$(1-b)I_{p, r, \epsilon}$ for some $b>0$, $b=b(a, p)$.

On the other hand, the integral defining $J_{p, r, \epsilon}$ is
performed over the set
$M_{r, \epsilon} = N_{r, \epsilon} \cap \{(z_1, \dots, z_p) : r/2 \le
\max_j |z_j|\le r \}$.
In particular, there is a non-empty subset $B$ of $\{1, \dots p\}$
such that $|z_k|\ge r/2p$ and $|z_k - z_j|\ge r/2p$ for all $k\in B$,
$j\in A= B^c$. Note that $B$ might be equal to $\{1, \dots p\}$. In
this case, we only have that $|z_k|\ge r/2p$ for all $1\le k\le p$.

Estimating $|z_k|^{-a}$ and $|z_k-z_j|^{-a}$ by
$(2p/r)^a$, $j\in A$, $k\in B$, we obtain that
\begin{equation*}
J_{p, r, \epsilon} \; \le \; \sum_B  \Big( \frac{2p}{r}\Big)^{(A+1)B}
\int_{\bb B(0,r)^A}  f(z_A) \,  dz_A \int_{\bb B(0,r)^B}  
\prod_{i\not = j\in B} \frac 1{ \Vert z_j-z_i \Vert^{a} }
\,  dz_B  \;,
\end{equation*}
where the sum is performed over all non-empty subsets $B$ of
$\{1, \dots p\}$ and $z_A = (z_j : j\in A)$, $z_B = (z_k : k\in B)$.
The first integral is equal to $W(|A|,r)$ which is finite by the
induction hypothesis because $B$ is non-empty. The second integral is
equal to $Z(|B|,r)$ which, by \eqref{51}, is bounded by
$C(r) W(|B|-1,2r)$. Hence, by the induction assumption,
\begin{equation*}
J_{p, r, \epsilon} \; \le \; C(r) \sum_B  \Big( \frac{2p}{r}\Big)^{(A+1)B}
W(|A|,r) \, W(|B|-1,2r) \;\le\; C(r,p) 
\end{equation*}
for some finite constant $C(r,p)$.

Putting together the previous estimates yields that
\begin{equation*}
I_{p, r, \epsilon} \;\le\; (1-b)I_{p, r, \epsilon} \;+\; C(r,p) 
\end{equation*}
for some finite constant $C(r,p)$ and $b>0$, $b=b(a, p)$. Thus,
$I_{p, r, \epsilon} \le C(r,p)$ for all $\epsilon>0$. Letting
$\epsilon\to 0$, we conclude that $W(p,r)<\infty$, as claimed.
\end{proof}

Next result follows from \eqref{51} and the previous lemma.

\begin{corollary}
\label{s12}
Fix $a<2$. For every $p\in \bb N$ such that $p<4/a$ and
$0<r<\infty$,
\begin{equation*}
Z(p,r) \;<\; \infty \;.
\end{equation*}
\end{corollary}

\section{Proof of Theorem \ref{mt2}}
\label{sec06}

We prove in this section Theorem \ref{mt2}.  To avoid an additional
term, we prove Theorem \ref{mt2} for the equation \eqref{09} with the
hyperbolic sinus replaced by the exponential. The arguments presented
below apply without modifications to the original equation.

Fix $T_0\ge 1$ as in Section \ref{sec04}, and recall the definition of
the Gaussian random fields $\mf v_\varepsilon$, $\mf v$ introduced in
\eqref{08}. Fix $\beta_0>0$, $u_0$ in $C^{\beta_0}(\bb T)$,
$\gamma\in \bb R$, and denote by $\mf u_\varepsilon$,
$0< \varepsilon <1$, the solution of
\begin{equation}
\label{63}
\begin{cases}
\partial_t \mf u_\varepsilon \;=\; - \,  (-\Delta)^{1/2} \mf u_\varepsilon
\;-\; A(\varrho) \, \varepsilon^{\gamma^2/4\pi} \,
e^{\gamma \mf u_\varepsilon} \,+\, \xi_\varepsilon \\
\mf  u_\varepsilon(0) \,=\, u_0 \,+\, \mf
v_\varepsilon (0)\;.
\end{cases}
\end{equation}
Let $\color{bblue} \mf w_\varepsilon = \mf
u_\varepsilon - \mf v_\varepsilon$.  An elementary computation yields
that
\begin{equation}
\label{64}
\left\{
\begin{aligned}
& \partial_t \mf w_\varepsilon \;=\;
- \,  (-\Delta)^{1/2} \mf w_\varepsilon
\,-\,  X_{\gamma, \varepsilon} \, e^{\gamma \mf w_\varepsilon}
\;-\; R_\varepsilon \;, \\
& \mf w_\varepsilon(0) \,=\, u_0  \;,
\end{aligned}
\right.
\end{equation}
where $X_{\gamma, \varepsilon}$ is the random field
$A(\varrho) \, \varepsilon^{\gamma^2/4\pi} e^{\gamma \mf
  v_\varepsilon}$
examined in the previous section, and $R_\varepsilon$ the one given by
\begin{equation}
\label{84}
R_\varepsilon \;=\; \partial_t \mf v_\varepsilon
\;+\; (-\Delta)^{1/2} \mf v _\varepsilon \;-\; \xi_\varepsilon\;.
\end{equation}
We denote by $\mf w$ the solution of the same equation with
$X_{\gamma, \varepsilon}$, $R_\varepsilon$ replaced by $X_\gamma$,
$R=\partial_t \mf v \;+\; (-\Delta)^{1/2} \mf v \;-\; \xi$,
respectively.

The solutions $\mf w_\varepsilon$ can be represented as
\begin{equation}
\label{79}
\mf w_\varepsilon (t) \;=\; -\, \int_0^t 
P_{t-s} \Big\{ X_{\gamma, \varepsilon} \, e^{\gamma \mf w_\varepsilon (S)}
\,+\, R_\varepsilon (s) \Big\} \; ds \;+\; P_t u_0\;,
\end{equation}
where $(P_t : t\ge 0)$ represents the semigroup of the generator
$-\, (-\Delta)^{1/2}$. 

Theorem \ref{t02} below, a fixed point theorem, establishes that this
equation, for $0\le \varepsilon \le 1$, has a unique solution in
$C^{\beta_0}([0,T] \times \bb R)$ for $T$ sufficiently small.

We first recall Theorem 2.52 in \cite{bcd11}, which permits to define
the product of two distributions provided they are not too irregular.


\begin{proposition}
\label{s07}
Fix $\alpha_0$, $\beta_0\in \bb R$ such that $\alpha_0 + \beta_0 >0$,
and $S<T$. Then, there exists a bilinear form
$B:C^{\alpha_0}([S,T] \times \mathbb{T} )\times C^{\beta_0}([S,T] \times
\mathbb{T} ) \rightarrow C^{\alpha_0 \wedge \beta_0}([S,T] \times \mathbb{T}
)$ such that $B(f,g) = f\, g$ if $f$ and $g$ belong to
$C^\infty([S,T] \times \bb T)$. Moreover, there is a finite constant
$C_0=C_0(S,T)$ such that
\begin{equation*}
\|B(f,g) \|_{C^{\alpha_0\wedge \beta_0}([S,T]\times \mathbb{T} )}
\; \le \; C_0 \, \|f \|_{C^{\alpha_0}([S,T]\times \mathbb{T} )}
\, \|g \|_{C^{\beta_0}([S,T]\times \mathbb{T} )} 
\end{equation*}
for all $f\in C^{\alpha_0}(\bb R \times \mathbb{T})$,
$g \in C^{\beta_0}(\bb R \times \mathbb{T} )$.
\end{proposition}

\begin{remark}
\label{rm5}
We apply below this proposition to a distribution $X$ in $C^\alpha$
and to a function $\mf w$ in $C^\beta$. This explains the hypothesis
below that $\alpha >-1/2$ which yields that $\alpha + \beta>0$.
\end{remark}

Recall the definition of the function $q = q_{T_0}$ introduced in
\eqref{06}. Fix $\alpha_0<0$, $\beta_0>0$, such that $\alpha_0 +
\beta_0 >0$, $R$ in $C^1(\bb R \times \mathbb{T})$, $u$ in
$C^{\beta_0}(\mathbb{T})$, $X$ in $C^{\alpha_0}(\bb R \times \mathbb{T})$,
$\gamma\in\bb R$ and $0<T_1< \min\{T_0, 1/4\}$.  For $\mf w$ in
$C^{\beta_0}([0,T_1] \times \mathbb{T})$ such that $\mf w(0, \cdot) =
u(\cdot)$, let $\Psi_{T_1,\gamma,X,R,u} (\mf w) = \Psi(\mf w)$ be
given by
\begin{equation*}
\Psi (\mf w) (t) \;:=\;
\int_{0}^t q_{t-s}  \big\{\, X(s) \, e^{\gamma \mf w(s)} \,\big\} \; ds \,+\,
\int_{0}^t P_{t-s} \, R(s)\; ds \;+\; P_t u\;, \quad 0\le t\le T_1\;.
\end{equation*}
Note that $\Psi(\mf w)(0, \cdot) = u(\cdot)$. Sometimes we write
$\Psi(\mf w)$ as $\Psi_{T_1}(\mf w)$ to stress its dependence on
$T_1$.

Denote by $\mf x$ the first term on the right-hand side. It has to be
understood as follows. Extend the definition of $\mf w$ to
$\bb R\times \bb T$ by setting $\mf w(t,x)=\mf w(T_1,x)$ for
$t\ge T_1$ and $\mf w(t,x)=\mf u(x)$ for $t\le 0$. Denote the extended
function by $\widehat {\mf w}$. It is clear that $\widehat {\mf w}$
belongs to $C^{\beta_0}(\bb R \times \mathbb{T})$ and that
\begin{equation}
\label{81}
\|\, \widehat {\mf w} \, \|_{L^{\infty}([S,T]\times \mathbb{T} )} \;=\;
\|\, \mf w\, \|_{L^{\infty}([S',T']\times \mathbb{T} )}\;, \quad
\|\, \widehat {\mf w} \, \|_{C^{\beta_0}([S,T]\times \mathbb{T} )} \;=\;
\|\, \mf w\, \|_{C^{\beta_0}([S',T']\times \mathbb{T} )}
\end{equation}
for all $S<T \wedge T_1$, $T>0$, where $S' = S \vee 0$, $T' = T \wedge
T_1$.  We may also replace $\mf w$ by $\widehat {\mf w}$ in the
formula for $\Psi$ because they coincide on $[0,T_1]\times \bb T$.

By Assertion \ref{a01} below, $\exp\{\gamma \widehat {\mf w} \}$
belongs to $C^{\beta_0}(\bb R \times \mathbb{T})$. Hence, by Proposition
\ref{s07},
$\color{bblue} X_{\mf w} := X\, \exp\{\gamma \widehat {\mf w} \}$
belongs to $C^{\alpha_0}(\bb R \times \mathbb{T})$. 

As $q$ vanishes for $(-\infty, 0) \times \bb T$, we may include in the
domain of integration the time-interval $[t,\infty)$.  As $X$ belongs
to $C^{\alpha_0}(\bb R \times \bb T)$,
\begin{equation*}
\int_{0}^\infty q_{t-s}  \, X_{\mf w}(s) \; ds \;=\;
X_{\mf w} \big(\, q_{T_0, z}\, \chi_{\bb R_+ \times \bb T}\,\big)
\;=\; X^+_{\mf w} \big(\, q_{T_0, z}\,\big)\;, 
\end{equation*}
where, recall, $\chi_A$ represents the indicator function of the set
$A$, $q_{T_0,z}$ has been introduced just before the statement of
Theorem \ref{t01}, and the distribution $ X^+_{\mf w}$ at the end of
Section \ref{sec04}.

From now on, $\alpha$, $\beta$ and $\kappa$ are fixed.  We first pick
$-1/2<\alpha<0$ and then choose $\kappa$ small enough for
$0<2\kappa<1+2\alpha$. Let $\beta= \alpha + 1 - 2\kappa$.  Note that
$0< \beta <1$ and $\alpha + \beta>0$: On the one hand,
$\beta > \beta+\alpha = 1 + 2\alpha - 2\kappa >0$. On the other,
$\beta = 1 + \alpha - 2\kappa <1$. In Theorem \ref{t03} we further
require $\alpha < \alpha_\gamma$.

All constants below may depend on these parameters without any
reference. In contrast, any dependence on other variables will be
explicitly mentioned.

\begin{theorem}
\label{t02}
Fix $T_0>0$, $\gamma\in\bb R$, $R$ in $C^1(\bb R \times \mathbb{T})$,
$u$ in $C^{\beta}(\mathbb{T})$, $X$ in $C^\alpha(\bb R \times
\mathbb{T})$. There exists $0<\tau< \min\{T_0, 1/4\}$ such that the
equation
\begin{equation}
\label{83}
\Psi_T (\mf w) \;=\; \mf w
\end{equation}
has a unique solution in $C^\beta([0,T] \times \mathbb{T})$ for all
$0<T\le \tau$.
\end{theorem}

Recall the formula for $\alpha_\gamma$ introduced just above Theorem
\ref{mt2}. By definition, $\alpha_\gamma<0$ and, according to Remark
\ref{rm3}, $\alpha_\gamma>-1/2$ if $0<\gamma^2 < \pi/7$.

\begin{theorem}
\label{t03}
Fix $0<\gamma^2 < \pi/7$, $\alpha \in (-1/2, \alpha_\gamma)$,
$T_0>0$, and $u$ in $C^\beta(\bb T)$. There exists a strictly positive
random variable $\tau$, $\bb P[\tau>0]=1$, satisfying the next
statement.

Denote by $\mf w_\varepsilon$, $0\le \varepsilon\le 1$, the solution
of the fixed point problem \eqref{83} in
$C^\beta([0,\tau] \times \bb T)$, with $R_\varepsilon$ given by
\eqref{84} and
$X_{\gamma, \varepsilon} = A(\varrho) \, \varepsilon^{\gamma^2/4\pi}
e^{\gamma \mf v_\varepsilon}$.  Then, $\mf w_\varepsilon$ converges in
probability to $\mf w_0=\mf w$, as $\varepsilon\to 0$, in
$C^\beta([0,\tau] \times \bb T)$.
\end{theorem}

\begin{proof}[Proof of Theorem \ref{mt2}]
Let $\tau$ be the a. s. strictly positive random time given by Theorem
\ref{t03}.  The solution $\mf u_\varepsilon$ of \eqref{09} can be
represented as $\mf v_\varepsilon + \mf w_\varepsilon$. According to
Theorem \ref{t03}, $\mf w_\varepsilon$ converges in probability to
$\mf w$, as $\varepsilon\to 0$, in $C^\beta([0,\tau] \times \bb
T)$. On the other hand, it is not difficult to show that
$\mf v_\varepsilon$ converges in probability to $\mf v$, as
$\varepsilon\to 0$, in $C^\beta([0,\tau] \times \bb T)$.  Since
neither $\mf w$ nor $\mf v$ depend on the mollifier $\varrho$, the
theorem is proved.
\end{proof}

\begin{proposition}
\label{s06} 
Fix $0<T_1< \min\{T_0, 1/4\}$.  For all $\mf w \in C^{\beta}([0,T_1]
\times \mathbb{T})$, the function $\Psi_{T_1}(\mf w)$ belongs to
$C^{\beta}([0,T_1] \times \mathbb{T})$. Moreover, there exist finite
constants $A_1 = A_1(T_0, \| \, u\,\|_{ C^{\beta}(\mathbb{T})})$, $A_2
= A_2(\gamma, T_0, \| \, u\,\|_{ L^{\infty}(\mathbb{T})})$, $A_3 =
A_3(\gamma, T_0)$ such that
\begin{equation*}
\begin{aligned}
\Vert \Psi(\mf w) \Vert_{C^{\beta}([0,T_1]\times \mathbb{T})}
\; & \le\;  A_1 \, \big(\, 1 \,+\,  \Vert R \Vert_{C^1([0,T_0]\times \bb
  T)} \, \big) \\
\qquad 
\; &+\; A_2\, T^\kappa_1\, \Vert X \Vert_{C^\alpha([-T_0-3,5]\times
  \bb T)}\, \exp \Big\{\, A_3\,
\Vert\, \mf w \, \Vert_{C^\beta
([0,T_1]\times \bb T)}\, \Big\}  \;.
\end{aligned}
\end{equation*}
\end{proposition}

\begin{proof}
We examine separately each term appearing in the
definition of $\Psi (\mf w)$. We start with $P_t u$. By Lemma
\ref{s03}, $P_t \, u$ belongs to $C^{\beta}(\bb R \times \mathbb{T})$,
and there exists a finite constant $C_0$, depending only on $\beta$,
such that
\begin{equation*}
\| \, P_t \, u \, \|_{ C^{\beta}([0,T_1] \times
\mathbb{T})} \le C_0 \, \| \, u\, \|_{ C^{\beta}(\mathbb{T})} \;.
\end{equation*}

We turn to the term involving $R(s)$. Let
$\mf x(t) = \int_0^t P_{t-s} R(s)\, ds$.  Let
$M_0 = \Vert \, (\partial_x R)\, \Vert_{L^\infty([0,T_0]\times \bb
  T)}$,
\begin{equation*}
M_1 \;=\; M_0 \; T_0 \, \big\{\, 1 \,+\,
E\big[\, |Z_1|^\beta\,\big] \,\big\} \;+\; T_0^{1-\beta}\, \Vert
R \Vert_{L^\infty([0,T_0]\times \bb T)} \;.
\end{equation*}
A computation, similar to the one presented in the proof of Lemma
\ref{s03}, yields that
\begin{equation*}
\big | \,  \mf x(t,x) \,-\, \mf x(t',x') \, \big | \;\le\;
M_1 \, \big \{\, |x'-x|^\beta \,+\, |t'-t|^\beta \, \big\}
\end{equation*}
for all $t$, $t'\in [0,T_1]$, $x$, $x'\in \bb T$. We used here the
fact that $T_1\le T_0$.  This proves that $\mf x$ belongs to
$C^{\beta}([0,T_1] \times \mathbb{T})$ and that
$\Vert \mf x \Vert_{C^{\beta}([0,T_1]\times \mathbb{T})} \,\le\,
M_1$.

Finally, let
$\mf x (z) = \int_{0}^t q_{t-s} [ \, X(s) \, e^{\gamma \mf w(s)} \, ]
\, ds = X^+_{\mf w}(q_{T_0,z})$.  Since $X_{\mf w}$ belongs to
$C^{\alpha}(\bb R \times \mathbb{T})$, as
$\beta = 1+ \alpha - 2\kappa$, by Corollary \ref{s04}, $\mf x$,
belongs to $C^{\beta}([0,T_1] \times \mathbb{T})$ and there exists a
finite constant $M_3=M_3(T_0)$, whose value may change from line to
line, such that
\begin{equation*}
\Vert \mf x \Vert_{C^{\beta}([0,T_1]\times \mathbb{T})} \;\le\; M_3 \,
T^\kappa_1 \, \Vert  X_{\mf w} \Vert_{C^{\alpha}([-T_0-3,5]\times
  \mathbb{T})}\;.
\end{equation*}
By Proposition \ref{s07} and \eqref{66}, $\Vert  X_{\mf w}
\Vert_{C^{\alpha}([-T_0-3,5]\times \mathbb{T})}$ is less than or equal to
\begin{equation*}
M_3\, \Vert  X \Vert_{C^{\alpha}([-T_0-3,5]\times \mathbb{T})}
\, \exp \Big\{\, |\gamma| \, \Vert\, \widehat {\mf w} \, \Vert_{L^\infty
  ([-T^*_0,T^*_0]\times \bb T)}\, \Big\} \; |\gamma| \,
\Vert\, \widehat {\mf w} \, \Vert_{C^\beta
  ([-T^*_0,T^*_0]\times \bb T)} \;,
\end{equation*}
where $T^*_0 = \max\{T_0+3, 5\}$. By definition of $\widehat {\mf w}$
and \eqref{81}, we may replace $\widehat {\mf w}$ by $\mf w$ and the
interval $[-T^*_0,T^*_0]$ by $[0,T_1]$.

Let $M_4 = M_3 \, \Vert  X \Vert_{C^{\alpha}([-T_0-3,5]\times
  \mathbb{T})}$, use the bound $a\le e^a$, $a>0$, and apply the
inequality \eqref{70} below to bound the previous expression by
\begin{equation*}
M_5\, \exp \Big\{\, |\gamma| \, (1+T^\beta_0)\,
\Vert\, \mf w \, \Vert_{C^\beta
([0,T_1]\times \bb T)}\, \Big\} \;,
\end{equation*}
where
$M_5 = M_4 \, \exp \{\, |\gamma| \, \Vert\, u \, \Vert_{L^\infty (\bb
  T)}\, \}$.

To complete the proof of the proposition, it remains to recollect the
previous estimates.
\end{proof}

Next result asserts that the function $\Psi_{T,\gamma,X,R,u}$ depends
continuously on the parameters $X$ and $R$.

\begin{lemma}
\label{s09}
Fix $0<T_1< \min\{T_0, 1/4\}$, $\gamma\in \bb R$, $u\in
C^\beta(\bb T)$. There exists a finite constant $A_4 = A_4(T_0,
\gamma, u)$ such that
\begin{align*}
& \big \Vert\, \Psi_{T_1, \gamma, X, R, u}(\mf w) \,-\, \Psi_{T_1, \gamma,
  X', R', u}(\mf w)
\,\big\Vert_{C^\beta([0,T_1]\times \bb  T)} 
\;\le\; A_4 \, 
\big \Vert\, R \,-\, R' \,\big\Vert_{C^1([0,T_0]\times \bb  T)} \\
&\quad \;+\; A_4\, T^\kappa_1\, \big \Vert\, X \,-\, X' \,\big\Vert_{C^\alpha([-T_0-3,5]\times \bb  T)}
\, \exp \Big\{\, A_4\, \Vert\, \mf w \, \Vert_{C^\beta ([0,T_1]\times \bb T)}\, \Big\}
\end{align*}
for all $X$, $X'$ in $C^\alpha(\bb R\times \bb T)$, $R$, $R'$ in
$C^1(\bb R\times \bb T)$, and $\mf w$ in $C^\beta([0,T_1] \times \bb
T)$.
\end{lemma}

\begin{proof}
We estimate the difference term by term. We start with
the one involving $R$.  For a function $U$ in $C^1(\bb R \times \bb
T)$, let $\mf x_U(t) = \int_0^t P_{t-s} U(s)\, ds$.  The  proof of
Proposition \ref{s06} yields that there 
exists a finite constant $M_0=M_0(T_0)$ such that $\Vert \mf x_R - \mf
x_{R'} \Vert_{C^{\beta}([0,T]\times \mathbb{T})} \,\le\, M_0 \, \big
\Vert\, R \,-\, R' \,\big\Vert_{C^1([0,T_0]\times \bb T)}$.

For $Y$ in $C^\alpha(\bb R \times \bb T)$, let
$\mf x_Y (z) = \int_{0}^\infty q_{t-s} [ \, Y(s) \, e^{\gamma \mf
  w(s)} \, ] \, ds = Y^+_{\mf w}(q_{T_0,z})$. By the proof of
Proposition \ref{s06}, there exists a constant
$M_1=M_1(T_0, \gamma, u)$ such that
$\Vert \, \mf x_X - \mf x_{X'}\, \Vert_{C^{\beta}([0,T_1]\times
  \mathbb{T})}$ is bounded above by
\begin{equation*}
M_1\, T^\kappa_1\, \big \Vert\, X \,-\, X' \,\big\Vert_{C^\alpha([-T_0-3,5]\times \bb  T)}
\, \exp \Big\{\, M_1\, \Vert\, \mf w \, \Vert_{C^\beta ([0,T_1]\times \bb T)}\, \Big\}\;.
\end{equation*}
The assertion of the lemma follows from the two previous estimates.
\end{proof}

The next result asserts that $\Psi$ is a contraction provided the
time-interval is small enough.  It follows from the third part of the
proof of Proposition \ref{s06} and from Assertion \ref{a02} below.

\begin{lemma}
\label{s08}
Fix $0<T_1< \min\{T_0, 1/4\}$, $\gamma\in \bb R$, $u\in
C^\beta(\bb T)$. There exists a finite constant $A_5 = A_5(T_0,
\gamma, u)$ such that
\begin{equation*}
\begin{aligned}
& \big \Vert\, \Psi(\mf w_1) \,-\, \Psi(\mf w_2)
\,\big\Vert_{C^\beta([0,T_1]\times \bb  T)} \;\le\;
A_5 \, T^\kappa_1\, 
\Vert  \,X \,\Vert_{C^{\alpha}([-T_0-3,5]\times  \mathbb{T})} \; \times\\
&\quad \times \; \, 
\exp A_5\, \big\{\, \Vert\, \mf w_1 \,\big\Vert_{C^\beta([0,T_1]\times \bb T)} 
+ \Vert\, \mf w_2 \,\big\Vert_{C^\beta([0,T_1]\times \bb  T)} \big\}
\, \big \Vert\, \mf w_1 \,-\, \mf w_2
\,\big\Vert_{C^\beta([0,T_1]\times \bb  T)}
\end{aligned}
\end{equation*}
for all $X$ in $C^\alpha(\bb R\times \bb T)$, $R$ in $C^1(\bb
R\times \bb T)$, and $\mf w_1$, $\mf w_2$ in $C^\beta([0,T_1] \times
\bb T)$ such that $\mf w_k(0, \cdot) = u(\cdot)$, $k=1$, $2$.
\end{lemma}

Let $B_{K_1}=B(\gamma, T_0, u, K_1)$, and
$\tau_{K_1, K_2,B} = \tau (\gamma, T_0, u, K_1, K_2, B)$, $K_1$,
$K_2$, $B>0$ be given by
\begin{equation*}
B_{K_1} \;=\; 1\;+\; A_1 \, \big(\, 1 \,+\,
K_1 \, \big)\;, \quad \tau_{K_1, K_2,B}
\,=\, \tau_1 \wedge \,\tau_2\;, 
\end{equation*}
where $A_1$, $A_2$, $A_3$, $A_5$ are the constants appearing in the
statement of Proposition \ref{s06} and Lemma \ref{s08} and
\begin{equation*}
\tau_1^{-\kappa} \;=\; A_2\, K_2 \, e^{A_3 \, B} \;,\;\;
\tau_2^{-\kappa} \;=\; 2\, A_5\, K_2 \, e^{2\, A_5 \, B} \;.
\end{equation*}

Next result is a straightforward consequence of Proposition \ref{s06}
and Lemma \ref{s08}. It asserts that $\Psi$ is a contraction from the
ball of radius $B$ in the $C^\beta([0,T]\times \bb T)$-topology to
itself provided $T\le \tau$.

\begin{lemma}
\label{s16}
Fix $T_0>0$, $\gamma\in \bb R$, $u\in C^\beta(\bb T)$, $K_1>0$,
$K_2>0$ and $B\ge B_{K_1}$.  Then, $\big \Vert\, \Psi(\mf w)
\,\big\Vert_{C^\beta([0,T]\times \bb T)} \le B$ if $\big \Vert\, \mf w
\,\big\Vert_{C^\beta([0,T]\times \bb T)} \le B$ and
\begin{equation*}
\big \Vert\, \Psi(\mf w_1) \,-\, \Psi(\mf w_2)
\,\big\Vert_{C^\beta([0,T]\times \bb  T)} \;\le\; (1/2)
\, \big \Vert\, \mf w_1 \,-\, \mf w_2
\,\big\Vert_{C^\beta([0,T]\times \bb  T)}
\end{equation*}
for all  $0<T\le \tau_{K_1, K_2,B}$, $R \in C^1(\bb R\times \bb T)$ such that
$\Vert R \Vert_{C^1([0,T_0]\times \bb T)} \le K_1$,
$X \in C^\alpha(\bb R\times \bb T)$ such that
$\Vert \, X_\gamma \, \Vert_{C^\alpha([-T_0-3,5] \times \bb T)} \le
K_2$, $\mf w_1$, $\mf w_2$ in $C^\beta([0,T]\times \bb T)$ such that
$\mf w_k(0, \cdot) = u(\cdot)$, $\big \Vert\, \mf w_k
\,\big\Vert_{C^\beta([0,T]\times \bb T)} \le B$, $k=1$, $2$.
\end{lemma}

\begin{proof}[Proof of Theorem \ref{t02}]  The result follows from
  the previous lemma and a fixed point theorem in Banach spaces.
\end{proof}

\begin{proof}[Proof of Theorem \ref{t03}]
Fix $0<\gamma^2 < \pi/7$ and
$\alpha \in (-1/2, \alpha_\gamma)$. By Theorem
\ref{s27}, $X_{\gamma, \varepsilon}$ converges in probability to
$X_{\gamma}$ in $C^\alpha$, so that
$\Vert \, X_\gamma \, \Vert_{C^\alpha([-T_0-3,5] \times \bb T)}$ is
almost surely finite.  On the other hand, by Proposition \ref{s13}
below, $\Vert R \Vert_{C^1([0,T_0]\times \bb T)}$ is almost surely
finite and
$\Vert\, R \,-\, R_\varepsilon \, \Vert_{C^1([0,T_0]\times \bb T)}$
converges to $0$ in probability.

Fix $0<\zeta\le 1$, $\eta>0$. It follows from the previous
observations that there exists $K>0$ and $\varepsilon_0>0$ such that
\begin{equation*}
\begin{aligned}
& \bb P \big[ \, \Vert R \Vert_{C^1([0,T_0]\times \bb T)} >
K\, \big] \;\le\; \eta \;, \quad
\bb P \big[ \, \Vert \, X_{\gamma} \,
\Vert_{C^\alpha([-T_0-3,5] \times \bb T)} > K\, \big]
\;\le\; \eta \; , \\
&\quad
\bb P \big[ \, \Vert R_\varepsilon \,-\, R
\Vert_{C^1([0,T_0]\times \bb T)} >
\zeta \, \big] \;\le\; \eta \;, \quad
\bb P \big[ \, \Vert \,
X_{\gamma, \varepsilon} \,-\, X_{\gamma} \,
\Vert_{C^\alpha([-T_0-3,5] \times \bb T)} > \zeta \, \big]
\;\le\; \eta \; ,
\end{aligned}
\end{equation*}
for all $0\le \varepsilon < \varepsilon_0$. Note that we included
$\varepsilon =0$.

Denote by $\Omega_{K,\zeta}$ the union of the four sets appearing in
the previous displayed formula. Let $B=B_{K+1} = 1 + A_1(2+K)$,
$\tau = \tau_{K+1,K+1,B}$. On the set $\Omega^c_{K,\zeta}$, by Theorem
\ref{t02},
$\Vert \, \mf w_\varepsilon\, \Vert_{C^\beta([0,\tau] \times \bb T)}
\le B$ for all $0\le \varepsilon \le \varepsilon_0$.

We claim that on the set $\Omega^c_{K,\zeta}$
\begin{equation}
\label{82}
\Vert \, \mf w \,-\, \mf w_\varepsilon\, \Vert_{C^\beta([0,\tau] \times \bb  T)} 
\;\le \; A\, \zeta\, e^{AB}
\end{equation}
for some constant $A = A(\gamma, T_0, u)$. 

To prove this claim, let $\Psi = \Psi_{X,R,u}$, $\Psi_\varepsilon =
\Psi_{X_\varepsilon,R_\varepsilon,u}$.  Since $\mf w$, $\mf
w_\varepsilon$ are solutions of the fixed point problem stated in
Theorem \ref{t02},
\begin{equation*}
\begin{aligned}
& \Vert \, \mf w \,-\, \mf w_\varepsilon\, \Vert_{C^\beta([0,\tau] \times \bb  T)} 
\;=\; \Vert \, \Psi(\mf w) \,-\, \Psi_\varepsilon(\mf w_\varepsilon)
\, \Vert_{C^\beta([0,\tau] \times \bb T)} \\ 
&\quad \;\le \; \Vert \, \Psi(\mf w) \,-\, \Psi_\varepsilon(\mf w)
\, \Vert_{C^\beta([0,\tau] \times \bb T)} \;+\;
\Vert \, \Psi_\varepsilon(\mf w) \,-\, \Psi_\varepsilon(\mf w_\varepsilon)
\, \Vert_{C^\beta([0,\tau] \times \bb T)} \;.
\end{aligned}
\end{equation*}
By Lemma \ref{s09}, on the set $\Omega^c_{K,\zeta}$, the first term is
bounded above by
\begin{equation*}
 A \,  \Big\{
\big \Vert\, R \,-\, R_\varepsilon \,\big\Vert_{C^1([0,T_0]\times \bb  T)}  \;+\;
\big \Vert\, X \,-\, X_\varepsilon \,\big\Vert_{C^\alpha([-T_0-3,5]\times \bb  T)}
\, e^{A\, B} \,\Big\} \;\le\; A\, \zeta\, e^{AB}
\end{equation*}
for some constant $A = A(\gamma, T_0, u)$. By Lemma \ref{s16} with
$K_1= K_2=K+1$, $B=B_{K+1}$, on the set $\Omega^c_{K,\zeta}$, the
second term is bounded by
$(1/2)\, \Vert\, \mf w_1 \,-\, \mf w_2
\,\big\Vert_{C^\beta([0,\tau]\times \bb T)}$.
This proves \eqref{82}.

Hence, there exists a finite constant $A = A(\gamma, T_0, u)$ with the
following property. For all $\zeta>0$, $\eta>0$, there exists
$\varepsilon_0>0$ such that for all $0\le \varepsilon < \varepsilon_0$
\begin{equation*}
  \bb P \Big[ \, \Vert \, \mf w \,-\, \mf w_\varepsilon\,
  \Vert_{C^\beta([0,\tau] \times \bb  T)} 
\;> \; A\, \zeta\, e^{AB} \, \Big] \;\le\; 4\, \eta\;.
\end{equation*}
This proves the theorem.
\end{proof}

\smallskip
We conclude this section with some elementary estimates used above.
Recall from \eqref{13} that we denote by $C^b_+(\bb R \times \bb T)$,
$0<b<1$, the elements $f$ of $C^b$ such $f(t) =0$ for $t\le 0$.

\begin{asser}
\label{a03}
Fix $0<b<1$. For all $f$ in $C^b_+(\bb R \times \bb T)$ and $T>0$,
\begin{equation*}
\big \Vert\, f \,\big\Vert_{L^\infty([0,T]\times \bb T)}
\;\le\; T^b \,
\big \Vert\, f \,\big\Vert_{C^b([0,T]\times \bb T)}\;.
\end{equation*}
\end{asser}

\begin{proof}
Fix $T>0$ and $(t,x) \in [0,T]\times \bb T$. Since $ f(0)=0$, $|\,
f(t,x)\,| \,=\, |\, f(t,x) - f(0,x) \,|$. Hence, by definition of the
norm $\Vert\, \cdot \,\Vert_{C^b([0,T]\times \bb T)}$, $|\, f(t,x)\,|$
is bounded by $T^b \, \Vert\, f \,\Vert_{C^b([0,T]\times \bb T)}$,
which proves the assertion.
\end{proof}

Fix $0<b<1$. It follows from the previous assertion that for all $f$
in $C^b(\bb R \times \bb T)$ and $T>0$,
\begin{equation}
\label{70}
\big \Vert\, f \,\big\Vert_{L^\infty([0,T]\times \bb T)}
\;\le\; \big \Vert\, f(0,\cdot) \,\big\Vert_{L^\infty(\bb T)}
\;+\; T^b \, \big \Vert\, f \,\big\Vert_{C^b([0,T]\times \bb T)}\;,
\end{equation}
where $f(0,\cdot)$ represents the restriction of the function $f$ to
$\{0\} \times \bb T$.

\begin{asser}
\label{a01}
Fix $\gamma\in\bb R$, $0<b<1$ and an element $f $ in $C^b(\bb R
\times \bb T)$. Then, $\exp\{ \gamma \, f\}$ belongs to $C^b(\bb R
\times \bb T)$ and for all $T>0$, $-T\le T_1<T_2\le T$,
\begin{equation}
\label{66}
\big \Vert\, e^{\gamma\, f}\,\big\Vert_{C^b([T_1,T_2]\times \bb T)}
\;\le\; C_3(T)\, |\gamma|\,
\big \Vert\, f \,\big\Vert_{C^b([T_1,T_2]\times \bb T)}\;,
\end{equation}
where $C_3 (T) = \exp\{\, |\gamma| \, \Vert\, f\, \Vert_{L^\infty
  ([-T,T]\times \bb T)}\, \}$. Moreover, if $f$ belongs to $C^b_+(\bb
R \times \bb T)$,
\begin{equation}
\label{68}
\big \Vert\, e^{\gamma\, f}\,\big\Vert_{C^b([0,T]\times \bb T)}
\;\le\; \exp\Big\{\, (1+T^b)\, |\gamma| \,
\big \Vert\, f \,\big\Vert_{C^b([0,T]\times \bb T)} \, \Big\}\;.
\end{equation}
\end{asser}

\begin{proof}
The first claim follows from the bound $|e^y - e^x| \le \max\{ e^{|x|}
, e^{|y|} \}\, |y-x|$, $x$, $y\in \bb R$. Now, suppose that $f$
belongs to $C^b_+(\bb R \times \bb T)$. By Assertion \ref{a03},
$\Vert\, f\, \Vert_{L^\infty ([0,T]\times \bb T)} \le \Vert\, f\,
\Vert_{C^b ([0,T]\times \bb T)} T^b$. To complete the proof of the
second assertion, it remains to recall that $a\le e^a$ for $a\ge 0$.
\end{proof}

\begin{asser}
\label{a02}
Fix $\gamma\in\bb R$, $0<b<1$ and $f$, $g$ in $C^b(\bb R \times \bb
T)$. Assume that $f(0,x) = g(0,x) = J(x)$. Then, for all $T>0$,
\begin{equation*}
\big \Vert\, e^{\gamma\, f}\, -\, e^{\gamma\, g}\,
\big\Vert_{C^b([0,T]\times \bb T)}
\;\le\; A_0 \, e^{4 \gamma A_1}
\, \big \Vert\, f \, -\, g\, \big\Vert_{C^b([0,T]\times \bb T)}\;,
\end{equation*}
where $A_0 \,=\, (1 + 2 T^b)\, \gamma$ and
\begin{equation*}
A_1 = \Vert\, J \, \Vert_{L^\infty(\bb T)} \;+\;
T^b\, \Big\{\, \Vert\, f \, \Vert_{C^b([0,T]\times \bb T)} \;+\;
\Vert\, g \, \Vert_{C^b([0,T]\times \bb T)} \,\Big\} \;.
\end{equation*}
\end{asser}

\begin{proof}
Fix $T>0$, $z$, $w\in [0,T] \times \bb T$. Let $h=f-g$, and write
$[\, \exp\{\gamma\, f (z)\}\, -\, \exp\{\gamma\, g (z)\}\,] \,-\,
[\, \exp\{\gamma\, f (w)\}\, -\, \exp\{\gamma\, g (w)\}\,]$ as
\begin{align*}
& \frac{e^{\gamma\, f (z)} \,-\, e^{\gamma\, g (z)}}{f(z) \,-\, g(z)}\,
\big[\, h(z) - h(w)\,\big]
\;+\; h(w)\,
\Big(\, e^{\gamma\, g(z)} \,-\, e^{\gamma\, g (w)} \,\Big)
\frac{e^{\gamma h(z)}-1}{h(z)} \\
&\quad
+\; h(w)\, e^{\gamma\, g (w)}
\Big(\, \frac{e^{\gamma h(z)}-1}{h(z)}
\,-\, \frac{e^{\gamma h(w)}-1}{h(w)} \,\Big) \;.
\end{align*}

Denote by $f_\infty$, $g_\infty$, $h_\infty$ the
$L^\infty([0,T] \times \bb T)$ norms of $f$, $g$ and $h$,
respectively, and by $f_b$, $g_b$, $h_b$, the
$C^b ([0,T] \times \bb T)$ norms of these three functions. Clearly,
$h$ belongs to $C^b_+ ([0,T] \times \bb T)$. Hence, by Assertion
\ref{a03}, $h_\infty \le T^b h_b$, while, by \eqref{70},
$f_\infty \le J_\infty + T^b f_b$, with a similar inequality for
$g$. Here, $J_\infty$ stands for the $L^\infty(\bb T)$ norm of $J$.

Consider separately the three terms of the previous displayed
equation. It is not difficult to show that the first one is bounded by
$\gamma\, \exp\{ \gamma\, [g_\infty + f_\infty]\} \, h_b \, \Vert z -
z'\Vert^b$, and the second one by $\gamma^2\, \exp\{ \gamma\,
[g_\infty + h_\infty]\} \, h_\infty \, g_b \, \Vert z - z'\Vert^b$.
Let $f: \bb R \to \bb R$ be given by $f(\theta) = (e^\theta
-1)/\theta$. Since $f'(\theta) \le e^{2|\theta|}$, the third term is
bounded by $\gamma^2\, \exp\{ \gamma\, [g_\infty + 2h_\infty]\} \, h_b
\, h_\infty \Vert z - z'\Vert^b$.

To complete the proof of the assertion it remains to add the bounds
and to recall the estimates of $f_\infty$, $g_\infty$, $h_\infty$ in
terms of $J_\infty$, $f_b$, $g_b$ and $h_b$.
\end{proof}

Let $f$ be a function in $C^1(\bb R_+ \times \mathbb{T})$. Denote by
$\widehat f: \bb R_+ \times \bb R \to \bb R$ the function which is
$1$-periodic in space and which coincides with $f$ on $\bb R_+\times
[0,1)$.

\begin{asser}
\label{a04}
For all $T>0$, $0 \le s\le T$, $x$, $y\in \bb R$,
\begin{equation}
\label{65}
\big|\, \widehat f(s,x) \,-\,
\widehat f(s,y)\,\big| \;\le\; M_0 \,
\big|\, x \,-\,  y\,\big|^\beta \;,
\end{equation}
where
\begin{equation*}
M_0 \;=\; M_0(f,T) \;:=\; \big\Vert \, (\partial_x f)\,
\big\Vert_{L^\infty([0,T]\times \bb T)} \;.
\end{equation*}
\end{asser}

\begin{proof}
As $\widehat f$ is $1$-periodic, we may replace $y$ by $y'$ such that
$|y'-x|\le 1$. Then, use that $f$ is uniformly Lipschitz on
$[0,T]\times \bb T$, and finally that $|y'-x|\le |y'-x|^\beta \le
|y-x|^\beta$ because $|y'-x|\le 1$, $\beta<1$.
\end{proof}

We conclude this section proving that the sequence of random fields
$R_\varepsilon$ introduced in \eqref{64} converges in probability to
$R=R_0$.

\begin{proposition}
\label{s13}
We have that $\bb P\big[\, \Vert R \Vert_{C^1([0,T_0]\times \bb T)}
\,<\, \infty\,\big] \,=\, 1$. Moreover, for every $\eta>0$,
\begin{equation*}
\lim_{\varepsilon \to 0}
\bb P\big[\, \Vert\, R \,-\, R_\varepsilon \, \Vert_{C^1([0,T_0]\times \bb T)}
\,>\, \eta \,\big] \,=\, 0\;.
\end{equation*}
\end{proposition}

\begin{proof}
By \eqref{78}, 
\begin{equation*}
R_\varepsilon \;=\; \partial_t \mf G_\varepsilon \,+\, (-\Delta)^{1/2}\, 
\mf G_\varepsilon \;=\;
\int_{-(T_0+1)}^{T_0} h(t-s)\, \xi_{\varepsilon}(s) \; ds 
\end{equation*}
where $h = [\,\partial_t \,+\, (-\Delta)^{1/2}\,] \, r$ and $r=
q_{T_0}-p$ is a smooth function.  A similar identity holds
with $R_\varepsilon$, $\xi_\varepsilon$ replaced by $R$, $\xi$,
respectively. 

By \cite[Proposition 1.3.3]{AdlTay07}, $\sup_{z\in [0,T_0]\times \bb
T} R(z)$ has finite expectation as well as $- \inf_{z\in [0,T_0]\times
\bb T} R(z)$. The same bound holds for $\partial_x R$, $\partial_t
R$. This proves that $\Vert R \Vert_{C^1([0,T_0]\times \bb T)}$ is
almost-surely finite.

As $r$ is smooth, the same theorem guarantees that there exists a
finite constant $C_0$ such that $\bb E \big[\, \sup_{z\in
  [0,T_0]\times \bb T} \{\, R_\varepsilon (z) \,-\, R (z) \,\} \,
\big] \,<\, C_0\, \varepsilon$ for all $0<\varepsilon \le 1$. The same
result holds for $R (z) \,-\, R_\varepsilon (z)$ and for the first
partial derivatives. It follows from these estimates that $\Vert\, R
\,-\, R_\varepsilon \, \Vert_{C^1([0,T_0]\times \bb T)}$ converges to
$0$ in probability as $\varepsilon \to 0$.
\end{proof}

\bibliographystyle{abbrv}
\bibliography{referencias}

\end{document}